\documentclass[11pt,a4paper,reqno]{amsart}
\usepackage[english]{babel}
\usepackage[applemac]{inputenc}
\usepackage[T1]{fontenc}
\usepackage{palatino}
\usepackage{amsmath}
\usepackage{amssymb}
\usepackage{amsthm}
\usepackage{amsfonts}
\usepackage{graphicx}
\usepackage{color}
\usepackage{mathtools}

\usepackage{caption}
\usepackage{vmargin}

\setmarginsrb{23mm}{20mm}{23mm}{20mm}{10mm}{10mm}{10mm}{10mm}

\usepackage{mathrsfs}

\newtheorem{theorem}{Theorem}[section]
\newtheorem{obs}{Observation}[section]

\newtheorem{cor}[theorem]{Corollary}
\newtheorem{prop}[theorem]{Proposition}

\newtheorem{lem}[theorem]{Lemma}

\theoremstyle{definition}
\newtheorem{defn}{Definition}[section]

\newtheorem{remark}[theorem]{Remark}
\newtheorem{ex}[theorem]{Example}

\def\vint{\mathop{\mathchoice%
          {\setbox0\hbox{$\displaystyle\intop$}\kern 0.22\wd0%
           \vcenter{\hrule width 0.6\wd0}\kern -0.82\wd0}%
          {\setbox0\hbox{$\textstyle\intop$}\kern 0.2\wd0%
           \vcenter{\hrule width 0.6\wd0}\kern -0.8\wd0}%
          {\setbox0\hbox{$\scriptstyle\intop$}\kern 0.2\wd0%
           \vcenter{\hrule width 0.6\wd0}\kern -0.8\wd0}%
          {\setbox0\hbox{$\scriptscriptstyle\intop$}\kern 0.2\wd0%
           \vcenter{\hrule width 0.6\wd0}\kern -0.8\wd0}}%
          \mathopen{}\int}

\renewcommand{\a}{\operatorname{a}}

\newcommand{\ep}{\epsilon}
\newcommand{\Om}{\Omega}
\newcommand{\om}{\omega}
\newcommand{\B}{\mathbb{B}^n}
\newcommand{\BB}{\mathbb{B}^2}

\newcommand{\N}{\mathbb{N}}
\newcommand{\Nm}{\mathcal{N}}

\newcommand{\R}{\mathbb{R}}
\newcommand{\Rn}{{\mathbb R}^n}

\newcommand{\ud}{\mathrm {d}}
\newcommand{\Hp}{\mathcal{H}^p}
\newcommand{\dist}{\operatorname{dist}}
\newcommand{\diam}{\operatorname{diam}}
\newcommand{\Sn}{\mathbb{S}^{n-1}}
\newcommand{\A}{\mathcal{A}}
\newcommand{\Leb}{\mathrm {d}{\mathcal L}^{n}}

\definecolor{blau}{rgb}{0.1,0.0,0.9}
\definecolor{violet}{rgb}{0.54, 0.17, 0.89}
\newcommand{\blue}{\color{blau}}

\newcommand{\kom}[1]{}
\renewcommand{\kom}[1]{{\bf \blue /#1/}}

\newcounter{komcounter}
\numberwithin{komcounter}{section}

\usepackage[colorlinks = true, linkcolor = red,  citecolor=red]{hyperref}

\begin{document}

\title[Hardy spaces and quasiregular mappings: averaged derivatives and the BMO case]{Hardy spaces and quasiregular mappings: averaged derivatives and the $\mathbb{BMO}$ case}
\author[T.\ Adamowicz]{Tomasz Adamowicz}
\address{The Institute of Mathematics, Polish Academy of Sciences \\ ul. \'Sniadeckich 8, 00-656 Warsaw, Poland}
\email{tadamowi@impan.pl}
\author[I. Caama\~{n}o]{Iv\'an Caama\~{n}o{\small$^1$}}
\address{The Institute of Mathematics, Polish Academy of Sciences \\ ul. \'Sniadeckich 8, 00-656 Warsaw, Poland}
\email{icaamanoaldemunde@impan.pl}

\thanks{{\small$^1$}
The research of I. C. is partially supported by grant PID2022-138758NB-I00 (Spain).}
\date{\today}

\begin{abstract}
 We study the Hardy spaces $\Hp$, $0<p<\infty$ of quasiregular mappings on the unit ball $\B$ in $\Rn$ 
under the appropriate growth and multiplicity conditions. Our focus is on the \emph{averaged derivatives} of maps and their Harnack and quantitative Harnack estimates. The averaged derivatives are employed to study the non-tangential limit functions and non-tangential maximal functions of quasiregular mappings and to characterize $\Hp$ in the case of finite multiplicity of $f$. Moreover, we study relations between quasiregular mappings, averaged derivatives, BMO spaces and Carleson measures on $\B$ and the role of the multiplicity of a map. We also apply our results to the second order elliptic PDEs and $\A$-harmonic equations. 

Our paper extends results by Astala and Koskela~\cite{ak} and Nolder~\cite{no} to the setting of quasiregular maps.
\end{abstract}
	
\maketitle

\tableofcontents

\section{Introduction}

Hardy spaces are one of the key tools in geometric and harmonic analysis and their studies have led to several profound  
results. Originally defined in the setting of complex analytic functions, see Hardy~\cite{ha}, have been only recently systematically studied for quasiconformal mappings in Euclidean spaces beyond the planar case, see Astala--Koskela~\cite{ak}. Recall that quasiconformal mappings are one of the most fruitful generalizations of planar conformal mappings and are studied in various settings including the Euclidean, the Riemannian and the sub-Riemannian settings, as well as, the metric measure spaces. Moreover, recently the studies of Hardy spaces has been initiated also for the quasiregular mappings, see~\cite{ag1, ag2}. On the contrary to quasiconformal mappings, the quasiregular ones need not be injective and, therefore, the interplay between the topology, geometry and analysis lies at the heart of their studies. In particular, one has to control the multiplicity of mappings and their growth in order to ensure the existence of the boundary limits, a problem solved for harmonic and complex analytic functions by the classical Fatou theorem. Nevertheless, a counterpart of the seminal result by Hardy--Littlewood~\cite{hl} characterizing the Hardy spaces in terms of the integrability of the nontangential maximal function and the non-tangential limit map can be proven in the setting of quasiregular maps as well, see~\cite{ag2} and also Theorem~\ref{thm11-ag2} in Section 2 below. 

In this work, we would like to shed a new light on another important tool in the geometric mapping theory, the so-called \emph{averaged derivative of a map}, and apply it to the studies of Hardy spaces. In order to give a wider perspective on our results, let us mention that they correspond to an important line of research in harmonic analysis where geometry of the domain intertwines with the studies of non-tangential maximal functions and the boundary behaviour of functions, see e.g.~\cite{gmt}, \cite{hmm}, \cite{hmmtz}, \cite{ht} to mention only few. In our work we consider the domain, classical for the studies of Hardy spaces, namely the unit ball. However, we hope that in the long run our results can be extended to other more general types of domains, such as Lipschitz or NTA domains.

In order to motivate the studies of the averaged derivative for quasiregular mappings, recall that the derivative of an analytic function in plane is analytic, and so a number of results in the $H^p$-theory for the derivative of an analytic function follow naturally. This is not the case already in the setting of quasiconformal mappings. In particular, the image of a circle under a quasiconformal mapping of the disk can fail to be rectifiable. Hence, in the quasiconformal studies a notion of the \emph{averaged derivative} is introduced in Astala--Gehring~\cite{asge2} for domains in $\R^n$ that generalizes $|Df|$ for conformal mappings $f$ and is subsequently studied in~\cite{ak} in the context of Hardy spaces, see~\cite[Formula (1.5)]{asge1} which in our notation reads:
\begin{equation}\label{AG-af}
a_f(x) := \exp \left(\frac{1}{n}\frac{1}{|2B_x|}\int_{2B_x} \log J_f \Leb\right),
\end{equation}
where $2B_x$ stands for an open hyperbolic ball centered at point $x$ with the radius equal to the distance of $x$ to the boundary of the underlying domain. As a comment aside let us remark, that the notion of the averaged derivative is quite robust and can be studied also beyond the Euclidean setting, for example in the first Heisenberg group, see~\cite[Definition 1.3]{afw}. However, this definition for the not necessarily injective quasiregular mappings is difficult to study even in the planar setting. Indeed, if $f$ is conformal in the plane, then $\a_f=|f'|$ due to the harmonicity of $\log |f'|$, but it is not the case for general analytic functions (which are $1$-quasiregular). Therefore, in the lack of injectivity we define the following counterpart of $a_f$:
\begin{equation}\label{eq:af}
	\a_{f,\lambda}(x):=\left(\frac{1}{|2B_x|}\int_{\lambda B_x}|Df|^n \Leb \right)^\frac{1}{n},\quad x\in \B,\,0<\lambda<2.
\end{equation}
In the special $\lambda=1$ this notion is introduced in Nolder~\cite{no}. Moreover, in order to avoid confusion with the weak Jacoby matrix $Df$, we changed the notation comparing to~\cite{no}, where the author uses $D_f$ instead. If $f$ is quasiconformal, then both notions of $a_f$ coincide, see~\cite[Lemma 2.7]{no}. Nevertheless, in general it trivially holds that
 $$
  \a_{f,\lambda}\leq D_f, \hbox{ for } 0<\lambda<1.
 $$

{\bf Our key general observation in this work is that by studying the family of averaged derivatives parameterized by $\lambda$ instead of a single averaged derivative for $\lambda=1$, as in~\cite{no}, we may obtain a good deal of results corresponding to those in Sections 2, 5 and 7 in~\cite{ak}.} 

\begin{quote}
In particular, we present various pointwise and quantitative Harnack estimates for $\a_{f, \lambda}$ and show how the interplay between different values of $\lambda$ can be utilized in the studies of Hardy spaces and quasiregular mappings, see Lemma~\ref{lem:harnack} and Corollary~\ref{cor:harnack-est} and, moreover, Theorem~\ref{thm:afintegrability}, Lemma~\ref{lem:2.5AK} and Corollary~\ref{cor:lem32}.
\end{quote}

Let us also add that, in fact, several of the results in Sections 3.2-3.3 hold under mild assumptions on maps, such as the local $L^n$-integrability of the Jacoby matrix and, therefore, can potentially be extended beyond the setting of quasiregular mappings, also for more general metric measure spaces.

 Furthermore, our results point to the role of multiplicity of a map, since stronger results can be proven if the multiplicity of a map in a ball is finite, see for instance Corollary~\ref{est-harnack-N}. As presented below, controlling the multiplicity function becomes important in our main results, see the condition~\eqref{cond-m} and the discussion in Section 2 below.  
Another crucial condition is the growth of the mapping, see Section 2 for the discussion on the Miniowitz class~\eqref{est-min} and Theorem~\ref{thm-min}.
%

The following is the first main result of our study.

\begin{theorem}\label{thm:characterizationsnolder}
	Let $f:\B\rightarrow \R^n\backslash\{ 0\}$ be a $K$-quasiregular mapping satisfying the multiplicity condition~\eqref{cond-m} and in the Miniowitz class~\eqref{est-min}. For any $p>0$ consider the following statements:
\begin{itemize}
\item[(1)] The non-tangential limit map $\tilde f\in L^p(\Sn)$.
\item[(2)]  It holds that $\int_{\B}\a_f^p(x)\, (1-|x|)^{p-1}\ud x<\infty$.
%
\item[(3)] The following non-tangential maximal function is $L^p(\Sn)$-integrable:
$$
\Nm\big(\a_f (x)\, (1-|x|)\big) \in L^p(\Sn).
$$
 \end{itemize}
Then, it holds that 
\[
  (1) \Rightarrow (2) \Rightarrow (3).
\]
Moreover, if the multiplicity of $f$ in ball $\B$ is finite $N(f,\B)<\infty$, then it holds that
\[
(1) \Leftrightarrow (2) \Leftrightarrow (3).
\]
\end{theorem}

In the special case of quasiconformal mappings we retrieve~\cite[Theorem 5.1]{ak}, since then $N(f, \B)=1$. Moreover, recall that for quasiconformal mappings $\a_f=D_f$, the Nolder operator, and is equivalent to the Astala--Gehring averaged derivative, see Lemma 2.7 in~\cite{no}. Thus, assertions (2) and (3) in Theorem~\ref{thm:characterizationsnolder} reduce for the quasiconformal mappings to the corresponding assertions (2) and (3) in Theorem 5.1 in~\cite{ak}.

\begin{cor}\label{cor-thm-nolder} Let $f:\B\rightarrow \R^n\backslash\{ 0\}$ be a $K$-quasiregular mapping satisfying \eqref{cond-m} and \eqref{est-min}.  For any $p>0$ consider the following statements:
\begin{itemize}
\item[(1)] $f\in \mathcal{H}^p$.
\item[(2)] $\tilde f\in L^p(\Sn)$.
\item[(3)] $\int_{\B}|f(x)|^{p-1} |Df(x)|\,\ud x <\infty$.
\item[(4)] $\int_{\B}\a_f^p(x)\, (1-|x|)^{p-1}\ud x<\infty$.
\item[(5)] $\Nm\big(\a_f (x)\, (1-|x|)\big) \in L^p(\Sn)$.
%
\end{itemize}
It holds that 
\[
(1) \Leftrightarrow (2) \Leftrightarrow (3)\,\,\hbox {and }\, (2) \Rightarrow (4) \Rightarrow (5).
\]
Moreover, if the multiplicity of $f$ in ball $\B$ is finite $N(f,\B)<\infty$, then it holds that
\[
(1) \Leftrightarrow (2) \Leftrightarrow (3) \Leftrightarrow (4) \Leftrightarrow (5).
\]
\end{cor}

We prove both Theorem~\ref{thm:characterizationsnolder} and Corollary~\ref{cor-thm-nolder} at the end of Section~\ref{sect35}. The proof of Theorem~\ref{thm:characterizationsnolder} involves the full range of tools introduced and proven in Sections 3.2-2.5:
\begin{itemize} 
\item the Harnack estimates for the averaged derivatives $\a_{f,\lambda}$, see Section 3.2;
\item the reverse quantitative Harnack (Theorem~\ref{thm:afintegrability} in Section 3.3) and the quantitative Harnack estimates (Lemma~\ref{lem:2.5AK} and Corollary~\ref{cor:lem32} in Section 3.4)
\item Proposition~\ref{prop:Lp-norm-bound} - an analog of the Lusin area integral estimate for quasiregular mappings, see Section~\ref{sect35}. 
\end{itemize}

One of the interesting consequences of Theorem~\ref{thm:characterizationsnolder} is a counterpart of the Riesz conjugate theorem for quasiregular mappings presented in Corollary~\ref{cor-Riesz}, see Section 3.6.

The second main result of the manuscript is Theorem~\ref{thm72}, see Theorem 7.2 in~\cite{ak} for the quasiconformal case.

\begin{theorem}\label{thm72}
  Let $n \geq 2$ and $f:\B\to \Rn \setminus \{0\}$ be a $K$-quasiregular mapping in class~\eqref{est-min} satisfying the multiplicity condition~\eqref{cond-m} with some $0\leq a<n-1$. Let us consider the following assertions:
  \begin{itemize}
   \item[(1)] The non-tangential limit map $\tilde{f}\in BMO(\Sn)$.
   \item[(2)] $\sup_{T_{x_0}} \int_{\Sn} |\tilde{f}\circ T_{x_0}- \tilde{f}(T_{x_0}(0))|\,\ud\sigma<\infty$, where the supremum is taken over all M\"obius automorphisms $T_{x_0}$ of $\B$, for $x_0\in \B$.
   \item[(3)] The following measure 
   \[
   \ud \mu= |Df(x)|^{n}(1-|x|)^{n-1+a}\,\ud x
   \] is the Carleson measure on $\B$.
   \item[(4)] It holds that 
   \[
   a_f(x)(1-|x|)^{1+\frac{a}{n}}\leq C\quad \hbox{ for all }x\in \B,
   \]
   where $C$ depends on $n, K$ and $a, m$. 
   \item[(5)] Each coordinate function $f^j$ of $f=(f^1,\ldots, f^n)$ is in BMO($\B$).
   \item[(6)] $f\in BMO(\B)$.
  \end{itemize}
  Then it holds that
  \[
   (1) \Rightarrow (2) \Rightarrow (3) \Rightarrow (4).
  \]
  Moreover, if the multiplicity of $f$ in ball $\B$ is finite, i.e., $N(f,\B)<\infty$, then $a=0$ and also $(4)\Rightarrow (1)$ and $(4)\Rightarrow (5)$, and so it holds that the following conditions are equivalent:
\[
(1) \Leftrightarrow (2) \Leftrightarrow (3) \Leftrightarrow (4) \Leftrightarrow (5) \Leftrightarrow (6).
\]
  \end{theorem}
 In order to prove Theorem~\ref{thm72} we will need a number of auxiliary observations, which might be of the independent interest to the reader:
\begin{itemize}
\item Lemma~\ref{lem75} which provides a family of the Carleson measures on the unit ball $\B$ for quasiregular mappings satisfying the BMO condition. See also Lemma~\ref{lem94} for the global counterpart of Lemma~\ref{lem75} for the maps in the $n$-th Hardy space.
\item Lemma~\ref{lem-gar33} - a characterization of the Carleson measures on the unit ball $\B$ in terms of its M\"obius automorphisms $T_x$, cf. the corresponding~\cite[Lemma 3.3, Ch VI]{g} for the $2$-dimensional case and~\cite[Theorem 1.2]{adgr} for a counterpart of this lemma in the setting of the Heisenberg group $\mathbb{H}_1$.
\end{itemize}

In order to complement our discussion we obtain a counterpart of Theorem 7.1 in~\cite{ak} and show that the logarithm of the extension $|\tilde{f}|$ belongs to $BMO(\Sn)$, see Theorem~\ref{thm:log-bmo} in Section 4. Moreover, we apply Theorems~\ref{thm:characterizationsnolder} and~\ref{thm72} in Section 6 to show how the estimates for quasiregular maps carry over to estimates for the elliptic PDEs in the plane and for $\A$-harmonic equations. 

For the sake of completeness of the introduction, let us mention that in Section 2 we recall the necessary basic definitions of the quasiregular mappings and the Hardy spaces and the related notions. Moreover, we explain the importance of the multiplicity condition and the Miniowitz class, including examples of quasiregular maps in the class~\eqref{est-min}.

Let us also add that our additional motivation for this work is to extend the corresponding discussion in~\cite{ak} and fills in a number of details not presented in that seminal paper. Moreover, we hope to attract and draw attention of a broader audience of researchers to the studies of the Hardy spaces of mappings.

\section{Preliminaries}
For any $x\in \R^n$ and $r>0$ we denote the ball centered at $x$ with radius $r$ by $B(x,r)$ and define $\lambda B(x,r):=B(x,\lambda r)$ for any $\lambda >0$. Denote $\B=B(0,1), \Sn=\partial \B$ and for $x\in \B$ we define the (Euclidean) hyperbolic ball as follows:
$$
B_x:=B\Big(x, \frac{1}{2}(1-|x|)\Big).
$$
Next, we define the \emph{non-tangential region in $\B$ with parameter $\alpha$ centered at $\omega\in \Sn$} as follows:
\begin{equation}\label{def-cone}
\Gamma_{\alpha}(\omega):=\{x \in \B:\; |x-\om| < (1+\alpha)(1-|x|)\}.
\end{equation}
When $\alpha$ is fixed, we write for simplicity $\Gamma:=\Gamma_{\alpha}$. 

For a point $x\in \B$ we define its \emph{shadow} $S(x)\subset \Sn$ as follows
\[
S(x):=B(x, (1+\alpha)(1-|x|)) \cap \Sn.
\]
By the definition, $\om\in S(x) \Leftrightarrow x\in \Gamma_\alpha(\om)$, and so $S(x)$ denotes the maximal set of points on the unit sphere such that a non-tangential region with apex in $S(x)$ contains point $x$.
\smallskip

\noindent {\bf Quasiregular mappings.}  Let $\Om\subset \Rn$ be a domain. A continuous mapping $f:\Om \rightarrow \mathbb{R}^n$  is called $K$-\emph{quasiregular} for $K\geq 1$, if $f$ belongs to the Sobolev space $W_{loc}^{1,n}(\Om, \Rn)$ and the distortion inequality
\begin{equation*}
  |Df(x)|^n\leq  K J_f (x)
\end{equation*}
  holds for almost every $x\in \Om$. Here $|Df(x)|$ denotes the operator norm of the formal derivative of $f$ at $x\in \Om$ and $J_f(x)$ stands for the Jacobian determinant of $Df$ at $x$.

 If in addition we require $f$ to be a homeomorphism, then we say that $f$ is $K$-\textit{quasi\-con\-formal}. For comprehensive introductions to the topic and further references we refer, for instance, to~\cite{Ric}, Chapter 14 in~\cite{hkm},\cite{resh} for results in $n\geq 2$, and~\cite{aim} for $n=2$. Let us comment that the quasiregular mappings can be defined in various other equivalent ways, e.g. via the modulus of curve families, see e.g.~\cite{va, vuo2}. 
\smallskip
 
\noindent {\bf Hardy spaces of mappings.} We say that a mapping $f:\B \to \Rn$ belongs to \emph{the Hardy space $\mathcal{H}^p$} for a $0<p<\infty $, if
\[
\sup_{0<r<1} \left(\vint_{\Sn} |f (r\om)|^p ~\ud\om\right)^{\frac1p} := \|f\|_{\mathcal{H}^p}< \infty.
\]
The classical Hardy spaces of holomorphic functions on unit disc are due to Hardy and Littlewood \cite{hl}. Their importance in the complex analysis follows, for instance, from the relations to harmonic analysis, the studies of the boundary behaviour of functions, Carleson measures and in the Nevanlinna theory, see e.g.~\cite{D, g}. 
 
The systematic studies of the Hardy spaces of quasiconformal maps in the Euclidean spaces for $n\geq 2$ are presented in~\cite{ak}, see also~\cite{bk, bkl}, \cite{z}, \cite{nol}; see~\cite{af} for the setting of first Heisenberg group $\mathbb{H}_1$. The lack of injectivity for quasiregular maps leads to the necessity of some additional assumptions as discussed in \cite{ag1, ag2}. We recall them now.
\smallskip

\noindent{\bf Multiplicity condition~\eqref{cond-m}.} One of the fundamental problems to handle in our studies is the existence of non-tangential boundary values for quasiregular functions. Even the special case of bounded quasi\-re\-gu\-lar mappings, which in fact belong to any Hardy space, already illustrates the difficulty of the problem and is not yet fully solved, see~\cite{ra08}. Moreover, already in the planar setting, the existence of the radial limits requires additional assumptions, see~\cite{ma, ag1} and the discussion in the introduction to~\cite{ag2}. On the other hand, bounded quasiconformal mappings have boundary limits almost everywhere. Therefore, the control on the multiplicity of the quasiregular mapping turns out to be crucial in several of the known results regarding the existence of boundary values.

Recall that by following the discussion in~\cite{hkm} (see pg. 254), for a subset $E\subset \B$ we denote by $N(y, f, E)$ \emph{the crude multiplicity} of mapping $f$ and define it as follows: $N(y, f, E):=\#\{x\in E: f(x)=y\}$. Similarly, we define 
\emph{the multiplicity function} of $f$:
\[
N(f, E)=\sup_{y\in \R^n} N(y, f, E).
\]
 In~\cite[Theorem 4.1]{kmv}, see also condition (3) in~\cite{Ak}, the authors impose on a quasiregular mapping $f:\B\to \Rn$  the following multiplicity growth condition:
\begin{equation}\label{precond-m}\tag{${\rm M}_0$}
N(f, B(0,r))\leq \frac{C}{(1-r)^a}\quad \hbox{for all } 0<r<1,
\end{equation}
for some constant $C>0$ and an exponent $a\in [0, n-1)$. For bounded quasiregular mappings such growth restriction allows to show the existence of radial limits almost everywhere on the unit sphere, see~\cite{mr}; see also ~\cite[Ch. IV]{vuo2} and \cite[Ch. VII]{Ric} for further results and references.

In order to study the nontangential maximal function for quasiregular mappings we introduce in~\cite{ag2}  a stronger condition on the multiplicity of the function. Namely, we require estimate (\ref{precond-m}) to hold not only for balls  centered at zero, but for any hyperbolic ball in $\B$. To this end, let $T_x$ denote the self-conformal map in $B$, such that $T_x(x)=0$, see~\eqref{def:Tx}, and let $B_\rho (x,r)$ stand for the hyperbolic ball defined by $B_\rho (x,r):= T_x^{-1}(B(0,r))$.   We  refer the reader to  e.g.~\cite[Ch. I.2]{vuo2} for more information on the properties of hyperbolic balls and their relations with Euclidean balls.  

\begin{defn}[Definition 1.1 in~\cite{ag2}]\label{def-multi}
	We say that a quasiregular mapping $f:\B\to \Rn$ satisfies \emph{the multiplicity condition (M)}, if there exist constant $C>0$ and an exponent $a\in [0, n-1)$, such  that
	\begin{equation}\label{cond-m}\tag{M}
	\sup_{x\in \B} N(f, B_\rho(x, r))\leq \frac{C}{(1-r)^a}\quad \hbox{for all } 0<r<1.
	\end{equation}
\end{defn}
In particular, for $x=0$ we retrieve condition~\eqref{precond-m}. Notice that since the M\"obius transformations of a unit ball $\B$ are the isometries of $\B$ considered with respect to the hyperbolic metric, it holds that $T_x^{-1}(B(0,r))$ equals to an Euclidean ball whose center and radius can be computed directly, see (2.22)-(2.24) in~\cite[Ch. I.2]{vuo2}. In a consequence, condition~\eqref{cond-m} reads:
\begin{align}
\sup_{x\in \B} N(f\circ T_x^{-1}, B(0, r)) &=\sup_{x\in \B} N(f, B_\rho(x, r)) \nonumber \\
&=\sup_{x\in \B} N\left(f, B\left(\frac{x(1-r^2)}{1-|x|^2r^2}, \frac{(1-|x|^2)r}{1-|x|^2r^2}\right)\right)\leq \frac{C}{(1-r)^a}. \label{cond-m-tech}
\end{align}
Let us analyze the behaviour of $R$, the radius of the (Euclidean) ball in~\eqref{cond-m-tech}, as the function of $|x|$ and $r$:
\medskip

 $\bullet$\,\, for any fixed $x\in \B$ we have that $R\to 1$, as $r\to 1$ and $R\to 0$, as $r\to 0$, and thus the growth of $N$ is similar in nature to the condition~\eqref{precond-m}; \\
\indent $\bullet$\,\, for a fixed $r$, it holds that $R\to r$, when $|x|\to 0$ and the growth of $N$ is again similar to~\eqref{precond-m}.
\medskip

However, the nature of the condition~\eqref{cond-m} differs from the condition~\eqref{precond-m} for the (Euclidean) hyperbolic balls, as the following observation shows. Recall that Lemma 2.2 in~\cite{ag2} asserts that if $0<\lambda<2(\sqrt{2}-1)\approx 0,83$, then
\[
 \lambda B_x \subset T_x^{-1} B\left(0, \left(\lambda/2\right)^2+\lambda\right):=B_\rho\left(x, \left(\lambda/2\right)^2+\lambda\right).
\]  
Therefore, by the~\eqref{cond-m} condition we have for all $x\in \B$ and $0<\lambda<2(\sqrt 2 -1)$ that
\begin{equation}\label{eq:Mhyperboliclambda}
1\leq N(f, \lambda B_x)\leq \frac{C}{(1-(\lambda/2)^2-\lambda)^{n-1}}, 
\end{equation}
due to $a<n-1$. On the other hand, if $\lambda\in (0,2)$ then for each $x\in\B$ we can cover $\lambda B_x$ by balls $\{\frac{1}{2}B_{x_i}\}_{i=1}^N$, where $N$ only depends on $n$ and $\lambda$, see the covering Lemma~\ref{lem:cover-aux}. Then, by the subadditivity of the multiplicity for finite coverings and \eqref{eq:Mhyperboliclambda} applied to balls $\frac{1}{2}B_{x_i}$, we get
\begin{equation}\label{eq:Mhyperbolic}
	1\leq N(f, \lambda B_x)\leq C(\lambda ,n).
\end{equation}
\begin{remark}
If $\lambda$ approaches $2$, then in the above estimate the constant $C(n,\lambda)$ grows unbounded, since then the cardinality $N$ of the $\lambda B_x$ covering grows. Indeed, the closer to $\partial \B$ the ball $\lambda B_x$ is, the smaller the covering balls become and, thus, $N$ grows large.
\end{remark}

\smallskip

\noindent{\bf Miniowitz class~\eqref{est-min}.} Let us now recall the following estimate by Miniowitz and the related class of mappings, see~\cite{ag2} for more details. Let $E \subset \mathbb{R}^n$ be such that
\begin{equation}\label{cond:omit-set}
E\cap S^{n-1}(r)\not=\emptyset,\quad \hbox{for all } r\geq 0.
\end{equation}
Examples of sets $E$ encompass an unbounded curve starting at the origin and the nonnegative part of the $x_i$-axis in $\Rn$ for any $i=1,\ldots,n$.
\begin{theorem}[Theorem 3 in~\cite{M1}]\label{thm-min}
Let $f:\B\to \Rn\setminus E$ be a $K$-quasiregular mapping for set $E$ as in~\eqref{cond:omit-set}. Then

\begin{equation}\label{est-min}\tag{Min}
\frac{1}{C}\left(\frac{1-r}{1+r}\right)^{m} \leq \frac{|f(y)|}{|f(0)|} \leq C\left(\frac{1+r}{1-r}\right)^{m},\quad r=|y|,
\end{equation}
 where $m:=2^{n-1}K_I$ and $C:=2^{8m}$ are constants depending only on the inner distortion $K_I$ and $n$.
\end{theorem}

Moreover, Theorem~\ref{thm-min} has counterparts for quasiregular mappings with bounded multiplicity omitting the origin and spherically mean $1$-valent quasiregular mappings, see~\cite{M1}. Observe that all of these classes of quasiregular functions are invariant by conformal self-maps $T_x$ of the unit ball. This observation motivates the following definition of the class of mappings.

\begin{defn}[Definition 1.2 in~\cite{ag2}] 
	We say that a $K$-quasiregular mapping  $f:\B \to \Rn\setminus\{0\} $  is in the \emph{Miniowitz class}, denoted by ${\rm (Min)}$, if  for any $x\in \B$, the quasiregular map $g_x:=f\circ T_x^{-1}$ satisfies the growth estimate~\eqref{est-min} with some positive constants $C$ and $m$ depending only on $n$ and $K$.
\end{defn}
We also recall if a quasiregular mapping $f$ satisfies condition~\eqref{cond-m} and is in the Miniowitz class~\eqref{est-min}, then $g_x:=f\circ T_x^{-1}$ satisfies the multiplicity condition~\eqref{precond-m}, i.e. the multiplicity condition~\eqref{cond-m} only on hyperbolic balls centered at zero, as well as $g_x$ satisfies the growth estimate~\eqref{est-min}. 

It is worthy noticing that the class of quasiregular mappings in Miniowitz class is rich and enclosses the following examples:
\begin{itemize}
\item[(1)] $f:\B\to \Rn\setminus \{0\}$ with bounded multiplicity,
\item[(2)] $f:\B\to \Rn\setminus E$, for a set $E$ as in~\eqref{cond:omit-set},
\item[(3)] $f:\B\to \Rn$ and is spherically mean $1$-valent mapping, see~\cite{M1}.
\end{itemize}
\smallskip

\noindent{\bf Growth condition.} Following the discussion in~\cite{ag2} we recall the following \emph{growth condition (G)} on a quasiregular mapping $f:\B\to \Rn$ with a constant $C>0$ and exponent $0<\beta<\infty$
\begin{equation}\label{cond-g}\tag{G}
 |f(x)|\leq \frac{C}{(1-|x|)^{\beta}}.
\end{equation}

Observation 3.1 in~\cite{ag2} shows that a quasiregular map in the Hardy space $\Hp$ satisfies the growth condition~\eqref{cond-g} with $\beta=\frac{n-1}{p}$. In the special planar case of $n=2$ and analytic $f$ we retrieve the growth exponent $\frac1p$ as in formula (3.9) in~\cite[Chapter I.3]{g}.
\smallskip

\noindent{\bf Radial limits.} We define the \emph{radial limit function} $\tilde{f}$ of a mapping $f$ as follows. Let $f:\B\to \Rn$, then
\[
 \tilde{f}(\om):=\lim_{r\to 1} f(r\om) \quad\hbox{for }\om\in \Sn,
\]
whenever this limit exists. Theorem 4.1 in~\cite{kmv} asserts that a \emph{$K$-quasiregular map satisfying the growth condition~\eqref{cond-g} with some $\beta>0$ and the multiplicity condition~\eqref{precond-m} with $0\leq a< n-1$ has non-tangential limits at all points in $\Sn$ except possibly of a set $E$ of the Hausdorff dimension $\dim_{H}(E)<\frac{na}{1+a}$.} Notice further, that since existence of the non-tangential limit is independent of choice of a curve nontangentially approaching the given boundary point,  we can reduce our investigations to the radial limits only and justify definition of $\tilde{f}$.
\smallskip

\noindent{\bf Non-tangential maximal functions.} If $f:\B\to \Rn$ is any map, then the \emph{non-tangential maximal function of f} is defined as follows:
\[
\Nm_{\alpha} f (\om):= \sup_{x\in \Gamma_{\alpha} (\om)} |f(x)|,\quad \om\in \Sn,
\]
where $\Gamma_\alpha(\om)$ stands for a non-tangential region in $\B$ centered at $\om$ with the aperture $\alpha$, see~\eqref{def-cone}. If $\alpha$ is fixed we denote $\Nm f (\om):=\Nm_{\alpha} f (\om)$.

We are now in a position to recall Theorem 1.1 in~\cite{ag2} which characterizes quasiregular mappings in Hardy spaces.

\begin{theorem}\label{thm11-ag2}
	Let $n \geq 2$ and $f:\B\to \Rn \setminus \{0\}$ be a $K$-quasiregular mapping in class~\eqref{est-min} satisfying the multiplicity condition~\eqref{cond-m} with $0\leq a<n-1$. Then the following conditions are equivalent for every $p>0$:
\[
  f\in \mathcal{H}^p\, \Leftrightarrow\, \tilde{f}\in L^p(\mathbb{S}^{n-1})\,\Leftrightarrow\, \Nm f\in L^p(\mathbb{S}^{n-1})
\]	
and the norms are equivalent with the equivalence constants depending only on $n, p$ and $a, K, m$.
\end{theorem}
\smallskip

{\bf BMO spaces.} The space of functions with bounded mean oscillations (BMO) plays an important role in the harmonic analysis, for example due to their duality to the real Hardy space $H^1$, relations to the Poisson integral and the Carleson measures, see e.g.~\cite[Chapter VI]{g}, also in the context of the John--Nirenberg lemma, see~\cite[Chapter 18]{hkm}. In the mapping theory the BMO spaces appear in relation to quasiconformal mappings, as the logarithm of the Jacobian determinant of such a map belongs to BMO, also quasiconformal mappings are the so-called BMO maps, see~\cite{rei74, rr} and also~\cite{go, kkms, j} for the definition and properties of the BMO maps.

In our work the BMO spaces will appear in Sections 4 and 5, especially in the context of the boundary BMO spaces, whose definition we now recall, but will not directly appeal to in what follows.

For a given $\om\in\Sn$ and $r>0$ we set $\Delta(\om,r):=B(\om,r)\cap \Sn$, where $B(\om, r)$ is a ball in $\Rn$.  We say that a function $v:\Sn \to \R$ belongs to the space $BMO(\Sn)$ equipped with the surface measure $\sigma$ in $\Sn$, if 
\begin{equation*}
\sup_{\Delta(x,r)}\frac{1}{\sigma(\Delta(x,r))}\int_{\Delta(x,r)}|v(y)-v_{\Delta(x,r)}|\ud \sigma(y)<\infty,
\end{equation*}
where $v_{\Delta(x,r)}$ denotes the mean-value of a function $v:\Sn \to\R$ on $\Delta(x,r)$, i.e. $v_{\Delta(x,r)}:=\vint_{\Delta(x,r)} v(y)\,\ud \sigma(y)$.
\smallskip

\noindent {\bf Conformal automorphisms of the Euclidean unit ball.} In several results below we use the self-conformal maps of a unit ball $\B\subset \R^n$, denoted by $T_x$, such that $T_x(x)=0$ for a fixed $x\in \B$, $x\not=0$. Properties of such maps have been substantially used in~\cite{ak, ag2, z} and are studied in great depth, e.g. in Chapter II in~\cite{ahl}. For the readers convenience we recall their formula: 
\begin{equation}\label{def:Tx}
T_x(y):=\frac{(1-|x|^2)(y-x)-|y-x|^2x}{|x|^2\left|y-\frac{x}{|x|^2}\right|^2},\quad y\in B,\quad x\not=0.
\end{equation}	
\smallskip

\noindent {\bf Covering lemma}. The following observation will be used in some of our key results in Sections 3 and 5 and is a refinement of Proposition 4.1.15 in~\cite{hkst}. 

\begin{lem}[Whitney covering]\label{lem:WhitneyCovering}
 Let $(X,d)$ be a doubling metric space and $\Omega\subset X$ open. Fix $0<\eta <1$, then there exists a covering of $\Omega$ by balls $B_i=B(x_i, \eta d(x_i,X\setminus\Omega ))$, $i\in\N$, such that for each $1\leq \tau <\frac{1}{\eta}$ there exists $C_\tau >0$ so
  $$
  \sum_{i\in\N} \chi_{\tau B_i}\leq C_\tau.
  $$
\end{lem}

We present the proof of the lemma in Appendix~\ref{appA}.

\smallskip

\noindent {\bf Notation convention.} In what follows $A \lesssim B$ means that there is a constant $c>0$ such that $A\leq cB$ and $A\lesssim_kB$ if the constant $c$ depends on $k$. Similarly, $A\approx B$ stands for $A\lesssim B \lesssim A$ and the notation $A\approx_k B$ is defined in an analogous way.

\section{Characterization of Hardy spaces via the averaged derivative}

The main goal of this section is to discuss and prove Theorem~\ref{thm:characterizationsnolder}, a characterization of the $L^p$-integrability of \emph{the non-tangential limit map} of a quasiregular mapping in terms of its \emph{averaged derivative}; see also Corollary~\ref{cor-thm-nolder} for the related characterization of the Hardy spaces. Therefore, we generalize the corresponding result for quasiconformal mappings, see Theorem 5.1 in~\cite{ak}. In order to present a slightly wider perspective on main results of this section, let us recall that they belong to the line of research in harmonic analysis studying relations between the geometry of the domain, the $L^p$-integrability of boundary data for the harmonic Dirichlet problem and the $N\approx S$ estimates for the square function/nontangential maximal function. See, for instance, Section 1 in~\cite{hmm} and Theorem 1.1 in~\cite{az} for the case of Lipschitz domains and domains with uniformly rectifiable boundary, respectively.

In Section 3.1 we discuss some difficulties related to showing the Harnack estimates for the operators $\a_{f,\lambda}$ and the role of the finite multiplicity of mappings in proving such estimates. Then, in Section 3.2 we discuss various pointwise Harnack inequalities for the averaged operators $\a_{f,\lambda}$ without the multiplicity conditions and even without the quasiregularity assumption. Sections 3.3 and 3.4 are devoted to the quantitative (integral) versions of the Harnack estimates for $\a_{f,\lambda}$, see Theorem~\ref{thm:afintegrability}, Lemma ~\ref{lem:2.5AK} and Corollary~\ref{cor:lem32}. Moreover, we also discuss partial Koebe distortion theorems in Lemma~\ref{lem:2.3AK} and in Remark~\ref{rem-lem23-v2}. Finally, Section 3.5 contains the proof of Theorem~\ref{thm:characterizationsnolder}.

\subsection{Harnack inequalities: motivations and the finite multiplicity case} \hfill
\smallskip
 
\noindent Let us recall that one of the basic tools employed throughout the work~\cite{ak} are the following results:
\smallskip

\noindent (1) Lemma 2.1, giving the equivalence between the diameter of the image of a hyperbolic ball under a quasiconformal map $f$ and the distance to the boundary of the target space, and \\
(2) Lemma 2.3, the Koebe distortion theorem for $a_f$ (as defined in~\eqref{AG-af}), which plays an important role in understanding the behavior of $a_f$ and controlling its values near the boundary (see also~\cite[Theorem 1.8]{asge1}). 
\smallskip

An important consequence of these results is the Harnack inequality for $a_f$ when $f$ is quasiconformal, i.e.
\[
a_f(x)\simeq_{n, K} a_f(y),\,\hbox{ for each } x\in \B \hbox{ and } y\in B_x.
\]
For quasiregular maps we get a weaker Harnack estimate, see Lemma~\ref{lem:harnack} below. In fact, the result is obtained under mild assumptions on $f$ in order to emphasize the generality of $\a_f$ and its potential broad applications. However, as we will discuss in Corollary~\ref{est-harnack-N}, a stronger Harnack estimate is available provided that a quasiregular map $f$ has finite multiplicity $N(f,\B)<\infty$.  

Let us briefly argue that the Harnack inequality for more general maps with $N(f,\B)=\infty$ appears to be a more delicate problem. Recall, that the proof of the Koebe distortion theorem in~\cite{asge1} relies on the BMO estimate for $\log J_f$, where $J_f$ denotes the Jacobian of the quasiconformal map, which in turn is related to the Muckenhoupt condition for $J_f$ and, therefore, to the Harnack estimate for the average derivative $a_f$. As observed in Section 2.4(c) in~\cite{bkr} and in Remark 3.6 in~\cite{hk}, a Harnack estimate for quasiregular maps (in particular, in the lack of the injectivity), follows from two facts:

(1) the doubling property of the weighted measure $\ud\mu=J_f\ud\mathcal{L}^n$, and \\
\indent(2) the finite multiplicity assumption on $f$, i.e. $N(f,\B)<\infty$.

Recall that, for instance by~15.5 and 15.7 in~\cite{hkm}, the Muckenhoupt condition $J_f\in A_{\infty}(\B)$ implies the doubling property for $\mu$ on $\B$ (but only for such balls $B\subset \B$ with $B\subset 2B\subset \B$).  Furthermore, by the discussion in~\cite[Remark 3.6(b)]{hk}, the same condition $J_f\in A_\infty(\B)$ also implies the uniform boundedness of the local index of $f$, but the map $g(z)=e^z$ in the plane is a counterexample that the opposite implication need not hold and, moreover, observe that $N(g,\R^2)=\infty$. On the other hand, direct computations imply the Harnack estimate for $\a_g$ in $\B$.

Toward a Koebe estimate in the setting of quasiregular mappings, it turns out that a careful scrutiny of the proof of Theorem 1.5 in~\cite{hk} leads to the following proposition, a counterpart of the formula (3.7) in~\cite{hk}.

\begin{obs} If $f:\B\to \Rn$ is a $K$-quasiregular map with bounded multiplicity $N(f,\B)<\infty$, then the Koebe estimate holds for all $x\in \B$ and $0<\lambda\leq 1$:
\begin{equation}\label{est:qr-koebe}
\lambda^{c(n)} \frac{\diam f(B_x)}{1-|x|} \lesssim \a_{f, \lambda}(x) \lesssim \frac{\diam f(B_x)}{1-|x|}.
\end{equation}
The comparison constants in both inequalities depend on $n, K, N(f,\B)$ and, moreover, $c(n)>1$.
\end{obs}

\begin{proof}
In order to show the claim, we first note that as on pages 259-260 in~\cite{hk} one can show that for $1\leq t<2$ it holds that 
\begin{equation*}
|f(t B_x)|\leq C(n, K, t) |f(B_x)|,
\end{equation*}
and so 
\begin{equation}\label{Jacobian-dbl2}
\int_{tB_x} J_f \leq C(n, K, t) \int_{B_x} J_f,
\end{equation}
due to the bounded multiplicity and the quasiregular change of variables formula. Next, we fix $1<t<2$ and by the Sobolev embedding theorem and the reverse H\"older inequality we get for $y\in B_x$ the following estimate
\begin{align*}
 |f(y)-f(x)|\lesssim_{n, K} \diam B_x \left(\vint_{tB_x} J_f \right)^{\frac1n} &\lesssim_{n,K,N(f,\B), t, \lambda} \frac{2}{\lambda} \diam B_x \left(\vint_{\lambda B_x} J_f \right)^{\frac1n} \\
 &\lesssim_{\lambda, n, K, N(f,\B), t} \frac{1}{\lambda} (1-|x|) \a_{f, \lambda}(x),
\end{align*} 
and so, the left hand-side in estimate~\eqref{est:qr-koebe} follows. In fact, the dependence on $\lambda$ can be found more explicitly, again by computing the appropriate moduli of curve families at the end of the proof of~\cite[Theorem 1.5]{hk}. Namely, it turns out that in~\eqref{est:qr-koebe} we have $\lambda^{c}$ for the constant $c=c(n)>1$.
 
The right-hand side estimate follows immediately from the definition of $\a_{f,\lambda}$ and the quasiregular change of variables formula as follows:
\[
 \a_{f,\lambda}^n(x)\leq \frac{K}{|2B_x|} \int_{f(B_x)} N(f, B_x) \leq \frac{K N(f, B_x)}{|2B_x|}\,\big(\diam f(B_x)\big)^n.
\]
\end{proof}
An important consequence of the observation is the Harnack estimate for the averaged derivative, whose proof follows from the doubling property~\eqref{Jacobian-dbl2}.
 
\begin{cor}\label{est-harnack-N}
If $f:\B\to \Rn$ is a $K$-quasiregular map with bounded multiplicity $N(f,\B)<\infty$, then for all $x\in \B$, $0<\lambda\leq 1$ and $y\in \frac{\lambda}{3} B_x$, it holds that:
\[
   \a_{f, \lambda}(x) \approx \a_{f, \lambda}(y).
\]
The comparison constant depends on $n, K, N(f,\B)$ and $\lambda$. 
\end{cor}

\begin{proof}
 If $x\in \B$ and $y\in \frac{\lambda}{3} B_x$, then the following estimations give that $\lambda B_x\subset \lambda_2 B_y$ for $\lambda_2:=\frac{8\lambda}{6-\lambda}$. Indeed, we have for $z\in \lambda B_x$ that 
 \[
  |z-y|\leq |z-x|+|x-y|\leq \frac{\lambda}{2}(1-|x|)+\frac{\lambda}{6}(1-|x|)\leq \frac{8\lambda}{6-\lambda}\frac12 (1-|y|),
 \]
 as
  $$
   1-|x|\leq |x-y|+1-|y|< \frac{\lambda}{6}(1-|x|)+1-|y|\Rightarrow 1-|x|\leq \frac{6}{6-\lambda }(1-|y|).
  $$
 Then, the doubling property of the Jacobian~\eqref{Jacobian-dbl2} implies the following 
 \begin{align*}
  \a_{f,\lambda}^n(x) \leq \frac{K}{|2B_x|} \int_{\lambda B_x} J_f(z) \ud z & \leq  \frac{K}{|2B_x|} \int_{ \frac{8\lambda}{6-\lambda} B_y} J_f(z) \ud z \\
 &\lesssim_{n, K} \frac{1}{|2B_x|} \int_{\frac{4\lambda}{6-\lambda} B_y} J_f(z) \ud z \\
 &\lesssim_{n, K} \frac{1}{|2B_x|} \int_{\lambda B_y} |Df(z)|^n \ud z \approx_{n, K} \a_{f,\lambda}^n(y).  
 \end{align*}
 The opposite estimate follows from the similar reasoning upon noticing that $\lambda B_y\subset \frac{\lambda}{3}(4+\frac{\lambda}{2}) B_x\subset \frac{3}{2}\lambda B_x$, since for $z\in \lambda B_y$ and $y\in \frac{\lambda}{3} B_x$ it holds that 
 \[
 |z-x| \leq \frac{\lambda}{2}(1-|y|)+\frac{\lambda}{6}(1-|x|)\leq \frac{\lambda}{6}\Big[1+3\Big(\frac{\lambda}{6}+1\Big)\Big](1-|x|).
 \]
\end{proof}

\subsection{Harnack inequalities for averaged derivatives of differentiable mappings}
%

In this section we discuss several Harnack estimates for the differential operators $\a_f$ under mild assumption on the map $f$, i.e. without requiring quasiregularity of $f$. 

\begin{lem}\label{lem:harnack}
	Let map $f:\B\rightarrow \R^n$ be such that its Jacoby matrix satisfies ${D\!f \in L^n_{loc}(\B)}$. Then for any $\lambda \in (0,1/2 )$ there exist constants $C_H=C_H(n, \lambda)>0$ and $0<\lambda_1<1<\lambda_2<2$, with $\lambda_1, \lambda_2$ depending only on $\lambda$, such that for all $x\in\B$ and $y\in\lambda B_x$
	\begin{align}
		& \frac{1}{C_H} \a_{f,\lambda_1}(y)\leq \a_f(x)\leq C_H\a_{f,\lambda_2}(y), \label{est1-harnack}\\
		& \frac{1}{C_H} \a_{f,\lambda_1}(x)\leq \a_f(y)\leq C_H\a_{f,\lambda_2}(x). \label{est2-harnack}
	\end{align}
Moreover, $\lambda_1$ and $\lambda_2$ in fact can be chosen so that $0<\lambda_1<1-\frac32\lambda$ and $\frac{2+2\lambda}{2-\lambda}<\lambda_2<2$.
\end{lem}

Let us remark that both the right-hand side estimates in Lemma~\ref{lem:harnack} will be used in what follows: \eqref{est1-harnack} in Theorem~\ref{thm:characterizationsnolder} and~\eqref{est2-harnack} to obtain Corollary~\ref{cor:harnack-est}(b).

\begin{proof}
	Fix $x\in \B$ and $y\in\lambda B_x$. If $z\in \lambda_1 B_y$ for some $\lambda_1>0$, then by the triangle inequality and since $1-|y|\leq d(y, x_{min})$, where $x_{min}\in \Sn$ stands for a point realizing the distance $\dist(x,\Sn)=d(x, x_{min})$, we have the following estimate
	\begin{align*}
 		|x-z|&\leq |x-y|+|y-z| \nonumber \\
 		&< \frac{\lambda}{2}(1-|x|)+\frac{\lambda_1}{2}(1-|y|) \nonumber \\
 		&\leq \frac{\lambda}{2}(1-|x|)+\frac{\lambda_1}{2}( |x-y|+(1-|x|)) \nonumber \\
 		&< \frac{\lambda}{2}(1-|x|)+\frac{\lambda_1}{2}\left(\frac{\lambda}{2}+1\right)(1-|x|)=\left[\frac{\lambda_1}{2}\left(\frac{\lambda}{2}+1\right) +\frac{\lambda}{2}\right](1-|x|). 
	\end{align*}
	Therefore, we have shown that $z\in B_x$, provided that $\left[\frac{\lambda_1}{2}\left(\frac{\lambda}{2}+1\right) +\frac{\lambda}{2}\right]<\frac{1}{2}$, i.e., when 
	$$
	\lambda_1 < \frac{2-2\lambda}{2+\lambda}\in\left(\frac{2}{5},1\right).
	$$
	Hence, we get that, regardless of the value of $\lambda$, for any $\lambda_1\leq\frac{2}{3}$ it holds that $\lambda_1B_y\subset B_x$.

\smallskip
	
\noindent Now suppose $z\in B_x$. Then
 $$
  |y-z|\leq |y-x|+|x-z|< \left(\frac{\lambda}{2}+\frac{1}{2}\right) (1-|x|).
 $$ 
  On the other hand
  $$
   1-|x|\leq |x-y|+1-|y|< \frac{\lambda}{2}(1-|x|)+1-|y|\Rightarrow 1-|x|\leq \frac{2}{2-\lambda }(1-|y|).
  $$
  As consequence, we get $|y-z|<\frac{2+2\lambda}{2-\lambda}\frac{1}{2}(1-|y|)$, proving that $B_x\subset\lambda_2B_y$ for any
\begin{equation}\label{est-lambda-lem-harnack}
 \lambda_2\geq \frac{2+2\lambda}{2-\lambda}\in \left(1,2\right).
\end{equation}
 Let $C_d$ be the doubling constant of the Lebesgue measure in $\R^n$. Then
	\begin{align*}
		&|2B_y|\geq |\lambda_2B_y|\geq |B_x|\geq C_d|2B_x|,\\
		&|2B_x|\geq |B_x|\geq |\lambda_1 B_y|\geq C_d^{j(\lambda_1)}|2B_y|,
	\end{align*} 
	where $j(\lambda_1)$ is the smallest integer such that $2B_y\subset  2^{j(\lambda_1)}\lambda_1B_y$. From this we easily conclude the first assertion~\eqref{est1-harnack} of the lemma, namely:
	$$ 
	\frac{1}{C_d^{1/n}} \a_{f,\lambda_1}(y)\leq \a_f(x)\leq C_d^{\frac{j(\lambda_1)}{n}}\a_{f,\lambda_2}(y).
	$$
	Similarly we show the second Harnack inequality in~\eqref{est2-harnack}. Let $z\in \tilde{\lambda}_1B_x$ for some $\tilde{\lambda}_1>0$. Then
$$
|y-z|\leq |y-x|+|x-z|< \frac{\lambda+\tilde{\lambda}_1}{2} (1-|x|)<  \frac{\lambda+\tilde{\lambda}_1}{2-\lambda }(1-|y|),
$$
	 proving that $\tilde{\lambda}_1B_x\subset B_y$ whenever $\frac{\lambda+\tilde{\lambda}_1}{2-\lambda }<\frac{1}{2}$, i.e.,
	$$
	\tilde{\lambda}_1< \frac{2-3\lambda }{2}\in \left( \frac{1}{4},1\right).
	$$
	Finally, suppose that $z\in B_y$. Then
	\begin{align*}
		|x-z|&\leq |x-y|+|y-z| \\
		&< \frac{\lambda}{2}(1-|x|)+\frac{1}{2}(1-|y|)\leq 
		\frac{\lambda}{2}(1-|x|)+\frac{1}{2}( |x-y|+(1-|x|))
		\\ &< \frac{\lambda}{2}(1-|x|)+\frac{1}{2}\left(\frac{\lambda}{2}+1\right)(1-|x|)=\frac{1}{4}(3\lambda +2)(1-|x|).
	\end{align*}
	Hence, $B_y\subset \tilde{\lambda}_2B_x$, provided that
	$$\tilde{\lambda}_2\geq \frac{3\lambda +2}{2}\in \left( 1, \frac{5}{4}\right) .$$
	Thus, we arrive at the estimate~\eqref{est2-harnack}
	$$ 
	\frac{1}{C_d^{j(\lambda_1)/n}}\a_{f,\tilde{\lambda}_1}(x)\leq \a_f(y)\leq C_d^{1/n}\a_{f,\tilde{\lambda}_2}(x).
	$$
	In order to unify constants in both Harnack estimates we set $C_H:=C_d^{j(\lambda_1)/n}$ and notice that $C_H$ only depends on $n$ and $\lambda_1$ (thus, indirectly on $\lambda$). Moreover,
\begin{align}
 &\lambda_1:=\min \{ \lambda_1,\tilde{\lambda}_1\}<\min\left\{\frac{2-2\lambda}{2+\lambda}, \frac{2-3\lambda }{2}\right\}=\frac{2-3\lambda }{2}, \label{eq-lambda1}\\
 &\lambda_2:=\max\{ \lambda_2,\tilde{\lambda}_2\}>\max\left \{\frac{2+2\lambda}{2-\lambda}, \frac{3\lambda +2}{2}\right\}=\frac{2+2\lambda}{2-\lambda}. \label{eq-lambda2}
\end{align}
 \end{proof}

Similar arguments yield the following variants of the Harnack inequalities:
\begin{cor}\label{cor:harnack-est}
Let map $f:\B\rightarrow \R^n$ be such that its Jacoby matrix satisfies ${D\!f \in L^n_{loc}(\B)}$.
Then the following Harnack estimates hold for the average derivatives $\a_{f,\lambda}$:
\begin{enumerate}
\item[(a)] For every $0<\lambda_1<2$ and $\lambda \in (0,  \lambda_1/3 )$ it holds for all $x\in \B , y\in\lambda B_x$  that
\begin{equation}
\a_{f,\lambda}(x)\leq \a_{f,\lambda_1}(y). 
\label{eq:Harnacklambda1prime}
\end{equation}
\item[(b)] For every $\lambda_2' \in (1,2)$ there exist parameters $\lambda\in (0,1/2)$ and $\lambda_2 <\lambda_2'$ such that for all $x\in \B$ and $y\in \lambda B_x$ it holds that
\begin{equation}
\a_f(y)\lesssim \a_{f,\lambda_2}(x) \lesssim \a_{f,\lambda_2'}(y). \label{cor12:est2}
\end{equation}
\item[(c)] For every $\lambda_1'\in (0,1)$ there exist parameters $\lambda\in (0,1/2)$ and $\lambda_1>\lambda_1'$ such that for all $x\in \B$ and $y\in \lambda B_x$ it holds that
\begin{equation}
 \a_{f,\lambda_1'}(y)\lesssim \a_{f,\lambda_1}(x) \lesssim \a_{f}(y). \label{cor12:est3}
\end{equation}
\end{enumerate}
\end{cor}
\begin{proof}
{\bf Part (a).} In order to prove assertion~\eqref{eq:Harnacklambda1prime} let us fix $\lambda_1 >0$. Let further $x\in \B$ and $\lambda , \lambda_1' \in (0,1)$ for $\lambda_1'$ depending on $x$ whose value will be determined below. Then, for $y\in \lambda B_x$ and $z\in \lambda_1' B_x$ we have
$$
|y-z|\leq |y-x|+|x-z|\leq \left( \frac\lambda 2 + \frac{\lambda_1^\prime}{2}\right) (1-|x|)\leq \left( \frac\lambda 2 + \frac{\lambda_1^\prime}{2}\right)\frac{2}{2-\lambda } (1-|y|),
$$
and the last inequality follows from the estimate $1-|x|\leq |x-y|+1-|y|\leq (1+\frac{\lambda}{2})(1-|y|)$.
Thus, in order to get that $\lambda_1' B_x\subset \lambda_1 B_y$ we need to ensure that
$$
\frac{\lambda + \lambda_1^\prime }{2-\lambda}\leq \frac{\lambda_1}{2}.
$$
This condition can be achieved, for instance, by setting $\lambda_1':=\lambda <\frac{\lambda_1}{3}$. Then $\lambda  B_x\subset \lambda_1 B_y$ and the assertion~\eqref{eq:Harnacklambda1prime} follows from the definition of $\a_{f,\lambda}$. 
\medskip
		
\noindent {\bf Part (b).} Let  $\lambda\in (0,1/2)$ and $\lambda_2\in (1,2)$ be as in Lemma~\ref{lem:harnack}; moreover, recall that $\lambda_2$ depends on $\lambda$. Then, for a fixed $x\in \B$ and all $y\in\lambda B_x$ and $z\in \lambda_2 B_x$ we have
	\begin{align*}
		|z-y|\leq |z-x|+|x-y|\leq \left( \frac{\lambda_2}{2}+\frac{\lambda}{2}\right) (1-|x|)\leq \left( \frac{\lambda_2}{2}+\frac{\lambda}{2}\right)\frac{2}{2-\lambda}(1-|y|).
	\end{align*}
In consequence, $z\in \lambda_2' B_y$ for every $\lambda_2' \in (1,2)$, provided that $\frac{\lambda_2+\lambda}{2-\lambda}\leq\frac{\lambda_2'}{2}$. Furthermore, since by~\eqref{eq-lambda2} we have that $\lambda_2> \frac{2+2\lambda}{2-\lambda}>1$, one may get any $\lambda_2' \in (1,2)$, upon choosing small $\lambda \approx 0$, resulting in big $\lambda_2\approx 1$.
Therefore, we may choose $\lambda_2$ satisfying $\lambda_2<\lambda_2' \in (1,2)$ such that for all $x\in \B$ and $y\in \lambda B_x$ 
\begin{equation*}
    \a_{f,\lambda_2}(x) \lesssim \a_{f,\lambda_2'}(y).
\end{equation*}
From this estimate, assertion~\eqref{cor12:est2} follows directly upon combining with the right-hand side estimate in~\eqref{est2-harnack}.
\medskip

\noindent {\bf Part (c).} Fix $\lambda_1' \in (0,1)$ and a point $x\in \B$. Then, for each $\lambda \in (0, \frac12)$ and $y\in \lambda B_x$ and $z\in \lambda_1'B_y$ the triangle inequality implies that
$$
|z-x|\leq \frac{\lambda_1^\prime}{2}(1-|y|)+\frac{\lambda }{2}(1-|x|)\leq  \left(\frac{\lambda_1'}{2}(1+\frac{\lambda}{2}) +\frac{\lambda}{2}\right) (1-|x|)
$$
Now let $\lambda_1$ be as in Lemma~\ref{lem:harnack}, then $z\in\lambda_1 B_x$ if
$$
\lambda_1^\prime (1+\frac{\lambda}{2} )+\lambda \leq \lambda_1.
$$
However, recall that by~\eqref{eq-lambda1} in Lemma~\ref{lem:harnack}, it holds that $\lambda_1<1-\frac32 \lambda$ and, thus, letting $\lambda \rightarrow 0^+$ allows for $\lambda_1\rightarrow 1^-$. By a limiting argument, since $\lambda_1^\prime <1$, we can choose $\lambda$ and $\lambda_1$ which verify both Lemma~\ref{lem:harnack} and the assertion~\eqref{cor12:est3}.
\end{proof}

\subsection{Reverse quantitative Harnack inequality for averaged derivative} 

The goal of this section is to prove Theorem~\ref{thm:afintegrability}, an integral variant of the Harnack estimate for $\a_f$, which we call the~\emph{reverse quantitative Harnack inequality}. The justification of this name comes from the fact that in the assertion of Theorem~\ref{thm:afintegrability} the parameter $\lambda_2$ in $\a_{f,\lambda_2}$ on the left-hand side is larger than $\lambda_1$ in $\a_{f,\lambda_1}$ appearing on the right-hand side of the estimate. Theorem~\ref{thm:afintegrability} leads to a handy integral estimates relating powers of norms $|Df|$ and $\a_f$, see Corollary~\ref{cor:lem32}, and is also employed in the proof of Theorem~\ref{thm:characterizationsnolder}. Moreover, as remarked after the statement of the theorem, it holds without the quasiregularity assumption on the map.

In order to improve the clarity of the proof of the theorem, we first show the following three auxiliary lemmas.

\begin{lem}\label{lem: uniscale-balls}
	For any point $z \in \B$ and $\lambda_2\in (1,2)$ there exist $\tau_0=\tau_0(\lambda_2)\in (0,\frac23),\tau_1 =\tau_1 (\tau_0,\lambda_2)\in (1,2)$, both independent of $z$, such that for every $x\in \tau_0 B_z$ it holds that $\lambda_2 B_x\subset \tau_1 B_z$. 
	
Moreover, for each $\kappa$ such that $\tau_1 <\kappa <2$ there exists $\eta =\eta (\tau_0, \kappa )$ such that $\eta B_x\subset \kappa B_z$ for all $x\in \tau_1 B_z$. 
\end{lem}

The lemma is the key auxiliary tool for Theorem~\ref{thm:afintegrability} and captures the following geometric observations (see also Figure~\ref{fig:comparisondiagram}):
\begin{itemize}
\item First assertion of the lemma says that for a given point $z\in \B$ we can find a slightly (smaller) hyperbolic ball $\tau_0 B_z$ contained in the hyperbolic ball $B_z$, such that when considering hyperbolic balls $\lambda_2 B_x$ centered at points in that smaller ball  $\tau_0 B_z$, as close to $\partial \B$ as possible (since $\lambda_2$ is allowed to be close to $2$), there is a (larger) hyperbolic ball $\tau_1 B_z$ that contains all balls $\lambda_2 B_x$. Moreover, constants $\tau_1$ and $\tau_2$ can be chosen uniformly with respect to points $z$ and $x$. Observe further, that $\tau_0$ is chosen to be sufficiently small in order to handle the case when the ball $\tau_0B_z$ contains points $x$ close to the origin, and so the ball $\lambda_2B_x$ could a priori be large.
\item We can enlarge the above ball $\tau_1 B_z$ to the ball $\kappa B_z$, which can be as close to $\partial \B$ as possible and, similarly as above, be able to find that hyperbolic balls $\eta B_x$ are contained in $\kappa B_z$, now for points $x$ in the (larger) hyperbolic ball $\tau_1 B_z$. Furthermore, the constant $\eta$ is independent of $z$ and $x$.
\end{itemize}

\begin{figure}[ht]
\centering
\includegraphics[trim=2.8cm 1.6cm 1.5cm 1.6cm, clip, scale=0.6]{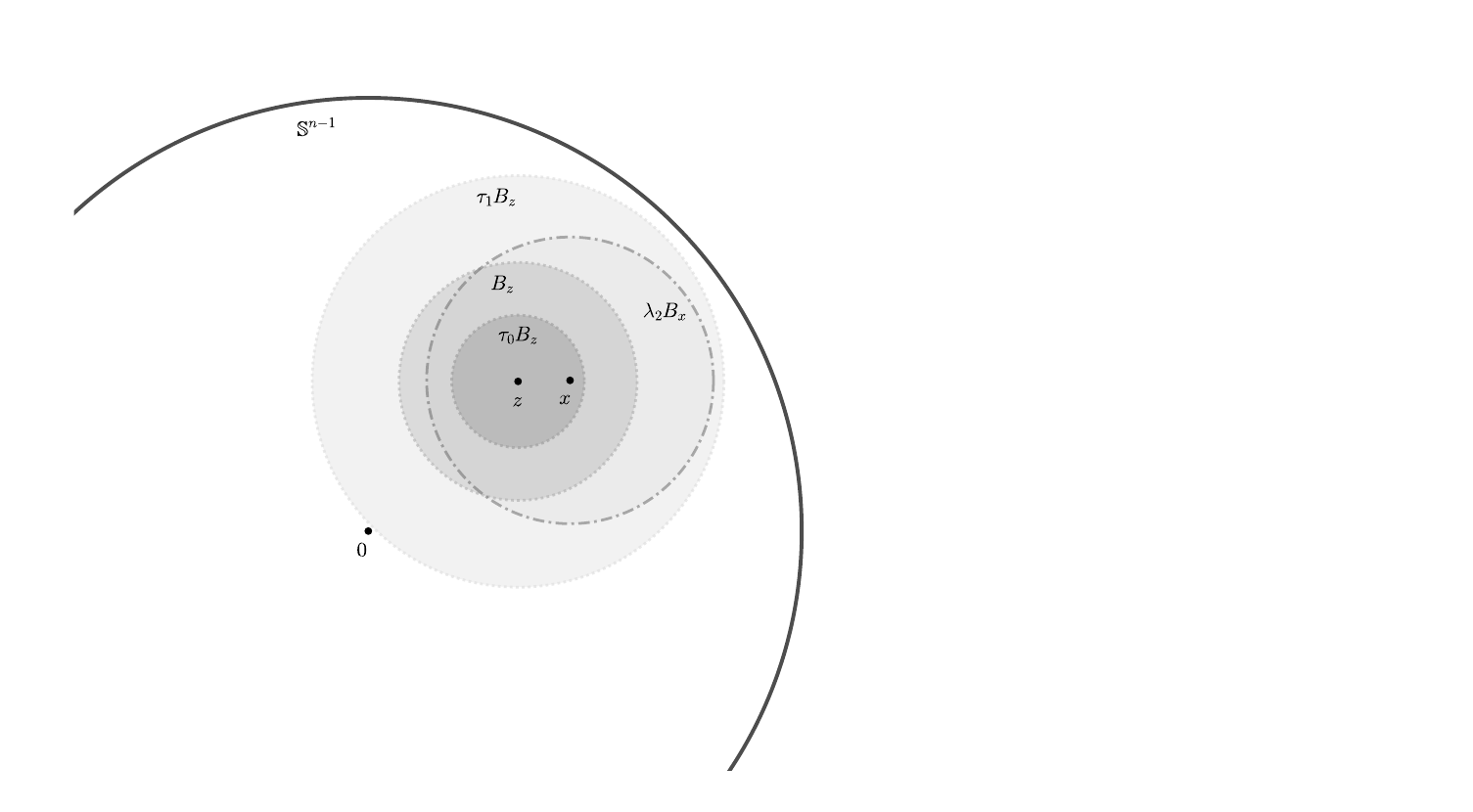}
\caption{Lemma~\ref{lem: uniscale-balls}}
\label{fig:comparisondiagram}
\end{figure}

\begin{proof}
	Let $z\in \B$ and $\lambda_2 \in (1,2)$ be any. Then, for all $x\in \tau_0 B_z$ and $y\in \lambda_2 B_x$ we have
	\begin{align*}
		|y-z|& \leq |y-x|+|x-z|\leq \frac{\lambda_2}{2}(1-|x|)+\frac{\tau_0}{2}(1-|z|)\leq \frac{\lambda_2}{2}(|x-z|+1-|z|) +\frac{\tau_0}{2}(1-|z|)\\
		&\leq \frac{\lambda_2}{2}(\frac{\tau_0}{2}+1)(1-|z|)+\frac{\tau_0}{2}(1-|z|)=\frac{\tau_0(\frac{\lambda_2}{2} +1)+\lambda_2}{2}(1-|z|).
	\end{align*}
	For $\tau_0< \frac{2(2-\lambda_2)}{\lambda_2+2}$ we then set $\tau_1 := \tau_0(\frac{\lambda_2}{2} +1)+\lambda_2<2$, and by the above estimate it follows that $y\in \tau_1 B_z$ and, hence, $\lambda_2 B_x\subset \tau_1 B_z$.
	
	For the second assertion of the lemma, let $\kappa$ be any number in $(\tau_1 ,2)$ and let $x\in \tau_1 B_z$. Then, by reasoning similar to the one in the first assertion, it holds for $\eta >0$ and $y\in \eta B_x$ that
	\begin{align*}
		|y-z|&\leq |y-x|+|x-z|\leq \frac{\eta}{2}(1-|x|)+\frac{\tau_1}{2}(1-|z|)\leq \frac{\eta }{2}(|x-z|+1-|z|) +\frac{\tau_1 }{2}(1-|z|)\\
		&\leq \frac{\eta }{2}(\frac{\tau_1}{2}+1)(1-|z|)+\frac{\tau_1}{2}(1-|z|)= (\frac{\eta}{2}(\frac{\tau_1}{2}+1)+\frac{\tau_1}{2})(1-|z|).
	\end{align*}
Therefore, in order to have $y\in \kappa B_z$ it must hold that $(\frac{\eta}{2}(\frac{\tau_1}{2}+1)+\frac{\tau_1}{2}) \leq \frac{\kappa}{2}$, or equivalently $\eta\leq \frac{2(\kappa -\tau_1)}{\tau_1 +2}$, which is possible since $\kappa >\tau_1$. Thus, for any such $\eta$ we have $\eta B_x\subset \kappa B_z$.
\end{proof}

\begin{lem}\label{lem:uniscale-balls2}
	Let $x\in\B$, $\lambda\in (0,1)$ and $y,z\in \lambda B_x$. Then $y\in 4\lambda B_z$. 
\end{lem}
\begin{proof}
	Let $y,z\in \lambda B_x$. It holds that
\[
1-|x|\leq |x-z|+1-|z|\leq \frac \lambda 2 (1-|x|)+1-|z|,\,\,\hbox{and so}\,\, 1-|x|\leq\frac{2}{2-\lambda}(1-|z|).
\]
Then, the assertion of the lemma follows immediately, as
\begin{align*}
|y-z|\leq |y-x|+|x-z|\leq \lambda (1-|x|)\leq\frac{2\lambda }{2-\lambda}(1-|z|)\leq 2\lambda (1-|z|),
\end{align*}
 and thus $y\in 4\lambda B_z$.
\end{proof}

\begin{lem}\label{lem:cover-aux}
Let $x\in\B$ and $\eta_1,\eta_2 \in (0,2)$. Then, there exists a finite covering of the ball $\eta_1B_x$ by balls $\{ \eta_2 B_{x_i(x)}\}_{i=1}^N$, where $N$ depends only on $n,\eta_1$ and $\eta_2$.
\end{lem}
\begin{proof}
By applying metrically doubling condition $m$ times, for the value of $m$ to be determined at the end of the proof, we have that $\eta_1 B_x$ can be covered by the following collection of 
balls
$$
  B_i(x):=B\left(x_i(x),2^{-m-1}\eta_1(1-|x|)\right),\quad x_i(x)\in \eta_1 B_x\quad \hbox{ for } 1\leq i\leq 2^{mn}.
$$
Since $x_i(x)\in \eta_1B_x$, the similar argument as in the proof of the previous two lemmas gives the following estimate
$$
 1-|x|\leq |x-x_i(x)|+1-|x_i(x)|\leq \frac{\eta_1}{2} (1-|x|)+1-|x_i(x)|,\hbox{ and hence }1-|x|\leq \frac{2}{2-\eta_1}(1-|x_i(x)|).
$$
Therefore, the radii of balls $B_i(x)$ satisfy
$$
\mathrm{rad} (B_i(x))=2^{-m-1}\eta_1(1-|x|)\leq 2^{-m}\frac{\eta_1}{2-\eta_1}(1-|x_i(x)|).
$$
Finally, upon choosing $m$ such that $2^{-m}\frac{\eta_1}{2-\eta_1}<\frac{\eta_2}{2}$, we conclude that $B_i(x)\subset \eta_2 B_{x_i(x)}$. 
\end{proof}


\begin{theorem}\label{thm:afintegrability}
Let $f:\B\rightarrow \R^n$ be a $K$-quasiregular map and $u$ a non-negative function satisfying the Harnack inequality:
	\[
	u(x)\simeq u(y)\quad \hbox{ for all } x\in\B \hbox{ and }y\in B_x,
	\]  
	with the comparison constant independent of $x$ and $y$.
%
	Then for every $p\in (0,\infty )$ and $0<\lambda_1 <\lambda_2<2$ we have
	$$
	 \int_{\B}\a^p_{f,\lambda_2}(x)\,u(x)\ud x \lesssim \int_{\B}\a^p_{f,\lambda_1}(x)\, u(x)\ud x,
	 $$
	where the comparison constant depends on $\lambda_1 ,\lambda_2, n$ and $p$.
\end{theorem}

We remark that Theorem~\ref{thm:afintegrability} holds under the weaker assumption, due to Lemma~\ref{lem:harnack}, that $Df$, the Jacoby matrix of $f$, satisfies ${D\!f \in L^n_{loc}(\B)}$ and no quasiregularity is employed. However, since we apply the result only in the case of quasiregular mappings, we formulate it in that setting, see also similar remark following the statement of Lemma~\ref{lem:2.5AK}.

\begin{remark}\label{u-HI-kappa}
 Let function $u$ be as in Theorem~\ref{thm:afintegrability} and the Harnack estimate for $u$ holds for all $x\in\B$ and $y\in B_x$, then we also have
\[
 u(x)\simeq u(y) \hbox{ for all } y\in\kappa B_x \hbox{ and } \kappa \in (1,2).
\]
The proof follows immediately from Lemma~\ref{lem:cover-aux}. Namely, there is a covering of $\kappa B_x$ by balls $\{ \frac 12 B_{x_i}\}_{i=0}^N$ with $N=N(n,\kappa )$ such that $x\in B_{x_0}, y\in B_{x_N}$ and $x_i\in B_{x_{i-1}}$ for all $i=1,\dots N$  and so, it holds that $u(x)\simeq u(x_0)\simeq\cdots\simeq u(x_N)\simeq u(y)$.
\end{remark}

 \begin{proof}
 We prove the theorem under the assumption that $\lambda_2>1$. The proof for $\lambda_2\leq 1$ follows by the trivial observation that $\a_{f,\lambda_2'}(x)\leq \a_{f, \lambda_2}(x)$ for any $\lambda_2'\leq \lambda_2$ and for almost all $x$ in $\B$.

Let $\tau_0, \tau_1, \kappa$ and $\eta$ as in Lemma~\ref{lem: uniscale-balls}. Furthermore, let $0<\lambda <\min \{ \lambda_1/12, \eta \}$. The choice of $\lambda_1/12$ is to ensure that we can apply the Harnack inequality~\eqref{eq:Harnacklambda1prime} in Corollary~\ref{cor:harnack-est}(b) on balls $4\lambda B_x$ instead of $\lambda B_x$, see Lemma~\ref{lem:uniscale-balls2}. 

Consider a Whitney covering $\{ B_j\}_{j\in\N}$ of the ball $\B$ with the following properties:
\begin{enumerate}
	\item $\displaystyle \B\subset\bigcup_{j\in\N}B_j$.
	\item $B_j=B(z_j, \frac{\tau_0}{2} (1-|z_j|))=\tau_0 B_{z_j}$ with $z_j\in \B$.
	\item $\displaystyle \sum_{j\in\N}\chi_{\kappa B_{z_j}} = \sum_{j\in\N}\chi_{\frac{\kappa }{\tau_0}B_{j}} \leq C<\infty $.
\end{enumerate}

Fix $j\in\N$ and let $x\in B_j$. By Lemma~\ref{lem:cover-aux} for $\eta_1:=\lambda_2$ and $\eta_2:=\lambda$ we have a covering $\{ \lambda B_{x_i(x)}\}_{i=1}^N$ of $\lambda_2B_x$. Notice that, since $x\in B_j$ and $x_i(x)\in \lambda_2B_x$, by Lemma~\ref{lem: uniscale-balls} we have $x_i(x)\in \tau_1 B_{z_j}$.

Moreover, observe that by the second assertion of Lemma~\ref{lem: uniscale-balls}. applied to points $z:=z_j$ and $x:=x_i(x)$. we also have that
for $\lambda<\eta$ it holds 
\begin{equation}\label{thm36-aux1}
\lambda B_{x_i(x)}\subset \kappa B_j,
\end{equation}

see Figure~\ref{fig:comparisondiagram-left}.
\begin{figure}
\centering
\begin{minipage}{.5\textwidth}
  \centering
  \includegraphics[trim=0cm 1.1cm 1.5cm 1cm, clip, scale=0.6]{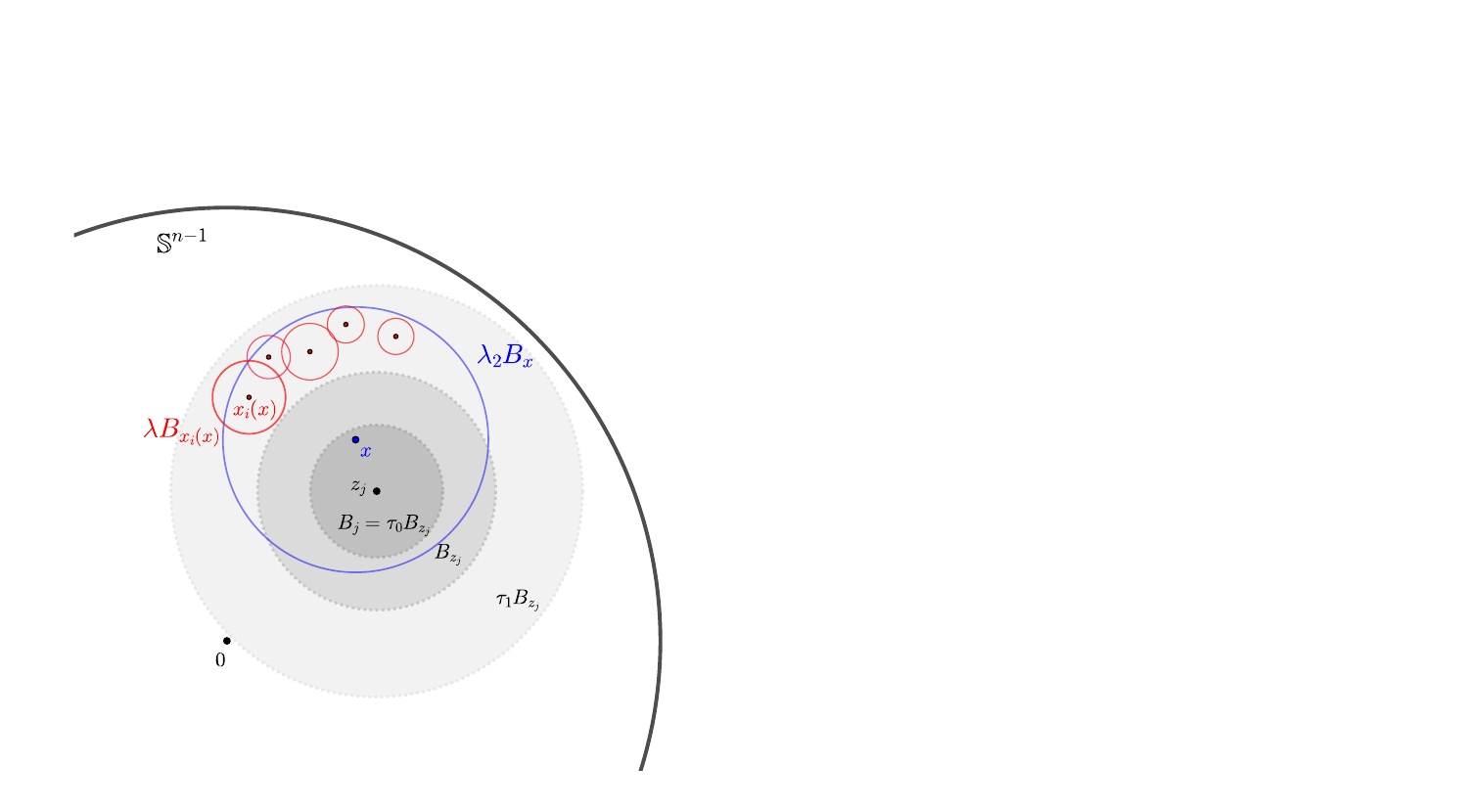}
  \caption{{\footnotesize Covering $\{ \lambda B_{x_i(x)}\}_{i=1}^N$ of $\lambda_2B_x$}.}
  \label{fig:comparisondiagram-left}
  \end{minipage}%
\begin{minipage}{.5\textwidth}
  \centering
  \includegraphics[trim=10.5cm 0.8cm 1.5cm 1cm, clip, scale=0.5]{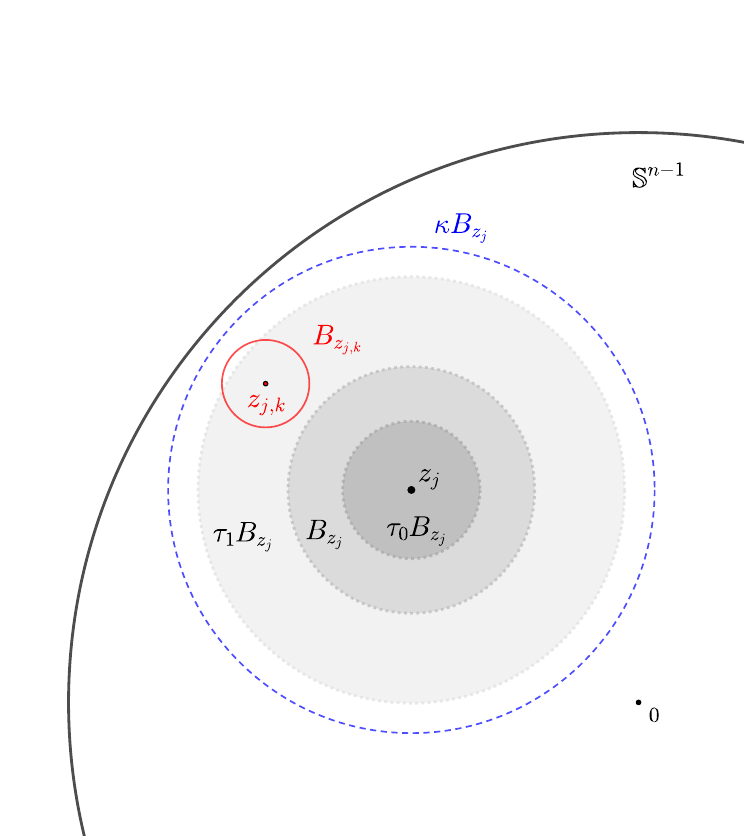}
  \caption{{\footnotesize Coverings $\{ \lambda B_{z_{j,k}}\}_{k=1}^M$ of $\tau_1 B_{z_j}$}.}
  \label{fig:comparisondiagram-right}
\end{minipage}
\end{figure}
%
Therefore,
\begin{equation}
a^p_{f,\lambda_2}(x)\!=\!\left(\frac{1}{|2B_x|}\int_{_{\!\lambda_2B_x}}|Df(y)|^n\ud y\right)^{\frac pn}\!\leq\!\left(\frac{1}{|2B_x|}\sum_{i=1}^N\int_{_{\!\lambda B_{x_i(x)}}}|Df(y)|^n\ud y\right)^{\frac pn}\!\lesssim\!\sum_{i=1}^Na_{f,\lambda }^p(x_i(x)).\label{eq:coveringargument1}
\end{equation}
In the next step we need to consider a family of coverings of balls $B_j$ for each $j\in\N$. Namely, we apply Lemma~\ref{lem:cover-aux} with $x:=z_j$, $\eta_1:= \tau_1$ and $\eta_2:=\lambda$ to cover balls $2\tau_1 B_j=\tau_1 B_{z_j}$ by balls $\{ \lambda B_{z_{j,k}}\}_{k=1}^M$. Similarly to the argument at~\eqref{thm36-aux1}, if $\lambda <\eta$, then $\lambda B_{z_{j,k}}\subset \kappa B_{z_j}$ for all $k$ (see Figure~\ref{fig:comparisondiagram-right}). Therefore, by assumptions and Remark~\ref{u-HI-kappa}, we have that
\begin{equation}\label{eq:harnack1coveringproof}
u(x)\simeq u(y) \quad\mbox{ for all }x\in B_j, y\in \lambda B_{z_{j,k}}.
\end{equation}
On the other hand, let $x\in B_j$ and $i\in \{ 1,\dots , N\}$. Since $x_i(x)\in \tau_1 B_{z_j}$, then there exists $k(i,x)\in\{ 1,\dots M\}$ such that $x_i(x)\in \lambda B_{z_{j,k(i,x)}}$, and by Lemma~\ref{lem:uniscale-balls2} we then have $y\in 4\lambda B_{x_i (x)}$ for all $y\in \lambda B_{z_{j,k(i,x)}}$. Since $4\lambda < \lambda_1 /3$ this implies, by the Harnack estimate~\eqref{eq:Harnacklambda1prime}, that
\begin{equation}\label{eq:harnack2coveringproof}
	a_{f,\lambda}(x_i(x))\lesssim a_{f,\lambda_1}(y),\quad \mbox{ for all }y\in \lambda B_{z_{j,k(i,x)}}.
\end{equation}
Observations~\eqref{eq:harnack1coveringproof} and~\eqref{eq:harnack2coveringproof} together lead to the following inequality:
\begin{align}
\a^p_{f,\lambda}(x_i(x))u(x)&=\frac{1}{|\lambda B_{z_{j,k(i,x)}}|}\int_{\lambda B_{z_{j,k(i,x)}}}\a^p_{f,\lambda}(x_i(x))u(x)\ud y \nonumber \\
	&\lesssim \frac{1}{|B_j|}\int_{\lambda B_{z_{j,k(i,x)}}}\a^p_{f,\lambda_1}(y)u(y)\ud y \nonumber \\
	& \leq \sum_{k=1}^M \frac{1}{|B_j|}\int_{\lambda B_{z_{j,k}}}\a^p_{f,\lambda_1}(y)u(y)\ud y. \label{eq:coveringargument2}
\end{align}
Notice that the comparison $|B_j|\simeq |\lambda B_{z_{j,k(i,x)}}|$ holds by applying the doubling condition, since $1-|z_{j,k}|\geq \frac{2-\tau_1}{2}(1-|z_j|)$. 

Finally, estimates~\eqref{eq:coveringargument1} and~\eqref{eq:coveringargument2} combined together with the bounded overlap of balls $\kappa B_{z_j}$ allow us to complete the proof of Theorem~\ref{thm:afintegrability}, as follows
\begin{align*} 
  &\int_{\B}\a^p_{f,\lambda_2}(x)\,u(x)\ud x \\
   &\phantom{AAA}\leq \sum_{j\in\N}\int_{B_j}\a^p_{f,\lambda_2}(x)\,u(x)\ud x   \\
   &\phantom{AA\,}\underset{\hbox{\tiny{(by~\eqref{eq:coveringargument1})}}}{\lesssim} \sum_{j\in\N} \sum_{i=1}^N \int_{B_j} \a^p_{f,\lambda} (x_i(x))u(x)\ud x \underset{\hbox{\tiny{(by~\eqref{eq:coveringargument2})}}}{\lesssim} \sum_{j\in\N}   \sum_{i=1}^N\int_{B_j}\frac{1}{|B_j|}\sum_{k=1}^M\int_{\lambda B_{z_{j,k}}} \a_{f,\lambda_1}^p (y)u(y)\ud y\\
    &\phantom{AAA}\lesssim N \sum_{j\in\N} \sum_{k=1}^M\int_{\lambda B_{z_{j,k}}} \a_{f,\lambda_1}^p (y)u(y)\ud y\lesssim NM\sum_{j\in\N}\int_{\kappa B_{z_j}} \a_{f,\lambda_1}^p (y)u(y)\ud y \\
    &\phantom{AAA}\lesssim \int_{\B}\a_{f,\lambda_1}^p(y)\, u(y)\ud y.
\end{align*}
\end{proof}

\subsection{Quantitative Harnack inequalities for averaged derivatives}

The Harnack estimates presented in Sections 3.2 and 3.3 result in the following counterpart of Lemma 2.5 in~\cite{ak}. As already mentioned, in consequence of Lemma~\ref{lem:2.5AK} we obtain a robust integral estimates relating powers of norms $|Df|$ and $\a_f$, see Corollary~\ref{cor:lem32}. We call such estimates \emph{the quantitative Harnack inequalities}, since they involve the integral inequalities for the averaged derivatives $\a_{f, \lambda}$ and the function satisfying the pointwise Harnack estimate, see~\eqref{assert1-lem32} and~\eqref{assert2-lem32}.
\begin{lem}\label{lem:2.5AK}
	Let $f:\B\rightarrow \R^n$ be $K$-quasiregular mapping and $u$ a non-negative function satisfying the Harnack inequality:
\[
   u(x)\simeq u(y)\quad \hbox{ for all } x\in\B \hbox{ and }y\in \frac{1}{2}B_x,
 \]  
   with the comparison constant independent of $x$ and $y$. Then for any $\lambda_2'\in (1,2)$ and for some $\lambda_1' \in (0,1)$ we have that for any $q\in (0,n]$ and $p\geq q$ it holds
\begin{align}
& \int_{\B}|Df|^q\a_{f}^{p-q}u\, \Leb\lesssim \int_{\B}\a_{f,\lambda_2^\prime }^pu\, \Leb, \label{assert1-lem32}\\
& \int_{\B}\a_{f,\lambda_1^\prime }^pu\, \Leb\lesssim \int_{\B}|Df|^q\a_{f}^{p-q}u\, \Leb. \label{assert2-lem32}
\end{align}
	Here, the comparison constants depend on $K,n,q$ and $u$.
	
\end{lem}
Let us remark that, in fact, the assertions~\eqref{assert1-lem32} and~\eqref{assert2-lem32} hold under weaker assumptions due to Lemma~\ref{lem:harnack}, namely that
\smallskip

\noindent (1) $Df$, the Jacoby matrix of $f$, satisfies ${D\!f \in L^n_{loc}(\B)}$,\\
(2) the reverse H\"older inequality holds for $|Df|$.
\smallskip

However, since the lemma in our work is applied only to the case of quasiregular mappings, we formulate it in that setting.

Before presenting the proof of Lemma~\ref{lem:2.5AK} let us discuss one of its consequences. The following 
result has the same assertion as Lemma 2.5 in~\cite{ak} in the setting of quasiconformal mappings. However, the proof of  Lemma 2.5 in~\cite{ak} relies on the Koebe theorem for $\a_f$ (\cite[Lemma 2.3]{ak}), a tool unavailable in the quasiregular setting, in general. Nevertheless, our Harnack estimates for $\a_{f,\lambda}$ in Lemma~\ref{lem:harnack} and flexibility in choosing $\lambda$ allow us to bypass the absence of the full (double-sided) Koebe estimate.

\begin{cor}\label{cor:lem32}
Let $f:\B\rightarrow \R^n$ be $K$-quasiregular mapping and $u$ be a non-negative function satisfying the Harnack inequality:
		\[
		u(x)\simeq u(y)\quad \hbox{ for all } x\in\B \hbox{ and }y\in B_x,
		\]  
		with the comparison constant independent of $x$ and $y$.  
		Then for any $q\in (0,n]$ and $p\geq q$ we have
		\begin{align*}
			\int_{\B}|Df|^q\a_{f}^{p-q}u\, \Leb\simeq \int_{\B}\a_{f}^pu\, \Leb
		\end{align*}
		Here, the comparison constants depend on $K,n,q$ and $u$. 
\end{cor}
\begin{proof} The proof follows directly from Lemma~\ref{lem:2.5AK} and Theorem~\ref{thm:afintegrability}, upon
applying the theorem with $\lambda_2:=\lambda_2'$ and $\lambda_1:=\lambda_1'$.
\end{proof}

\begin{proof}[Proof of Lemma~\ref{lem:2.5AK}]
%
%

Fix $\lambda\in (0,1/2)$ and let $\lambda_2\in (1,2)$ be as in Corollary~\ref{cor:harnack-est}(b), i.e. $\lambda_2 <\lambda_2' \in (1,2)$ and  for all $x\in \B$ and $y\in \lambda B_x$ the estimates~\eqref{cor12:est2} hold: 
\begin{equation}\label{eq:Harnack2lem3.4} 
	 \a_f(y)\lesssim \a_{f,\lambda_2}(x) \lesssim \a_{f,\lambda_2'}(y).
\end{equation}

By applying the second estimate in \eqref{eq:Harnack2lem3.4}, the definition of $\a_f$ and the doubling property of the Lebesgue measure together with Jensen's inequality we get the following lower estimate:
\begin{align*}
	\int_{\lambda B_x}\a^q_{f,\lambda_2^\prime}(y)\,\ud y &\gtrsim_{C_H}  \int_{\lambda B_x}\a^q_{f,\lambda_2}(x)\,\ud y =\int_{\lambda B_x}\left(\frac{1}{|2B_x|}\int_{\lambda_2 B_x}|Df(z)|^n\,\ud z \right)^\frac{q}{n}\,\ud y \nonumber \\
	&\gtrsim_{C_H}  \int_{\lambda B_x}\left( \vint_{\lambda_2 B_x}|Df(z)|^n\,\ud z \right)^\frac{q}{n}\ud y \gtrsim_{C_H} \int_{\lambda B_x} \vint_{\lambda_2 B_x}|Df(z)|^q\,\ud z \nonumber \\
	& \gtrsim_{C_H}\frac{|\lambda B_x|}{|\lambda_2B_x|} \int_{\lambda_2B_x}|Df(z)|^q\,\ud z \gtrsim_{C_H, C_d} \left(\frac{\lambda}{\lambda_2}\right)^{n}  \int_{\lambda B_x}|Df(z)|^q\,\ud z. 
\end{align*} 
Upon multiplying the both sides of the above estimate by $\a_{f,\lambda_2}^{p-q}(x)u(x)$ we obtain the following inequality
	\begin{equation}\label{eq: 3.4(2)}
		\a_{f,\lambda_2 }^{p-q}(x)u(x)\int_{\lambda B_x}|Df(y)|^q\,\ud y 
		\lesssim \a_{f,\lambda_2}^{p-q}(x)u(x)\int_{\lambda B_x}\a_{f, \lambda_2'}^q(z)\,\ud z.
	\end{equation}
Next, we apply the first inequality in~\eqref{eq:Harnack2lem3.4} together with the Harnack estimate for $u$ to obtain that
\begin{equation}	 
\int_{\lambda B_x}|Df(y)|^q(y)\a_{f}^{p-q}(y)u(y)\,\ud y \lesssim \a_{f,\lambda_2 }^{p-q}(x)u(x)\int_{\lambda B_x}|Df(y)|^q\,\ud y. \label{eq: Lemma34-aux}
\end{equation}
In consequence, estimate~\eqref{eq: 3.4(2)} and the second inequality in~\eqref{eq:Harnack2lem3.4} yield the following 
	\begin{align}
	\a_{f,\lambda_2 }^{p-q}(x)u(x)\int_{\lambda B_x}|Df(y)|^q\,\ud y \lesssim \a_{f,\lambda_2}^{p-q}(x)u(x)\int_{\lambda B_x}\a_{f, \lambda_2'}^q(z)\,\ud z \lesssim \int_{\lambda B_x}\a^p_{f,\lambda_2'}(z)u(z)\,\ud z. \label{eq:2.5(1)}
	\end{align}
Hence,~\eqref{eq: Lemma34-aux} and~\eqref{eq:2.5(1)} give the following estimate:
\begin{equation}\label{eq:2.5(2)}
 \int_{\lambda B_x}|Df(y)|^q(y)\a_{f}^{p-q}(y)u(y)\,\ud y \lesssim \int_{\lambda B_x}\a^p_{f,\lambda_2'}(z)u(z)\,\ud z.
\end{equation}
 In order to complete the proof of the assertion~\eqref{assert1-lem32} we appeal to the covering argument. Since the same argument gives us also the second assertion of the lemma, we now turn our attention to~\eqref{assert2-lem32} and complete both proofs in the next step. 
We apply Corollary~\ref{cor:harnack-est}(c) for $\lambda_1'$ small enough to find $\eta\in(0,\frac12)$ and $\lambda_1>\lambda_1'$ such that $\lambda_1<\frac14$ and the following estimate holds 
\begin{equation}
 \a_{f,\lambda_1'}(y)\lesssim \a_{f,\lambda_1}(x)\quad y\in \eta B_x. \label{aux333-lem32}
\end{equation}
 Here, we use $\eta$ instead of $\lambda$ in order to avoid confusion with $\lambda$ in the previous part of this proof.
  
By applying~\eqref{aux333-lem32},  the reverse H\"older estimate, see e.g. (5.2) in the proof of~\cite[Theorem 5.1]{bi} and Theorem 2 in~\cite{in}, we have 
\begin{align}
 \int_{\eta B_x}\a^q_{f,\lambda_1'}(y)\ud y &\lesssim  \int_{\eta B_x}\a^q_{f,\lambda_1}(x)\ud y 
 =\int_{\eta B_x}\left(\frac{1}{|2B_x|}\int_{\lambda_1 B_x}|Df(z)|^n\ud z \right)^\frac{q}{n}\ud y \nonumber \\
&  =  |\eta B_x|\left( \frac{|\lambda_1 B_x|}{|2B_x|}\vint_{\lambda_1 B_x} |Df(z)|^n \ud z \right)^{\frac{q}{n}} \nonumber\\
&\leq_{\tiny(\hbox{H\"older ineq.})} |\eta B_x|\left(\vint_{\lambda_1  B_x} |Df(z)|^{p_0}\ud z \right)^{\frac{q}{p_0}} \nonumber\\
& \lesssim_{\tiny(\hbox{Thm. 2 in~\cite{in}})} \frac{|\eta B_x|}{|\lambda_1 B_x|} \int_{\lambda_1 B_x}|Df(z)|^q\ud z\nonumber \\
&\simeq \int_{\lambda_1  B_x}|Df(z)|^q\ud z. \label{aux2-lem25}
\end{align}
 Here, Theorem 2 in~\cite{in} is applied with $p:=p_0$, the Gehring exponent $p_0>n$, $s:=n$, $r:=q$ and $\sigma=2$. Thus,~\eqref{aux2-lem25} together with~\eqref{aux333-lem32} and Lemma~\ref{lem:harnack} yield
\begin{align}
 \int_{\eta B_x} \a_{f,\lambda_1'}^{p}(y) u(y)\ud y & \underset{\hbox{\tiny{(by~\eqref{aux333-lem32})}}}{\lesssim} \a_{f,\lambda_1}^{p-q}(x)u(x)\int_{\eta B_x} \a^q_{f,\lambda_1'}(y) \ud y \nonumber \\
& \underset{\hbox{\tiny{(by~\eqref{aux2-lem25})}}}{\lesssim} \a_{f,\lambda_1}^{p-q}(x)u(x)\int_{\lambda_1 B_x}|Df(z)|^q\ud z  \nonumber \\
& \underset{\hbox{\tiny{(by~\eqref{est2-harnack})}}}{\lesssim} \int_{\lambda_1 B_x}|Df(z)|^q\a_{f}^{p-q}(z)u(z)\ud z. \label{eq:2.5(3)}
\end{align} 
In order to complete the proof of the lemma we appeal to a covering argument and consider the family of subsets $\{ B_i\}_{i\in\N}\subset\B$ with the following properties:
\begin{itemize}
\item[(C1)] $\B=\underset{i\in\N}{\bigcup}B_i$.
\item[(C2)] $B_i=c B_{x_i}$ for some collection of points $\{ x_i\}_{i\in\N}\subset\B$ and given $0<c<\lambda$ for the parameter $\lambda \in (0,\frac12)$ fixed in the beginning of this proof.
\item[(C3)] $\underset{i\in\N}{\sum}\chi_{\frac{\lambda}{c} B_i}\leq C<\infty$.
	\end{itemize}
The existence of such a covering follows from Lemma~\ref{lem:WhitneyCovering} applied with, in the notation of the lemma, $\eta:=c$ and $\tau:=\frac{\lambda}{c}\in (1,\frac{1}{c})$. In particular, the property (C3) implies that	
$\underset{i\in\N}{\sum}\chi_{\lambda B_{x_i}}\leq C<\infty$.
Then, the covering argument  together with~\eqref{eq:2.5(2)} imply the first assertion of Lemma~\ref{lem:2.5AK}:
\begin{align*}
\int_{\B}|Df(y)|^q\a_f^{p-q}(y)u(y)\ud y &\leq \sum_{i\in\N}\int_{\lambda B_{x_i}}|Df(y)|^q\a_f^{p-q}(y)u(y)\ud y\\
&\lesssim \sum_{i\in\N}\int_{\lambda B_{x_i}}\a_{f, \lambda_2'}^{p}(z)u(z)\ud z= \sum_{i\in\N}\int_{\lambda B_{i}}\a_{f, \lambda_2'}^{p}(z)u(z)\ud z\leq  \int_{\B}\a_{f, \lambda_2'}^{p}(z)u(z)\ud z.
\end{align*}
Similarly, the covering argument applied, in the notation of  Lemma~\ref{lem:WhitneyCovering}, with $c=\eta$ and $\tau=\frac{\lambda_1}{\eta}$ together with~\eqref{eq:2.5(3)} imply the second assertion of the lemma:
\begin{align*}
\int_{\B}\a_{f,\lambda_1'}^{p}(y)u(y)\ud y &\leq \sum_{i\in\N}\int_{\eta B_{x_i}}\a_{f,\lambda_1'
}^{p}(y)u(y)\ud y \\
&\lesssim \sum_{i\in\N}\int_{\lambda_1B_{x_i}}|Df(z)|^q\a_{f}^{p-q}(z)u(z)\ud z \\
& = \sum_{i\in\N}\int_{\frac{\lambda_1}{\eta} B_{i}}|Df(z)|^q\a_{f}^{p-q}(z)u(z)\ud z \\
&\leq \int_{\B}|Df(z)|^q\a_{f}^{p-q}(z)u(z)\ud z.
\end{align*}
This estimate completes the proof of Lemma~\ref{lem:2.5AK}.
\end{proof}

The following observation is a partial counterpart of the quasiconformal Lemma 2.3 in~\cite{ak} for quasiregular mappings.
\begin{lem}\label{lem:2.3AK}
	Let $f:\B\rightarrow\R^n$ be a nonconstant $K$-quasiregular mapping satisfying the multiplicity condition~\eqref{cond-m} with constants $C$ and $0\leq a<n-1$. Then for all $x\in \B$ and $\lambda\in (0,2)$ 
	$$
	 \a_{f,\lambda} (x)(1-|x|) \lesssim d(f(x),\partial f(\B)),
	$$
	with comparison constant depending on $n,\lambda, K$ and $C$.	
\end{lem}

The above estimate does not depend on the parameter $a$ due to the discussion on the multiplicity condition~\eqref{cond-m} on hyperbolic balls, see~\eqref{eq:Mhyperboliclambda} and~\eqref{eq:Mhyperbolic}. In what follows we use Lemma~\ref{lem:2.3AK} in the proof of Theorem~\ref{thm:characterizationsnolder} with $\lambda=1$. See also Remark~\ref{rem-lem23-v2} below for a variant of the lemma under the weaker multiplicity condition~\eqref{precond-m}.

Let us recall that in the case of quasiconformal mappings (and so, in particular for conformal ones), there is also a similar lower estimate for $\a_f$. However, in the lack of injectivity, such an estimate in general may fail, as can be seen in the following example already for analytic functions in the plane, i.e. for $1$-quasiregular maps.

\begin{ex}
Let us consider a holomorphic mapping $f(z)=z^2$ for $z=x+iy$ on the unit ball $\B\subset \R^2$, which is one of the simplest non-injective $1$-quasiregular mappings in the plane. 

By the direct computations we find that $f(z)=(x^2-y^2, 2xy)$ and the Euclidean norm of the Jacobi matrix reads $|Df|=2\sqrt{2}|z|$ and, since $f(\B)=\B$, we have that $d(f(z),\partial f(\B))=1-|f(z)|=1-|z|^2$. Thus, the Koebe distortion inequality
$$ 
 |Df(z)|\leq c \frac{d(f(z),\partial f(\B))}{1-|z|}\, \hbox{ takes the form }\, 2\sqrt{2}|z|\leq c (1+|z|),
$$
 and so it holds for all $z\in \B$ for appropriately large constant $c>0$, e.g. in the classical Koebe distortion inequality for conformal maps in the plane $c=4$, see Theorem 1.6 in~\cite{asge1}. However, if the lower estimate held, it would read: $2\sqrt{2}|z|\geq \frac{1}{c} (1+|z|)$. This fails for all $z$ on a small ball at $0$ with radius depending on $c$. Note that $0$ is the branching point of the quasiregular map $f$.

The similar computations show that $f(z)=z^n$ for $n>1$ also fails the lower estimate in Lemma~\ref{lem:2.3AK}.
\end{ex}
 
\begin{proof}[Proof of Lemma~\ref{lem:2.3AK}] 
We begin with the modulus of curve families argument based on the discussions in the proofs of~\cite[Theorem 1.8]{asge1} and~\cite[Theorem 5.3]{no} and going back to the seminal work~\cite{ge}. Even though the argument is standard for quasiconformal mappings, we provide more details in order to draw the readers attention to the differences between the setting of injective and quasiregular mappings.

For a fixed $x\in \B$ let us consider the following two quantities:
$$
d_1:=d(f(x),\partial f(\B)),\quad d_2 := \sup \{|f(x)-f(z)|\,:\,z\in\partial (\lambda B_x)\}=|f(x)-f(y)|,
$$
where $y$ satisfies $|y-x|=\frac12 \lambda (1-|x|)$ and is a point where the supremum in $d_2$ is attained. Moreover, we may assume that $f(x)\not=f(y)$, as otherwise if for all points $y\in \partial(\lambda B_x)$ it holds that $f(y)=f(x)$, then since the component functions of $f$ satisfy an $\mathcal{A}$-harmonic equation, see~\cite[Theorem 14.42]{hkm}, the strong maximum principle in~\cite[Thorem 6.5]{hkm} implies that $f=const$ on $\lambda B_x$. Then the assertion of Lemma~\ref{lem:2.3AK} holds trivially, since $\a_{f,\lambda}(x)\equiv 0$. 

Next, we consider a ring $R:=2B_x\setminus \overline{\lambda B_x}$ and recall that since $f$ is quasiregular and nonconstant, $f$ is continuous and open mapping in $\B$. Therefore, $f(\lambda B_x)\subset f(2B_x)$ and $\partial f(2B_x)\subset f(\partial( 2B_x))$. Moreover, $f(2B_x)$ is bounded and $\Rn\setminus f(2B_x)$ contains a unique unbounded component with non-empty boundary; also, by the definition of $d_2$ we have $f(\lambda B_x)\subset B(f(x),d_2)$. Let $f(R):=f(2B_x)\setminus \overline{f(\lambda B_x)}$. Then, by Theorem 4 in~\cite{ge}, see also the presentation in~\cite[Chapter 7]{vuo2}, we have for a constant $\gamma_n \in [4, 2e^{n-1})$ for $n>2$ and $\gamma_2=4$, that  
\[
 \frac{1}{M_n f(R)} \leq \ln \gamma_n^2\left(\frac{d_1}{d_2}+1\right),
\]
where the expression on the right-hand side is the estimate for the reciprocal of the $n$-modulus of the Teichm\"uller ring.
Upon combining this estimate with the modulus inequality for quasiregular mappings we have
\[ 
\frac{d_2}{d_1}\lesssim_{\gamma_n^2} \frac{1}{\ln \gamma_n^2\left(\frac{d_1}{d_2}+1\right)}\leq M_n f(R) \leq K M_n(R)   \approx_{\om_{n-1}, K} (\ln 2)^{1-n}.
\]
Hence, by the quasiregular change of variables and the multiplicity growth condition~\eqref{precond-m}
\begin{align*}
  \a_{f,\lambda }(x)^n &= \frac{1}{|2B_x|}\int_{\lambda B_x}|Df|^n \Leb\,
   \leq\, \frac{K}{|2B_x|}\int_{\lambda B_x}J_f \Leb	\\
	&=\frac{K}{|2B_x|}N(f,\lambda B_x)|f(\lambda B_x)|\,\underset{\hbox{\tiny{(by~\eqref{eq:Mhyperbolic})}}}{\leq}
	KC(\lambda, n)  \frac{|B(f(x), d_2)|}{|2B_x|} \\
	&\lesssim\, KC(\lambda, n) \left(\frac{d_2}{1-|x|}\right)^n \lesssim_{\gamma_n^2, K,\lambda, n} \left( \frac{d_1}{1-|x|}\right)^n
\end{align*}
and the assertion of the lemma follows.
\end{proof}

\begin{remark}\label{rem-lem23-v2}
Lemma~\ref{lem:2.3AK} holds with the stronger hypotheses, if we assume the multiplicity condition~\eqref{precond-m} instead of~\eqref{cond-m}, namely that
 $$
	 \a_{f,\lambda} (x)(1-|x|)^{1+\frac{a}{n}}\lesssim d(f(x),\partial f(\B)),
$$
with comparison constant depending on $n,\lambda, K, C$ and $a$. The proof goes the same lines as the the corresponding proof of the lemma, with the last estimate modified as follows:
\begin{align*}
  \a_{f,\lambda }(x)^n & \leq \frac{K}{|2B_x|}N(f,\lambda B_x)|f(\lambda B_x)|\,\leq\, \frac{K}{|2B_x|}N\left(f,B\left(0,\tfrac{\lambda}{2} (1-|x|)\right)\right)|B(f(x),d_2)| \\
	&\leq\frac{K\, C}{\left[ \left(1-\frac{ \lambda}{2}\right)(1-|x|)\right]^a} \frac{|B(f(x), d_2)|}{|2B_x|}\,\lesssim\, \frac{K\, C\left( 1-\frac{\lambda}{2}\right)^{-a}}{ (1-|x|)^a} \left( \frac{d_2}{1-|x|}\right)^n\\
	&\lesssim_{\gamma_n^2, K, n}  K\, C\left( 1-\frac{\lambda}{2}\right)^{-a } \left( \frac{d_1}{(1-|x|)^{1+\frac{a}{n}}}\right)^n.
\end{align*}
\end{remark}

\subsection{Proof of Theorem~\ref{thm:characterizationsnolder}}\label{sect35}

This section contains the core results of our work. In particular, we explain the role of the averaged derivatives $\a_{f,\lambda}$ in the studies of the boundary behavior of quasiregular mappings and their Hardy spaces. Moreover, we connect the geometry of mappings with studies of counterparts of some fundamental notions of harmonic analysis, such as \emph{the square function} and \emph{the non-tangential maximal function}. Indeed, below we consider the integral expressions involving the averaged derivatives $\a_{f,\lambda}$ multiplied by powers of the distance to the boundary function $1-|x|$, which is in the analogy to the square function of a Sobolev functions in $W^{1,2}(\Om,\R)$, e.g.
\[
 S_{\alpha}^2 u(\om):=\int_{\Gamma_{\alpha}(\om)} |\nabla u(x)|^2 \dist(x,\partial \Om)^{1-n}\ud x.
\] 
Our main auxiliary result is a quasiregular counterpart of Lemma 5.5 in~\cite{ak}. However, one can also view the proposition below as a counterpart of the Lusin area integral theorem, see for instance, Fefferman--Stein~\cite{fst}, Mitrea~\cite{mit}, Dahlberg~\cite{dah}, Burkholder--Gundy~\cite{bg}. For example, the second assertion in~\cite[Theorem 8 ]{fst} for a bounded harmonic function $u$ on a half-space reads for all $0<p<\infty$ (cf. Chapters VII.1 and VII.2 in~\cite{stein})
\[
\|\tilde{u}\|_{L^p}\lesssim \|\Nm(u)\|_{L^p} \approx \|S_\alpha u\|_{L^p}
\]
\begin{prop}\label{prop:Lp-norm-bound}
	Let $f:\B\rightarrow \R^n\setminus\{ 0\}$ be a quasiregular mapping in the Miniowitz class~\eqref{est-min} with finite multiplicity $N(f,\B)<\infty$. Then for any $0<p<\infty$, $0<\beta\leq 1$ and $0<\eta <2$. 
\begin{equation}\label{assert:Lp-norm-bound}
	 \int_{\Sn}|\tilde{f}(\om)|^pd\sigma (\om)\leq C \int_{\Sn}v(\om)^pd\sigma (\om),
\end{equation}
where the constant $C$ depends on $n,p, K$ and $\eta$ and the function $v:\Sn \to \R_{+}$ is defined as follows
	$$
	 v(\om):=\Nm\big(\mathrm a_{f,\eta} (x)(1-|x|)^\beta\big)= \displaystyle\sup_{x\in \Gamma (\om)}\mathrm a_{f,\eta} (x)(1-|x|)^\beta\,\hbox{ for a.e. point }\om\in \Sn.
	 $$
\end{prop}
Notice, that since by assumptions of Proposition~\ref{prop:Lp-norm-bound} the multiplicity $N(f,\B)<\infty$, the condition~\eqref{cond-m} trivially holds for any $a\geq 0$ and the constant $C:=N(f,\B)$.

The proof of Proposition~\ref{prop:Lp-norm-bound} is based on the following two technical lemmas whose proofs are presented in Appendix~\ref{appB}. First result corresponds to a number of results for quasiconformal mappings, see Lemma 4.2 in~\cite{ak}, Lemmas 4-5 in~\cite{z} and also Lemma 5.16 in~\cite{af} in the setting of the Heisenberg group $\mathbb{H}_1$.
Despite its technical nature, the result below is the level-sets type estimate which involves not only the map but also its averaged derivatives. Moreover, it leads to the good-lambda estimate, see the proof of Proposition~\ref{prop:Lp-norm-bound}. Such estimates are crucial in the harmonic analysis in the \emph{N/S} inequalities relating the non-tangential maximal functions and the area/square functions, see e.g.~\cite{dah, hmm, mit}.

\begin{lem}\label{lem:shadowmeasure}
	Let $f:\B\rightarrow \R^n$ be a $K$-quasiregular in the Miniowitz class~\eqref{est-min} with finite multiplicity $N(f,\B )<\infty $. Let $\beta ,\eta\in (0,1]$ and $M>1$ be sufficiently large. Then for every $x\in\B$ we have
$$
 \sigma\left(\left\{ \om \in S(x): |f(x)-\tilde{f}(\om)|\geq M(1-|x|)^\beta\mathrm a_{f,\eta} (x)\right\}\right)\leq C(n,K,\eta) \sigma (S(x))(\log M)^{1-n}.
$$
\end{lem}

The map $f$ is assumed to be in the Miniowitz class~\eqref{est-min} only so that together with the condition $N(f,\B )<\infty $ we may infer that the non-tangential limit map $\tilde{f}$ exists, but~\eqref{est-min} is not used in the proof of Lemma~\ref{lem:shadowmeasure}. 

Next result is a geometric observation about the shadow of points.

Below, by $d(x, y)$ we denote the intrinsic metric in $\Sn$ between points $x,y\in \Sn$.

\begin{lem}\label{lem:leavingshadow}
	Let set $U\subsetneq \Sn$ be closed. For a given parameter $c>0$, let $E$ denote the following set:
	\[
	 E\subset \{x\in \B : d(S(x),\Sn\setminus U)\simeq_{c} 1-|x| \}.
	\]
	Then there exists $\tau_0>1$ such that for every $x\in E$ there exists $y_x\in \B$ with the two properties:
	\begin{align}
	&(1)\,\,S(y_x)\cap (\Sn\setminus U)\neq\emptyset,  \label{prop1:lem-shadow}\\
	&(2)\,\,y_x \in \overline{\tau_0 B_x\cap \B}. \label{prop2:lem-shadow}
	\end{align}
\end{lem}

See Figure~\ref{fig3-right} for the illustration of Lemma~\ref{lem:leavingshadow}.

\begin{proof}[Proof of Proposition~\ref{prop:Lp-norm-bound}]
It suffices to prove the proposition for $\eta \in (0,\frac12)$, since if $\eta'\in [\frac12, 2)$ then it immediately holds that $\a_{f,\eta}<\a_{f,\eta'}$ and so, the function $v$ in the right-hand side of~\eqref{assert:Lp-norm-bound}, corresponding to $\a_{f,\eta}$, can be trivially estimated from above by the function $v$ corresponding to $\a_{f,\eta'}$.
\begin{figure}
\centering
\begin{minipage}{.5\textwidth}
  \centering
  \includegraphics[trim=1.5cm 0.8cm 1.5cm 1cm, clip, scale=0.7]{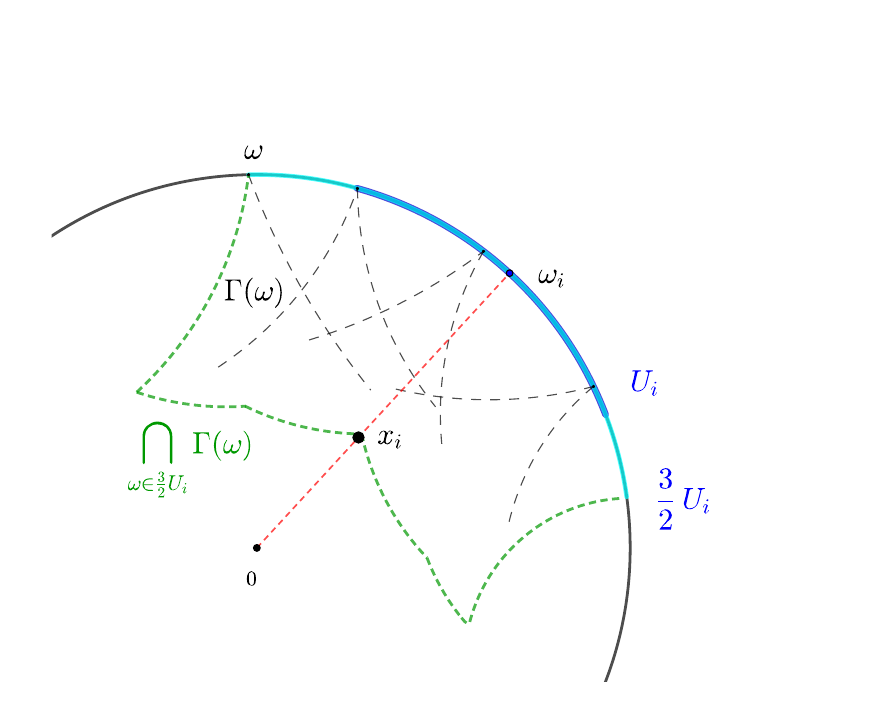}
  \caption{{\footnotesize Construction of point $x_i$ in Proposition~\ref{prop:Lp-norm-bound} for $n=2$}.}
  \label{fig3}
  \end{minipage}%
\begin{minipage}{.5\textwidth}
  \centering
  \includegraphics[trim=1.5cm 1.1cm 1.5cm 1.8cm, clip, scale=0.8]{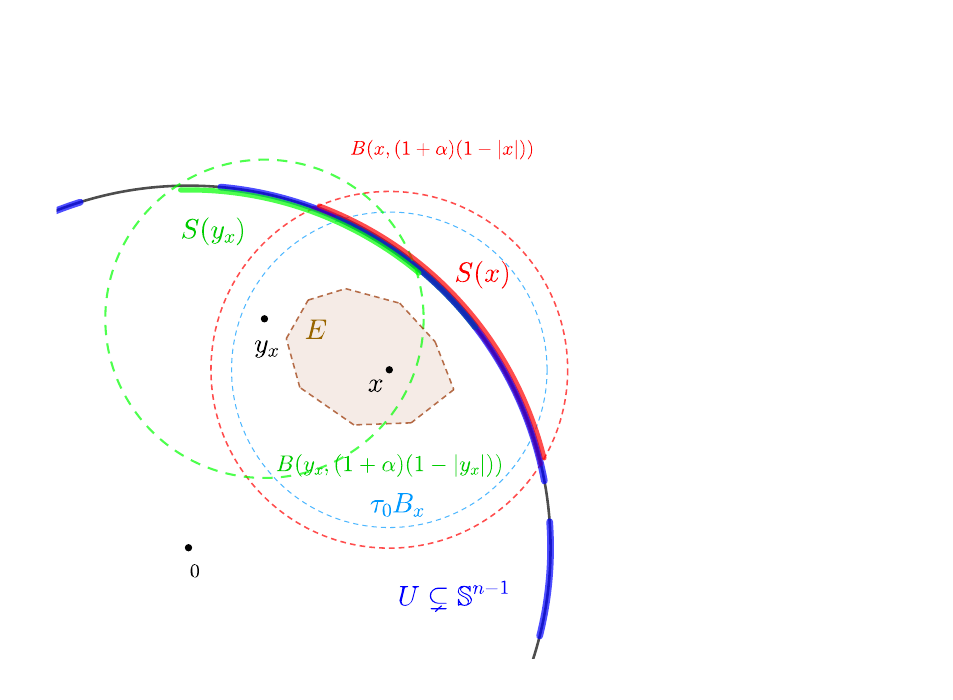}
  \caption{{\footnotesize {Points $y_x$ and set $E$} in Lemma~\ref{lem:leavingshadow} for $n=2$}.}
  \label{fig3-right}
\end{minipage}
\end{figure}	

For any $\lambda>0$ consider the following superlevel sets of the nontangential maximal function $\Nm f$ 
$$
 U(\lambda)=\{\om \in \Sn: \Nm f(\om)>\lambda \}.
$$
Given $\{ U_{i}\}_{i=1}^\infty$ a Whitney covering of $U(\lambda )$ given by Lemma~\ref{lem:WhitneyCovering} with $\eta \in (0,\frac12)$ and $\tau =2$. Denote by $\om_i$ the center of the hyperbolic ball $U_i$ and define the point $x_i$ as follows:
$$
 x_i:=\om_i\bigg(1-\inf\Big\{ |x-\om_i|: x\in \bigcap_{\om\in \frac{3}{2}U_i} \Gamma (\om)\Big\}\bigg).
$$
 See Figure~\ref{fig3}, in particular note that the set $\bigcap_{\om\in \frac{3}{2}U_i} \Gamma (\om)$ corresponds to the area bounded by green dotted curve and the part of $\Sn$. By the definition we have $U_i\subset S(x_i)\subset 2U_i$ for all $i\in\N$. In particular, $\{ S(x_i)\}_{i\in\N}$ is a covering of $U(\lambda )$ and sets $S(x_i)$ have uniformly bounded overlap; moreover, $d(S(x_i),\Sn\setminus U(\lambda ))\leq C(1-|x_i|)$.
	
	Fix $\gamma >0$ and suppose that $\om \in S(x_i)$ is such that $v(\om)<\gamma$. By Lemma~\ref{lem:leavingshadow}, applied for $x=x_i$ and $U:=U(\lambda)$, there exists $y_0=y_0(x_i)$ contained in the line segment joining $0$ and point $x_i$, denoted by $[0, x_i]$, with $S(y_0)\setminus U(\lambda )\neq\emptyset$; moreover, $y_0 \in\tau_0 B_{x_i}$, where $\tau_0>0$ is independent of $x_i$. Let further $\{y_1,\ldots, y_N\}$ be a collection of points in $[y_0, x_i]$, such that $y_N=x_i$, $|y_k|<|y_{k+1}|$ and $|y_{k}-y_{k-1}|\leq \frac{\eta}{4}(1-|x_i|)$ for all $k=0,\ldots, N-1$. Notice that it is possible to set $N=\left[\frac{4\tau_0}{\eta }\right] +1$ since $|x_i-y_0|\leq \tau_0 (1-|x_i|)$ and the collection of points is chosen within the segment $[y_0,x_i]$. We have, in particular, that $1-|y_k|< 1-|x_i|$, and thus $y_{k+1}\in \frac\eta 2B_{y_{k}}$,  for all $k=0,\ldots ,N-1$. Furthermore, we also obtain that $S(y_{k+1})\subset S(y_k)$ for all $k$ and then, if necessary, we may remove some of the first elements in $\{y_1,\ldots, y_N\}$ in order to obtain that $y_k\in U(\lambda)$ for all $k\geq 1$, while maintaining all other properties.
	Since $S(y_0) \setminus U(\lambda)\neq\emptyset$ there exists $\om_0\in S(y_0)\setminus U(\lambda )$ and so it holds that  
$$
 |f(y_0)|\leq \Nm f(\om_0)\leq \lambda
$$
 On the other hand, since $y_{k+1}\in \frac\eta 2 B_{y_{k}}$ Lemma \ref{lem:Morrey} implies that
$$
 |f(y_{k+1})-f(y_{k})|\leq C (1-|y_k|)^\beta \a_{f,\eta }(y_k)\leq C v(\om)\leq C \gamma, \quad k=0,\dots ,N-1.
$$
Here we also employ an observation that $S(x_i)\subset S(y_k)$, and so $y_k\in \Gamma (\omega )$ for all $\omega\in S(x_i)$.

By both of the above estimates we then have 
	\begin{align*}
		|f(x_i)|&\leq |f(y_0)|+\sum_{k=1}^{N_i}|f(y_k)-f(y_{k-1})|\leq \lambda +CN\gamma.
	\end{align*}
We are now in a position to show a good-$\lambda$ type estimate. Suppose now that $\om\in S(x_i)$ satisfies $|\tilde{f}(\om)|>2\lambda $ and $v(\om)\leq \gamma$, where $\lambda =(M+1)CN\gamma $. 
Then 
$$
 |\tilde{f}(\om)-f(x_i)|\geq |\tilde{f}(\om)|-|f(x_i)|\geq 2\lambda - (\lambda +CN\gamma )=MCN\gamma \geq M v(\om)\geq M (1-|x_i|)^\beta \mathrm a_{f,\eta }(x_i).
$$
This, together with Lemma \ref{lem:shadowmeasure} give
\begin{align*}
	\sigma (\{ \om \in S(x_i): |\tilde{f}(\om)|>2\lambda ,v(\om)\leq\gamma \})
		&\leq  \sigma (\{ \om\in S(x_i):|\tilde{f}(\om)-f(x_i)|\geq M (1-|x_i|)^\beta \mathrm a_{f,\eta }(x_i)  \})\\
		&\leq  C\sigma (S(x_i))(\log M)^{1-n}.
\end{align*}
Notice that $|\tilde{f}(\om)|>2\lambda$ implies that $|\Nm f(\om)|>2\lambda$ by the definition of the nontangential limit function $\tilde{f}$, and thus $\om\in U(\lambda )$. Then we have
\begin{align*}
\sigma (\{\om\in \Sn\!:\! |\tilde{f}(\om)|>2\lambda \})\leq &\, \sigma (\{ \om\in U(\lambda ): |\tilde{f}(\om)|>2\lambda,\, v(\om)\leq \gamma \}) +\sigma (\{ \om\in \Sn: v(\om) >\gamma  \}) \\
\leq\; &C\sum_{i\in\N}\sigma (S(x_i))(\log M)^{1-n} +\sigma (\{ \om\in \Sn: v(\om) >\gamma  \}) \\
\leq\; &C\sigma (U(\lambda ))(\log M)^{1-n} +\sigma (\{ \om\in \Sn: v(\om) >\gamma  \}) .
\end{align*}
Since $\gamma =\lambda /(CN(M+1))$, we compute the $L^p$-norm of $f$ by integrating over the superlevel sets and employing the above estimate:
\begin{align*}
\int_{\Sn}|\tilde{f}(\om)|^p \ud \om & = \int_{0}^{\infty} p \lambda^{p-1} \sigma (\{\om\in \Sn\!:\! |\tilde{f}(\om)|>\lambda \} ) \ud \lambda \\
& \leq C(\log (M))^{1-n}\int_{\Sn}|\Nm f(\om)|^p \ud \om + M^pCN^p\int_{\Sn}v(\om)^p \ud \om.
\end{align*}
	
By considering rescalings $f_t(x):=f(tx)$, for $0<t<1$ we may assume that the integrals $\int_{\Sn} |\tilde{f_t}|^p \ud\sigma$ and $\int_{\Sn}|\Nm f_t|^pd\sigma$ are both finite, in which case, by Theorem~\ref{thm11-ag2} (cf. the implication (2)$\Rightarrow$(3) in~\cite[Theorem 1.1]{ag2}) we have
$$
\int_{\Sn}|\Nm f_t(\om)|^p \ud\sigma(\om) \leq C\int_{\Sn}|\tilde{f_t}(\om)|^p \ud\sigma(\om).
$$
Finally, by taking $M$ sufficiently large and letting $t\to 1$, we arrive at the assertion
$$
\int_{\Sn}|\tilde{f}(\om)|^p\ud\sigma(\om) \leq C\int_{\Sn} v^p(\om) \ud\sigma(\om).
$$
\end{proof}

We now turn our attention to Theorem~\ref{thm:characterizationsnolder}, the main result of this section and a generalization of \cite[Theorem 5.1]{ak}. Both results relate the $L^p$ integrability of the non-tangential limit map to the integrability of the square function of a map and the non-tangential maximal functions. It is perhaps surprising, but it turns out that the strength of our result depends whether the multiplicity of the quasiregular mapping in the unit ball is bounded or not. In the first case, we obtain the equivalences, whereas the case of possibly unbounded multiplicity allows us to show only implications $(1) \Rightarrow (2) \Rightarrow(3)$ between the hypotheses of Theorem~\ref{thm:characterizationsnolder}.

Before proving Theorem~\ref{thm:characterizationsnolder} we present the proof of one of its consequences, see the Introduction for the statement of Corollary~\ref{cor-thm-nolder}, see also Corollary 5.4 in~\cite{ak} for the quasiconformal case. 

%
%
%
\begin{proof}[Proof of Corollary~\ref{cor-thm-nolder}]
 The equivalence of conditions (1) - (3) follows from Lemma 4.8 in~\cite{ag2} and Theorem~\ref{thm11-ag2}, whereas implications $(2) \Rightarrow (4) \Rightarrow (5)$ and the equivalence of all conditions (1) - (5) in the finite multiplicity case is the direct consequence of Theorem~\ref{thm:characterizationsnolder}. 
\end{proof}

\begin{proof}[Proof of Theorem~\ref{thm:characterizationsnolder}] 
\noindent $(1)\Rightarrow (2)$. Assume that $\tilde{f}\in L^p(\Sn)$ and consider two cases: (1) $0<p<1$ and (2) $p\geq 1$.
\smallskip

Suppose first $0<p<1$. Then, Corollary~\ref{cor:lem32} applied for $q=p$ and $u(x)=(1-|x|)^{p-1}$ we have that
\begin{equation*}
\int_{\B}\a_{f}^p(x)(1-|x|)^{p-1} \ud x\lesssim \int_{\B}|Df(x)|^p(1-|x|)^{p-1}\ud x.
\end{equation*}
Then, by Lemma 3.5 and Proposition 4.5 in~\cite{ag2} we obtain that
\begin{equation}\label{aux2-thm-nolder}
\int_{\B}|Df(x)|^p(1-|x|)^{p-1}\ud x =\int_{\B}|f(x)|^p\frac{|Df(x)|^p}{|f(x)|^p}(1-|x|)^{p-1}\ud x \leq \| \tilde{f}\|_{L^p(\Sn)}<\infty.
\end{equation}
Upon combining~\eqref{aux2-thm-nolder} with the previous estimate we arrive at assertion (2) for $0<p<1$.

Suppose now $p\geq 1$. Let $y\in\partial f(\B)$ such that $|y-f(0)|=d(f(0),\partial f(\B))$. Then by Corollary~\ref{cor:lem32} applied for~$q=1$ and $u(x)=(1-|x|)^{(p-1)(1+\frac{a}{n})}$ together with Lemma~\ref{lem:2.3AK} we find the following estimate
	\begin{align*}
	\int_{\B}\a_f^p(x)(1-|x|)^{p-1}\ud x &\lesssim\! \int_{\B}|Df(x)|\, \left[\a_{f}(x)(1-|x|)\right]^{p-1}\ud x \label{aux12: lem25AK} \\
	&\lesssim\!\! \int_{\B}|Df(x)|\, d(f(x),\partial f(\B ))^{p-1}\ud x \leq\! \int_{\B}|Df(x)|\, |f(x)-y|^{p-1}\ud x. \nonumber
	\end{align*}
  The latter inequality holds immediately, since point $y$ is further away from $f(x)$ than a point where the distance $d(f(x),\partial f(\B))$ is realized. Note that by Lemma 3.5 and Proposition 4.5 in~\cite{ag2} the latter integral is finite, as we apply these results to a quasiregular map $g:=f-y\not=0$ in $\B$:
	\[ 
	\int_{\B}|Df(x)|\, |f(x)-y|^{p-1} \ud x=\int_{\B}|f(x)-y|^{p} \frac{|Df(x)|}{|f(x)-y|}\, \ud x<\infty,
	\]
and so the assertion (2) follows for $p\geq 1$ as well.
\smallskip

\noindent $(2)\Rightarrow (3)$ 	Observe that by the very definition of the non-tangential region $\Gamma(\om):=\Gamma_{\alpha}(\om)$ with the apex at $\om\in \Sn$ and the aperture $\alpha$, it holds for points in $\Gamma(\om)\cap \B$ that
\[
 x\in \Gamma(\om) \Leftrightarrow \om\in B(x, (1+\alpha)(1-|x|))\cap \partial \Om.
\]
Therefore, by Fubini's theorem the following formula holds for any $u\in L^1(\B,\R)$
\begin{align}
\int_{\B}u(x)\ud x &= \int_{\B}u(x) (1-|x|)^{1-n} (1-|x|)^{n-1}\ud x \nonumber \\
& \approx_{n, \alpha} \int_{\B}u(x) (1-|x|)^{1-n} \left(\int_{\Sn}\chi_{B(x, (1+\alpha)(1-|x|))\cap \partial \Om} \ud \sigma(\om) \right)\ud x \nonumber \\
& = \int_{\B}u(x) (1-|x|)^{1-n} \left(\int_{\Sn}\chi_{\Gamma(\om)} \ud \sigma(\om) \right)\ud x\nonumber \\
& = \int_{\Sn} \left(\int_{\B} u(x) (1-|x|)^{1-n} \chi_{\Gamma(\om)}\ud x \right)\ud \sigma(\om)\nonumber \\
&\approx_{n,\alpha} \int_{\Sn}\int_{\Gamma (\omega )}u(x)(1-|x|)^{1-n}\, \ud x\,\ud\sigma (\omega). \label{aux3-thm37}
\end{align}
Define the following function for $\om\in \Sn$ 
\begin{equation}
v(\om):=\Big(\int_{\Gamma (\om)}\a_{f,\lambda_2}^p(x)(1-|x|)^{p-n} \ud x\Big)^{\frac1p} \label{aux4-thm37}
\end{equation}
where $\lambda_2\in (1,2)$ as in Lemma~\ref{lem:harnack}.
We apply~\eqref{aux3-thm37} with $u=v$ to obtain  
\begin{eqnarray*}
\Vert v\Vert_{L^p(\Sn)}^p&=&\int_{\Sn}v(\omega )^p \ud\sigma (\omega )=\int_{\Sn}\int_{\Gamma (w)} \a_{f,\lambda_2}^p(x)(1-|x|)^{p-n} \, \ud x\, \ud\sigma (\omega ) \nonumber \\
&=& \int_{\Sn}\int_{\Gamma (w)}\left[\a_{f, \lambda_2}^p(x)(1-|x|)^{(p-n)+(n-1)}\right]\, (1-|x|)^{1-n} \, \ud x\, \ud\sigma (\omega ) \nonumber\\
&\simeq &\int_{\B}\a^p_{f,\lambda_2}(x)\, (1-|x|)^{p-1}\ud x  \lesssim \int_{\B}\a^p_{f}(x)\, (1-|x|)^{p-1}\ud x, 
\end{eqnarray*}
where the last estimate comes from applying Theorem~\ref{thm:afintegrability} for $\lambda_1=1$ and $u(x)=(1-|x|)^{p-1}$. In consequence, by assertion (2), we have that $v\in L^p(\Sn)$.

Let $\lambda$ be as in Lemma~\ref{lem:harnack}. Then the following Harnack estimate holds for each $x\in \B$ and all $y\in\lambda B_x$ and $\lambda_2$ as above: $\a_f(x)\lesssim \a_{f,\lambda_2}(y)$. This, together with \eqref{aux4-thm37} yield for each $\omega\in\Sn$:

\begin{eqnarray}
\a_{f}(x)\, (1-|x|) &=&(1-|x|) \left(\vint_{\lambda B_x}\a_f^p(x)\ud y\right)^\frac{1}{p} \lesssim \frac{1-|x|}{|\lambda B_x|^\frac{1}{p}} \left(\int_{\lambda B_x}\a^p_{f,\lambda_2}(y)\ud y\right)^\frac{1}{p} \nonumber \\
&\lesssim & (1-|x|)^{1 -\frac{n}{p}} \left(\int_{\lambda B_x}\a^p_{f,\lambda_2}(y) \ud y\right)^\frac{1}{p} \lesssim
\left(\int_{\lambda B_x}\a^p_{f,\lambda_2}(y)(1-|y|)^{p-n} \ud y\right)^\frac{1}{p} \nonumber \\
&\lesssim & \left(\int_{\Gamma (\omega )}\a^p_{f,\lambda_2}(y) (1-|y|)^{p-n} \ud y\right)^\frac{1}{p} =v(\om), \label{aux8-lem37}
\end{eqnarray}
for all $x\in \Gamma_{\alpha'} (\omega )$ with $\alpha'<\alpha-\frac12$, so that $\lambda B_x\subset \Gamma(\om)$. In particular,
$$
 \Nm(\a_{f}(x)\, (1-|x|)^\gamma  )(\omega )\lesssim v(\omega ).
$$
We then conclude by integrating over $\Sn$ and applying that $v\in L^p(\Sn )$.

%
\smallskip
	 
\noindent Let us now assume that the multiplicity of $f$ in ball $\B$ is finite, i.e., ${N(f,\B)<\infty}$. Then, the implication  $(3)\Rightarrow (1)$ follows directly from Proposition~\ref{prop:Lp-norm-bound} by letting $\eta=\beta:=1$.

Observe that, under the assumption ${N(f,\B)<\infty}$, the proof of the implication $(2)\!\Rightarrow\!(3)$ simplifies, as instead of applying Lemma~\ref{lem:harnack} we may appeal to Corollary~\ref{est-harnack-N} and, in consequence, define the function $v$ with $\lambda_2=1$, thus, simplifying the estimate~\eqref{aux8-lem37}. This together with the previous discussion on the implication $(3)\Rightarrow (1)$  completes the proof of Theorem~\ref{thm:characterizationsnolder}.  
\end{proof}

\subsection{The Riesz theorem}

The classical Riesz conjugate function theorem for an analytic function in the plane asserts it belongs to the Hardy space $\Hp$ for ${1<p<\infty}$ if and only if its real part belongs to $\mathcal{H}^p$. The Riesz theorem fails for quasiconformal mappings, see~\cite[Section 6]{ak}. However, it has counterparts for the harmonic quasiconformal mappings in the plane, see Theorems 1.1-1.2 in~\cite{lz}, and for invariant harmonic quasiregular mappings in $\B$, see Theorem 1.3 in~\cite{lz}. Nevertheless, a theorem by Burkholder--Gundy--Silverstein~\cite{bgs} shows that the $L^p(\mathbb{S}^1)$-norm of a harmonic conjugate function $v$ of a harmonic one $u$ is bounded by the $L^p(\mathbb{S}^1)$-norm of $\Nm(u)$, up to a constant depending only on $p>0$. Similar result holds for quasiconformal mappings in $\B$, see Theorem 6.1 in~\cite{ak}. The following observation is a counterpart of this result and follows from Theorem~\ref{thm:characterizationsnolder}.

\begin{cor}\label{cor-Riesz} Let $f=(f^1,\ldots,f^n):\B\rightarrow \R^n\backslash\{ 0\}$ a $K$-quasiregular mapping satisfying \eqref{cond-m} and \eqref{est-min}; moreover, let the multiplicity of $f$ in ball $\B$ be finite $N(f,\B)<\infty$. Then $f\in \mathcal{H}^p$ if and only if $\Nm(f^i)\in L^p(\Sn)$ for one of the component functions of $f$ and $i=1,\ldots, n$.
\end{cor}
\begin{proof}
 If $f\in \mathcal{H}^p$ , then $\Nm(f)\in L^p(\Sn)$ by Theorem~\ref{thm11-ag2} and, hence, also $\Nm(f^i)\in L^p(\Sn)$ for all $i=1,\ldots, n$.
 
 For the opposite implication, let $\Nm(f^i)\in L^p(\Sn)$ for some $i=1,\ldots, n$. Since $|f^i(x)|\leq \Nm(f^i)(\omega)$ for all $x\in\Gamma (\omega)$, it holds for $t>0$ that
\[
\vint_{tB_x} |f^i(y)|^n\ud y\leq \Nm(f^i)(\omega )^n,
\]
for any ball $tB_x\subset \Gamma_{\alpha}(\omega)$ with $\om:=\frac{x}{|x|}$, $\alpha>2$ and all $1<t\leq \frac{2\alpha}{2+\alpha}$. The latter follows immediately from the definition of the non-tangential region, since if $y\in tB_x$, then
\[
 |y-\om|\leq |y-x|+|x-\om|\leq (1+\frac{t}{2})(1-|x|)\leq \frac{2+t}{2-t}(1-|y|)\leq (1+\alpha)(1-|y|).
\]
On the other hand, the Caccioppoli inequality for component functions of a quasiregular map, see~\cite[Proposition 2.6]{no}, gives that 
 for almost all $x\in \B$
\begin{equation}\label{cor-Riesz-aux}
\int_{B_x} |Df|^n\leq C(n, K) \left(\frac{t}{t-1}\right)^n \vint_{t B_x} |f^i|^n \lesssim_{n, K, t} \Nm(f^i)(\omega )^n,
\end{equation}

The estimate~\eqref{cor-Riesz-aux} equivalently reads 
\[
\a_f(x)(1-|x|)\leq C(n, K, t) \Nm(f^i)(\omega ),
\]
 which in turn gives that $\Nm(\a_f(x)(1-|x|))\in L^p(\Sn)$ for all $p>0$. Therefore, by Theorem~\ref{thm:characterizationsnolder} and the assumption $N(f,\B)<\infty$ we have that assertion (3) implies (1), meaning that $\tilde{f}\in L^p(\Sn)$ and hence $f\in \mathcal{H}^p$ by Theorem~\ref{thm11-ag2}.
\end{proof}

\section{Logarithm of a quasiregular map and ${\rm BMO}(\Sn)$}

The purpose of this section is to prove a counterpart of Theorem 7.1 in~\cite{ak} asserting that the logarithm of the extension $|\tilde{f}|$ belongs to the BMO space for a quasiconformal map $f$, cf. Theorem 4 in~\cite{j}. The proof can be derived from the discussion in the proof of Corollary 4.3 in~\cite{ag2}. However, here we present more details.

\begin{theorem}\label{thm:log-bmo}
Let $n \geq 2$ and $f: \B\to \Rn \setminus \{0\}$ be a $K$-quasiregular mapping in class~\eqref{est-min} satisfying the  multiplicity condition~\eqref{cond-m} with some $0\leq a<n-1$. Then, $\ln |\tilde{f}|\in BMO(\Sn)$.
\end{theorem}
Note that, a priori, in order to ensure the existence of the non-tangential limit map $\tilde{f}$ we would need also to assume the growth condition~\eqref{cond-g}, see Theorem 4.1 in~\cite{kmv}. However, since we impose the Miniowitz condition~\eqref{est-min}, its upper estimate immediately results in the growth condition with the constant $C|f(0)|2^{m}$ and the growth exponent $m$.
\begin{proof}
 First, we show the following estimates for the shadows $S(x)$ of points $x\in \B$.
 
 Let $u:=\ln |f|$ and $F(\om):=\int_{0}^{1}|\nabla u(t\om)|\,t^{n-1}\ud t$ for $\om\in \Sn$. Note that Lemma 3.5 in~\cite{ag2} implies the following
\begin{equation*}
 \|F\|_{L^1(\Sn)}=\int_{\Sn} \int_{0}^{1}|\nabla u(t\om)|\,t^{n-1}\ud t \ud \om \leq \int_{\B} \frac{|Df|}{|f|} \ud x<c(a, m, n, K). 
\end{equation*}

\noindent {\sc Claim:} \emph{There exists a constant $C$ independent of $f$ such that the following estimates hold for all $z\in \B$
\begin{equation}\label{eq: shadow-est1}
 \sigma\left(\left\{\om \in S(z): |\tilde{f}(\om)|\geq N |f(z)| \right\}\right)\leq \frac{7^n C \|F\|_{L^1(\Sn)}}{\ln N} \sigma(S(z)).
\end{equation}
\begin{equation}\label{eq: shadow-est2}
 \sigma\left(\left\{\om \in S(z): \big|\ln |\tilde{f}(\om)|-\ln|f(z)|\big|\geq \ln N \right\}\right)\leq \frac{7^n C \|F\|_{L^1(\Sn)}}{\ln N} \sigma(S(z)).
\end{equation}
}

\emph{Proof of the claim:} as in the proofs of Claims 1 and 2 in~\cite[Proposition 3.8]{ag2} we first show the following auxiliary estimate:
\begin{equation}\label{eq: shadow-est3}
 \sigma\left(\left\{\om \in \Sn: |\tilde{f}(\om)|\geq N |f(0)| \right\}\right)\leq \frac{7^n C \|F\|_{L^1(\Sn)}}{\ln N}.
\end{equation}
Denote the set of points in the assertion of the claim by $F_N\subset \Sn$. The Harnack inequality for $|f|$, see Corollary 3.4 in~\cite{ag2}, applied for a ball $B(0,|\frac17 \om)|$ for $\om\in F_N$ gives that 
\begin{equation}\label{est-aux-f}
 \frac{|\tilde{f}(\om)|}{|f(\frac17\om)|}= \frac{|\tilde{f}(\om)|}{|f(0)|}\, \frac{|f(0)|}{|f(\frac17\om)|}\geq \frac{N}{C_H}.
\end{equation}
Thus, we get
\[
 F(\om)\geq 7^{1-n}\int_{\frac17}^{1} |\nabla u(t\om)|\ud t\geq 7^{1-n}\left|\ln \frac{|\tilde{f}(\om)|}{|f(\frac17 \om)|}\right|
 \geq 7^{1-n}\left|\ln \frac{N}{C_H}\right|\geq \frac{\ln N}{7^n}.
\]
 The estimate~\eqref{eq: shadow-est3} now follows by direct integration:
 \begin{align*}
 \sigma (F_N)=\int_{F_N} 1\,\ud \sigma \leq \int_{F_n} \frac{7^n F(\om)}{\ln N} &\leq \frac{7^n}{\ln N} \int_{F_N} \int_{0}^{1}|\nabla u(t\om)|\,t^{n-1}\ud t \\
& \leq \frac{7^n}{\ln N} \int_{\B} \frac{|Df|}{|f|} \ud x.
 \end{align*}
Let us comment that the proof requires the Harnack estimate on a ball which is away from the boundary of $\B$ and for this, $f$ need not be in the Miniowitz class (a PDE argument or the Miniowitz estimate would suffice to show the Harnack estimate on a ball $B\subset 2B\Subset \B$). Moreover, in order to conclude that $\|F\|_{\Sn}<\infty$ we only need the Miniowitz estimate and condition~\eqref{cond-m}, see the first part of the proof of the Jones Lemma 3.5 in~\cite{ag2}.

\smallskip

 The proof of~\eqref{eq: shadow-est1} goes via the reduction to the claim~\eqref{eq: shadow-est3} and is similar to the proof of Claim 2 in~\cite{ag2}. Therefore, we reduce our presentation only to the key steps.
 
 Set $g:=f\circ T_z^{-1}(x)$ and recall that $g$ satisfies the multiplicity condition~\eqref{cond-m} on hyperbolic balls centered at the origin. Since by property (4) in~\cite{z} of map $T_z$  we have that for all $z\in B$ that $B(0,\frac17)\subset T_z\left( B(z, \frac14 (1-|z|))\right)$, and so
\[
g( B(0,\frac17) )\subset f \left( B(z, \frac14 (1-|z|))\right).
\]
By analogy to the proof of~\eqref{eq: shadow-est1}, we set $u:=\ln |g|$ and define the function $F$ corresponding to $u$, for the sake of simplicity of the presentation denoted again by $F$. Since $g(0)=f\circ T_z^{-1}(0)=f(z)$, by the Harnack inequality we get~\eqref{est-aux-f} for $g$, resulting in the estimate
\begin{equation}\label{est-gf-aux}
\sigma\left(\left\{\om\in \Sn: |\tilde{g}(\om)|\geq N |f(z)| \right\}\right)\leq \frac{7^n \|F\|_{L^1(\Sn)}}{\ln N},
\end{equation}
For the norm $\|F\|_{L^1(\Sn)}<\infty$ we observe that  $\|F\|_{L^1(\Sn)}\leq \int_{\B}\frac{|Dg|}{|g|}<\infty$. It is here that the Miniowitz class property of $f$ is employed in order to have the estimate for the latter integral independent of the choice of $z$ in $T_z^{-1}$, see the discussion in the second part of the proof of Lemma 3.5 in~\cite{ag2}.

Similarly to~\cite{ag2}, we now have by~\eqref{est-gf-aux} that
\[
\sigma\left(T_z\left(\left\{y \in S(z): |\tilde{f}(y)|\geq N |f(z)| \right\}\right)\right)\leq \frac{7^n \|F\|_{L^1(\Sn)}}{\ln N}.
\]
Finally, recall the following estimate (3) in~\cite{z} for $x,y\in S(z)$, the shadow associated with $z$:
\[
\frac{1}{9(1-|z|)} |x-y|\leq |T_z(x)-T_z(y)|\leq \frac{2}{1-|z|} |x-y|.
\]
 This, together with the observation that $\sigma(S(z)) \approx (1-|z|)^{n-1}$ gives the first assertion~\eqref{eq: shadow-est1} of the Claim. Then, the estimate~\eqref{eq: shadow-est2} is a direct consequence of~\eqref{eq: shadow-est1}.

The estimate~\eqref{eq: shadow-est2} allows us to make the last step. However, our estimate differs from the corresponding ones in Lemma 4.2 and Remark following that lemma in~\cite{ak} in terms of the power of $\ln N$, which in \cite{ak} appears as $(\ln N)^{n-1}$, and thus we cannot apply the Cavalieri principle to represent the integral of $\big|\ln|f(x)|-\ln|f(z)|\big|$ in terms of the superlevel sets as in~\eqref{eq: shadow-est2}. Instead, we employ a variant of the BMO spaces description provided by Lemma 3.6 in~\cite{kkms} for the metric measure spaces, see also similar Euclidean results in Str\"omberg~\cite{str}, Ex. 4 in~\cite[pg.  261, Ch VI]{g} for the formulation in $\R$ and Appendix in~\cite{agg} for the $\Rn$ case. The metric spaces result adapted to our setting reads:

\emph{Let $h:\Sn \to \R$ be a measurable function and suppose that there exist a constant $\gamma\in(0,(4c_{\sigma}^3)^{-1})$ for $c_{\sigma}$ the doubling constant of the surface measure $\sigma$ on $\Sn$ and $\lambda>0$, such that for any ball $B\subset \Sn$ and any $c\in \R$ it holds
\begin{equation}\label{aux2-thm:log-bmo}
 \sigma\left(\left\{x\in B: |h(x)-c|>\lambda \right\}\right)<\gamma\sigma(B).
\end{equation}
Then, $h \in BMO(\Sn)$ with the BMO-(semi)norm estimate $\|h\|_{{\rm BMO}}\lesssim_{c_{\sigma}} \lambda$.
}

We apply the claim with 
\[
 h:=\ln \big|{\tilde f}|_{_{\Sn}}\big|\quad \hbox{and }\lambda:=\ln N, \quad \gamma:= \frac{7^n C \|F\|_{L^1(\Sn)}}{\ln N},
\]
and so the constant $\gamma$ equals the one in the claim above, see~\eqref{eq: shadow-est1} and~\eqref{eq: shadow-est2}. Moreover, we set constants $c$ in~\eqref{aux2-thm:log-bmo} as follows $c:=\ln |f(z)|$, for $z\in \B$. Then the assumption~\eqref{aux2-thm:log-bmo} holds by~\eqref{eq: shadow-est2}, since shadows $S(z)=B(z, (1+\alpha)(1-|z|)) \cap \Sn$ are defined as intersections of balls centered at points in $\B$ with $\Sn$ and, hence, define also balls in $\Sn$.
\end{proof}

\section{Quasiregular maps and BMO spaces} 

The main goal of this section is to prove Theorem~\ref{thm72}, the non-injective counterpart of Theorem 7.2 in~\cite{ak}, characterizing quasiconformal maps in the BMO spaces via the BMO condition on their non-tangential limit maps, the Carleson measures and the growth of the averaged derivative $a_f$ and its non-tangential function. 

We begin with a counterpart of Lemma 9.4 in~\cite{ak} originally proven for the quasiconformal mappings.

\begin{lem}\label{lem94}
 Let $n \geq 2$ and $f:\B\to \Rn \setminus \{0\}$ be a $K$-quasiregular mapping  satisfying the multiplicity condition~\eqref{precond-m} with some $0\leq a<n-1$. Moreover, let us assume that $f$ satisfies the growth condition~\eqref{cond-g} with exponent $\beta<n-1$. Then,
 \[
  \int_{\B}|Df(x)|^{n}(1-|x|)^{n-1+a}\,\ud x\leq c<\infty,
 \]
 where $c$ depends only on $n, K,a$ and $\beta$. 
 
 Furthermore, if instead of condition~\eqref{cond-g} we assume that $f$ is in the $n$-Hardy space $\|f\|_{\mathcal{H}_n}<\infty$, then the constant $c$ equals $c= C(n, K, a)\|f\|_{\mathcal{H}_n}$.
\end{lem}
%

\begin{proof}
 Let us represent the unit ball $\B$ as the union of the following ring domains 
 $$
  R_j:=\{x\in \B:1-2^{-j}\leq |x| \leq 1-2^{-(j+1)}\}\,\hbox{ for } j=0,1,\ldots.
 $$
 Then, by the distortion inequality and the quasiregular change of variables it holds that
 \begin{align}
   &\int_{\B}|Df(x)|^{n}(1-|x|)^{n-1+a}\,\ud x \nonumber \\
   &\phantom{AAAA} \leq K \sum_{j=0}^{\infty} \int_{R_j}J_f(x)(1-|x|)^{n-1+a}\,\ud x \nonumber \\ 
   &\phantom{AAAA} \lesssim_{n, K, a} \sum_{j=0}^{\infty} 2^{-j(n-1+a)}N(f, B(0, 1-2^{-j})) |f(R_j)|. 
   \label{eq3-lem94}
 \end{align}
 Next we observe that if $f$ satisfies the growth condition~\eqref{cond-g} with some $\beta>0$, then we have that 
\begin{equation} \label{eq4-lem94}
|f(R_j)|\leq |f(B(0, 1-2^{-j})|\leq C2^{j\beta}.
\end{equation}
%
This, together with the multiplicity condition~\eqref{precond-m} leads to the following estimate:
\[
\int_{\B}|Df(x)|^{n}(1-|x|)^{n-1+a}\,\ud x \lesssim_{K, a, n} \sum_{j=0}^{\infty} 2^{-j(n-1+a)}2^{ja} C2^{j\beta}\lesssim_{K, a, n} C \sum_{j=0}^{\infty} 2^{-j(n-1-\beta)}.
\]
The series converges, provided that $\beta<n-1$ and the lemma is proven with $c=c(n,K,a, \beta)$. 

Suppose now that instead of~\eqref{cond-g}, we assume that $f\in \mathcal{H}_n$. Then, Observation 3.1 in~\cite{ag2}, holding without assuming~\eqref{cond-m} and~\eqref{cond-g}, shows that a quasiregular map with $\|f\|_{\mathcal{H}_n}<\infty$ satisfies the growth condition with $\beta=\frac{n-1}{n}$ and, moreover, the constant $C=C(n, K) \|f\|_{\mathcal{H}_n}$, see the estimate (5) in~\cite{ag2}. This completes the proof of Lemma~\ref{lem94}.
\end{proof}

We remark that the Miniowitz estimate~\eqref{est-min} implies the growth condition with exponent $\beta=m$ which, a priori, could be applied in~\eqref{eq4-lem94}. However, $m=2^{n-1}K_I$, where $K_I$ is the inner distortion of map $f$, see~\cite[Theorem 3]{M1}. Hence, the value of $m$ given by~\eqref{est-min} turns out to be too large to ensure convergence of the integral~\eqref{eq3-lem94} and either we assume that $\beta<n-1$ or that $\|f\|_{\mathcal{H}_n}<\infty$.  However, in the latter case, Observation 3.1 in~\cite{ag2} results in $\beta=\frac{n-1}{n}<n-1$.

Let us note that the following, slightly stronger, version of Lemma~\ref{lem94} can be obtained if we allow the multiplicity exponent $a$ to depend also on the Gehring exponent. However, in practice it is difficult to estimate how large this exponent can be. 

\begin{cor}\label{cor94}
 Let $n \geq 3$ and $f: \B\to \Rn \setminus \{0\}$ be a $K$-quasiregular mapping in class~\eqref{est-min} satisfying the multiplicity condition~\eqref{cond-m} with some $0\leq a<\frac{p_0-n}{n}$, where $p_0=p_0(n, K)>n$ is an exponent in the reverse H\"older inequality additionally satisfying $p_0<n^2$. Moreover, let us assume that $f$ is the $n$-Hardy space $\|f\|_{\mathcal{H}_n}<\infty$.
  Then,
 \[
  \int_{\B}|Df(x)|^{n+a}(1-|x|)^{n-1+a}\,\ud x\leq C \|f\|_{\mathcal{H}^n}^n,
 \]
 where $C$ depends only on $n, K$ and $a, m$. 
\end{cor}

Since in what follows we will not appeal to the corollary, we only sketch its proof. By the H\"older inequality and Lemma~\ref{lem94} one gets that
\begin{align*}
  \int_{\B}|Df(x)|^{n+a}(1-|x|)^{n-1+a}\,\ud x 
 \leq  C \|f\|_{\mathcal{H}_n}^{n-1} \left(\int_{\B}|Df(x)|^{n(a+1)}\,(1-|x|)^{n-1+a}\,\ud x \right)^{\frac1n}.
 \end{align*}
Then, in order to estimate the integral on the right-hand side, one employs the technique similar to the one in the proof of Lemma~\ref{lem94} and the Jones Lemma 3.5 in~\cite{ag2}, see also the proof of Theorem 4.1 in~\cite{kmv}. The proof follows by applying the reverse H\"older inequality, \cite[Theorem 5.1]{bi}.

Next we show a counterpart of Lemma 7.5 in~\cite{ak}. In order to prove it we need an $n$-dimensional analog of Lemma 3.3 in~\cite[Ch VI]{g}, see also Theorem 1.2 in~\cite{adgr} for a counterpart of this lemma in the setting of the Heisenberg group $\mathbb{H}_1$. Lemma~\ref{lem-gar33} below characterizes the Carleson measures on the unit ball $\B$ in terms of the M\"obius transformations on $\B$. In order to motivate and explain this result let us present it in the particular case of the unit disk $\mathbb{D}$. Then, the lemma says that \emph{a positive measure $\mu$ on $\mathbb{D}$ is a Carleson measure if and only if the following holds:
\begin{equation}\label{cond:plane}
 \sup_{z_0\in \mathbb{D}} \int_{\mathbb{D}} \frac{1-|z_0|^2}{|1-\bar{z_0}z|^2} \ud \mu(z)=M<\infty.
\end{equation}
Moreover, the constant $M$ is comparable to the Carleson constant, i.e. $M\approx \gamma_{\mu}$ with absolute constants.
}
Observe that for a given $z_0\in \mathbb{D}$, the integrand in~\eqref{cond:plane} satisfies the following
\begin{equation}\label{eq: mob-plane}
\frac{1-|z_0|^2}{|1-\bar{z_0}z|^2}=\frac{1-\left|\frac{z-z_0}{1-\bar{z_0}z}\right|^2}{1-|z|^2}=\frac{1-|T_{z_0}(z)|^2}{1-|z|^2}=|D T_{z_0}(z)|,
\end{equation}
see Formulas (33) and (34) in~\cite[Ch II]{ahl}. Here $T_{z_0}(z)=e^{-i\theta_0}\frac{z-z_0}{1-\bar{z_0}z}$, for $z_0=r_0 e^{i\theta_0}$ is the M\"obius self-mapping of $\mathbb{D}$ corresponding to maps $T$ discussed in this work.  The relation between~\eqref{eq: mob-plane} and the Carleson condition becomes clear, once we recall that for small enough radii $r>0$, any $\om\in \partial \mathbb{D}$ and $z\in \mathbb{D} \cap B(\om,r)$ it holds that 
\begin{equation}\label{eq:mob-plane2}
\frac{1-|T_{z_0}(z)|^2}{1-|z|^2} \approx \frac{1-|T_{z_0}(z)|}{1-|z|}\approx \frac{1}{r},
\end{equation}
see Lemma 2.2 in~\cite{ag2}. Hence, \eqref{eq: mob-plane} and \eqref{eq:mob-plane2} together with \eqref{cond:plane} imply the Carleson condition for $\mu$:
$$
\mu(\mathbb{D} \cap B(\om,r)) = r \int_{\mathbb{D} \cap B(\om,r)} \frac{1}{r} \ud \mu\lesssim Mr.
$$

\begin{lem}\label{lem-gar33} A measure $\mu$ on the unit ball $\B\subset \Rn$ is a Carleson measure if and only if
 \begin{equation*}
  \sup_{x_0\in \B} \int_{\B} |D T_{x_0}(y)|^{n-1} \ud \mu(y)=M<\infty,
 \end{equation*}
\end{lem}

The proof of the lemma is a simple adaptation of the proof of its planar counterpart, Lemma 3.3, presented on pg. 232 in~\cite[Ch VI]{g}. Since the geometry of sectors in the plane, employed in the proof in~\cite{g}, and in $\Rn$ is the same, the proof holds almost verbatim upon natural modifications and, therefore, we omit it.

The following result is a counterpart of Lemma 7.5 in~\cite{ak} and we retrieve its assertion for $a=0$.

\begin{lem}\label{lem75}
 Let $n \geq 2$ and $f:\B\to \Rn \setminus \{0\}$ be a $K$-quasiregular mapping in class~\eqref{est-min} satisfying the multiplicity condition~\eqref{cond-m} with some $0\leq a<n-1$. If $\tilde{f} \in BMO(\Sn)$, then the following expression defines the Carleson measure on $\Sn$:
 \[
  \ud \mu= |Df(x)|^{n}(1-|x|)^{n-1+a}\,\ud x.
 \]
\end{lem}

\begin{proof} We begin with the discussion of relations between the non-tangential limit map $\tilde{f}\in BMO(\Sn)$ and maps $g:=f\circ T_{x_0}$. One way to pursue such relations is through the Poisson extension and the equivalent BMO-norms, see e.g. proofs of Lemmas 7.5 and 7.6 in~\cite{ak}. However, we refer instead to the discussion on the BMO on metric measure spaces in~\cite{kkms}, see~\cite{go} for the corresponding discussion in $\Rn$ and cf. also Corollary 1.4 in~\cite[Ch VI]{g} for the $\mathbb{S}^1$ setting. First note that by the Carath\'eodory extension theorem any automorphism $T_{x_0}$ extends to a homeomorphism of the unit spheres (again denoted by $T_{x_0}$ for the simplicity of presentation). Moreover, by Formula (33) in~\cite[Chapter 2]{ahl} the formula defining $T_{x_0}$ in $\B$ remains true also for points in $\Sn$. Since 
$$
|x_0|~\Big|y-\frac{x_0}{|x_0|^2}\Big|=\Big|\frac{x_0 |x_0|}{|x_0|^2}-y|x_0|\Big|\geq 1-|y||x_0|= 1-|x_0|
$$
we have that by Formula (30) in~\cite[Chapter 2]{ahl} for all $x_0\in \B$ it holds 
\[
 |DT_{x_0}(y)|=\frac{1-|x_0|^2}{|x_0|^2|y-\frac{x_0}{|x_0|}|^2}=\frac{1-|x_0|^2}{(1-|x_0|)^2}\leq \frac{2}{1-|x_0|}.
\]
Hence, we get that for all $y,z \in \Sn$ it holds 
\[
 d_{\Sn}(T_{x_0}(y),T_{x_0}(z))\approx \frac{1}{1-|x_0|} d_{\Sn}(y,z),
\]
see also the discussion for $T_{x_0}$ in~\cite[Section 2]{ak}. Hence, for any ball $B\subset \Sn$ and any measurable subset $E\subset \Sn$ we get that 
\[
\sigma(T_{x_0}^{-1}(E)\cap B)\approx_{n} (1-|x_0|) \sigma(E).
\]
This observation, together with the compactness of $\Sn$ imply that the condition (ii) in~\cite[Theorem 3.1]{kkms} is satisfied with $\gamma:=(1-|x_0|)\lambda<\lambda$, for $\lambda>0$. Therefore, by \cite[Proposition 3.8]{kkms} the operators $C_{x_0}:=\tilde{f}\circ T_{x_0}:\Sn\to \Sn$ have their BMO (semi)norms bounded $\|C_{x_0}\|\leq C(n)$; in particular norms are independent of points $x_0$. 
Furthermore, recall that by compactness of $\Sn$, the $BMO(\Sn)\subset L^p(\Sn)$ for any $p>0$ and so, in particular, for $p=n$. Since, for any $x_0\in \B$, the map $g:=f\circ T_{x_0}$ is quasiregular, belongs to the Miniowitz class \eqref{est-min} and satisfies~\eqref{cond-m}, we may now apply Theorem~\ref{thm11-ag2} to obtain that  $\|g\|_{\mathcal{H}^n}<C$. Therefore, by Lemma~\ref{lem94} we get that
\begin{equation}\label{eq1-lem75}
\int_{\B} |D(f\circ T)(x)|^{n}(1-|x|)^{n-1+a}\,\ud x \leq C.
\end{equation}

We estimate the following integral, whose finiteness results in the assertion, see details below. By the H\"older inequality we have that
\begin{align}
&\int_{\B} |Df(x)|^{n}\,(1-|x|)^{n-1+a}\,|DT^{-1}(x)|^{n-1}\,\ud x \label{eq3-lem75}\\
&= \int_{\B} \Big(|Df(x)|^{n\frac{n-1}{n-1+a}}\,(1-|x|)^{n-1}\,|DT^{-1}(x)|^{n-1}\Big)\Big(|Df(x)|^{n\frac{a}{n-1+a}}\,(1-|x|)^{a}\Big)\,\ud x \nonumber \\
&\leq \left(\int_{\B} |Df(x)|^{n}\,(1-|x|)^{n-1+a}\,|DT^{-1}(x)|^{n-1+a}\,\ud x\right)^{\frac{n-1}{n-1+a}}
\left(\int_{\B} |Df(x)|^{n}\,(1-|x|)^{n-1+a}\,\ud x\right)^{\frac{a}{n-1+a}}. \nonumber
\end{align}
Lemma~\ref{lem94} gives us that the second integral on the right-hand side above is finite.  In order to estimate the first  integral, we proceed as follows. First, we recall a property of the M\"obius automorphisms $T_{x_0}$ of $\B$ that
\[
 1-|T_{x_0}^{-1}(y)|\approx_{n, |x_0|} |DT_{x_0}^{-1}| (1-|y|),\quad y\in \B.  
\]
This formula follows from (14) and (17) in~\cite[Ch 2]{ahl}, see also Appendix in~\cite{ag2}.

By using this formula together with the change of variables for a M\"obius transformation $T:=T_{x_0}$ and by the fact that each $T$ satisfies $|DT|^n=J_T$ a.e. in $\B$, we observe that
\begin{align}
& \int_{\B} |Df(x)|^{n}\,(1-|x|)^{n-1+a}\,|DT^{-1}(x)|^{n-1+a}\,\ud x \nonumber \\
&\phantom{AAA} \approx \int_{\B} |Df(x)|^{n}\, (1-|T^{-1}(x)|)^{n-1+a}\,\ud x  \nonumber \\
&\phantom{AAA} =\int_{\B} |Df(x)\,DT(T^{-1}(x))\,\big(DT(T^{-1}(x))\big)^{-1}|^{n}\, \,(1-|T^{-1}(x)|)^{n-1+a}\,\ud x \nonumber \\
&\phantom{AAA} \leq  \int_{\B} |Df(x)\,DT(T^{-1}(x))|^{n}\,|DT(T^{-1}(x))|^{-n}\,(1-|T^{-1}(x)|)^{n-1+a}\,\ud x \nonumber \\
&\phantom{AAA} = \int_{\B} |Df(x)\,DT(T^{-1}(x))|^{n}\,J_{T^{-1}}(x)\,(1-|T^{-1}(x)|)^{n-1+a}\,\ud x \nonumber \\
&\phantom{AAA} = \int_{\B} |Df(T(y)) DT(y)|^{n}\,(1-|y|)^{n-1+a}\,J_T(y) J_{T^{-1}}(T(y)) \,\ud y \nonumber \\
&\phantom{AAA} =\int_{\B} |D(f\circ T)(y)|^{n}(1-|y|)^{n-1+a}\,\ud y \leq C, \label{eq2-lem75}
\end{align} 
where the last estimate holds by~\eqref{eq1-lem75}.  Finally, upon combining~\eqref{eq2-lem75} with~\eqref{eq1-lem75} we arrive at the following estimate, cf.~\eqref{eq3-lem75}:
\[
 \sup_{T_{x_0}} \int_{\B} |Df(y)|^{n}\,(1-|y|)^{n-1+a}\,|DT^{-1}|^{n-1}\,\ud x \leq C,
\]
which upon applying Lemma~\ref{lem-gar33} results in the assertion. (Since the supremum in this lemma is taken over all $x_0\in \B$ and $T_{x_0}^{-1}=T_{-x_0}$, we may without the loss of generality consider as well $DT^{-1}$ in the last estimate above, instead of $DT$ as in the assertion of Lemma~\ref{lem-gar33}.)
\end{proof}

We are in a position to show the main result of this section, a counterpart of Theorem 7.2 in~\cite{ak} for quasiregular mappings. It turns out that to ensure a chain of implications, the standard assumptions imposed on a map throughout this work are enough, see Introduction for the statement of Theorem~\ref{thm72}. However, in order to obtain the characterization result, the additional finite multiplicity assumption is needed. 
\begin{proof}[Proof of Theorem~\ref{thm72}]
First, note that since map $f$ satisfies the Miniowitz estimate, the growth condition~\eqref{cond-g} holds in particular. This together with the assumptions of Theorem~\ref{thm72} allow us to apply Theorem 3.2 in~\cite{ag2} and conclude that the non-tangential limit function $\tilde{f}$ exists a.e. in $\Sn$.

\smallskip

\noindent $(1)\Leftrightarrow (2)$. For $n>1$ one repeats the discussion for the unit circle $\mathbb{S}^1$ presented in~\cite[Chaper VI.1]{g}, upon necessary modifications. In particular, by the proof of Corollary 1.4 in~\cite[Ch VI]{g} the reasoning consists of the following steps: first, we consider an equivalent (semi)norm in $BMO(\Sn)$ defined by the formula 
$$
\|\tilde{f}\|':=\sup_{x_0\in \B} \vint_{\Sn} \bigg|\tilde{f}(z)-\vint_{\Sn} \tilde{f}(\om) P_{z}(\om) \ud\om\bigg|P_{x_0}(z)\ud z,
$$
where $P_{a}(b)$ stands for the Poisson kernel on $\B$. Then, we show its invariance with respect to automorphisms $T_{x_0}$ via the change of variables formula. Thus, $\|\tilde{f}\circ T_{x_0}\|'\approx \|\tilde{f}\|_{BMO(\Sn)}<\infty$
and this equivalence relation results in the equivalence between  the assertions (1) and (2) as well.
\smallskip

\noindent $(2)\Rightarrow (3)$. As in the first part of the proof of Lemma~\ref{lem75} we observe that, since $BMO(\Sn)\subset L^n(\Sn)$ the boundary maps $g:=\tilde{f}\circ T_{x_0}\in L^p(\Sn)$. As in that proof we may appeal to Theorem~\ref{thm11-ag2} to obtain that  $\|g\|_{\mathcal{H}^n}<C$. Then, assertion (3) follows from Lemma~\ref{lem75}.

\smallskip

\noindent $(3)\Rightarrow (4)$. We estimate directly and obtain that for any $x\in \B$
\begin{align*}
\a_f(x) & = \left(\frac{1}{|2B_x|} \int_{B_x}|Df(y)|^n \ud y\right)^{\frac1n}  \\
& = \left(\frac{1}{|2B_x|} \int_{B_x}|Df(y)|^{n} (1-|y|)^{n-1+a}\,(1-|y|)^{-(n-1+a)} \ud y\right)^{\frac{1}{n}} \\
& \approx_{n} \frac{1}{1-|x|}\,\frac{1}{\,(1-|x|)^{\frac{n-1+a}{n}}} \left(\int_{B_x}|Df(y)|^{n} (1-|y|)^{n-1+a}\ud y\right)^{\frac{1}{n}} \\
& \leq \frac{1}{\,(1-|x|)^{\frac{2n-1+a}{n}}} \left(\int_{\B\cap B_{\frac{x}{|x|}}(1-|x|)}|Df(y)|^{n} (1-|y|)^{n-1+a} \ud y\right)^{\frac{1}{n}} \\
& \lesssim \frac{1}{\,(1-|x|)^{\frac{2n-1+a}{n}}} (1-|x|)^{\frac{n-1}{n}}=(1-|x|)^{-\frac{n+a}{n}},
\end{align*}

where in the last step we use the Carleson measure condition in assertion (3). Moreover, we also appeal to an observation that on a hyperbolic ball $B_x$ it holds that $1-|y|\approx 1-|x|$.
\smallskip

\noindent $(4)\Rightarrow (1)$. When $N(f, \B)<\infty$, we may let $a=0$ in which case the assertion $(4)$ takes the following simple form 
\begin{equation}\label{assert4-a-zero}
\a_f(x) (1-|x|)\leq C\, \hbox{ for all } x\in\B.
\end{equation}
Let us consider the function $v:=\Nm(\a_f(x) (1-|x|))$ as in Proposition~\ref{prop:Lp-norm-bound} applied for $\beta=1$ and $\eta=1$ and note that, by the assertion $(4)$, it holds that $v\in L^p(\Sn)$ for all $0<p<\infty$. Hence, Proposition~\ref{prop:Lp-norm-bound} gives us that $\tilde{f}\in L^p(\Sn)$ and so assertion $(1)$ follows by considering maps $g(y):=\tilde{f}(T_{x_0}(y))-\tilde{f}(T_{x_0}(0))$ for a fixed $x_0\in \B$ and applying the characterization of the BMO($\Sn$) space given by the equivalence between assertions (1) and (2) of the theorem. In consequence, we obtain equivalences between conditions $(1)$-$(4)$ under the assumption that $N(f, \B)<\infty$.
\smallskip

\noindent $(4)\Rightarrow (6)$ Let us recall the following result by Staples specialized to our setting, see Corollary 2.26 in~\cite{st}:
\medskip

\emph{If $\sup_{2B\subset \B} \vint_{B}|f-f_{B}| \leq c_0$, then $f\in BMO(\B)$ and $\|f\|_{BMO(\B)}\leq c(n) c_0$.
}
\medskip

\noindent The above result roughly says, that in order to determine whether a map belongs to a BMO space, it is sufficient to check the definition only for balls enough away from the boundary $\Sn$, i.e. for such balls $B\subset \B$ that $2B\subset \B$. 

By the H\"older and the Poincar\'e inequalities and by the doubling property of the Lebesgue measure, we have that for any ball $B$, centered at point $x\in \B$, such that $B\subset 2B\subset \B$ it holds 
\begin{align*}
\vint_{B}|f(y)-f_{B}|\ud y &\leq  \left(\vint_{B}|f(y)-f_{B}|^n\ud y\right)^{\frac1n} \\
&\lesssim_{n}   \diam B \left(\vint_{B} |Df(y)|^n\ud y\right)^{\frac1n} \\
&\lesssim_{n, c_{\mathcal{L}^n}} \left(\int_{B} |Df(y)|^n\ud y\right)^{\frac1n} \lesssim_{n, c_{\mathcal{L}^n}} \left(\int_{B_x} |Df(y)|^n\ud y\right)^{\frac1n} \\
& \approx_{n, c_{\mathcal{L}^n}} (1-|x| )\left(\frac{1}{|2B_x|} \int_{B_x} |Df(y)|^n\ud y\right)^{\frac1n} \\
& \approx_{n, c_{\mathcal{L}^n}} (1-|x|) a_f(x) \leq C <\infty.
\end{align*}
Here, we also appeal to an observation that a ball $B=B(x,r)$ with $2B\subset \B$ satisfies that $B\subset B_x$.

\noindent $(6)\Rightarrow (5)$ immediately.
\smallskip

\noindent $(5)\Rightarrow (4)$ 
First, since $f^j\in BMO(\B)$ for all $j=1,\ldots, n$, we employ a variant of the BMO spaces description provided in 
Str\"omberg~\cite{str}, see also Ex. 4 in~\cite[pg.  261, Ch VI]{g} for the formulation in $\R$ and Appendix in~\cite{agg} for the $\Rn$ case. As consequence, we get
\[
 \int_{B_x} |f_j-(f_j)_{B_x}|^n \leq C(n, K) |B_x|.
\]
This, combined with the Caccioppoli inequality for component functions of a quasiregular map, see~\cite[Proposition 2.6]{no} and~\cite[Formula (5.3)]{in}, gives that $\int_{B_x} |Df|^n\leq C(n, K)$ for all $x\in \B$ and so, equivalently, that
$\a_f(x)(1-|x|)\leq C(n, K)$. This, however, is the condition $(4)$ for $a=0$, cf. the discussion at~\eqref{assert4-a-zero}.
\end{proof}

\begin{remark}
 Note that if the quasiregular map $f$ is as in Theorem~\ref {thm72} and satisfies that $N(f,\B)<\infty$, then $(4)\Rightarrow (5)$ by the first part of Theorem 1.4 in~\cite{no}.
\end{remark}

\section{Applications to elliptic PDEs}

The purpose of this section is to briefly address how some of the above results transfer to the properties of elliptic PDEs in the plane and beyond it. Our discussion serves as an illustration of relations between PDEs and quasiregular mappings.

\subsection{Case $n=2$} Recall that if $u\in W^{2,2}_{loc}(\Om)$ is a solution to the planar uniformly elliptic PDE with bounded measurable coefficients $a, b, c$, i.e. $Lu=au_{xx}+2bu_{xy}+cu_{yy}=0$, then its complex gradient $f:=u_x-{\rm i}u_y=(u_x,-u_y)$ is a $K$-quasiregular mapping with $K=K(\lambda,\Lambda)$, where $0<\lambda\leq \Lambda$ are the bounds in the uniform ellipticity condition for $L$. Moreover, the opposite relation holds as well, namely a $K$-quasiregular map defines a uniformly elliptic operator $L$, see~\cite[Chapter 12]{gt} and~\cite[Theorem 3]{man}. Let us also add that the Stoilow factorization and the theory of generalized analytic functions, see~\cite{vek} and~\cite[Chapter 6]{bjs}, allow to observe that the complex gradient $f$ is quasiregular also for $W^{1, p}_{loc}$-solutions of elliptic equations in the divergence form $\mathcal{L}u={\rm div}(A(x,y) \nabla u)=0$, see e.g.~\cite[Section 2.3]{mag} and even for degenerate quasilinear equations, such as the $p$-harmonic one, see~\cite{man, arli} and~\cite[Chapter 16]{aim}.  

In order to apply our results to a quasiregular map $f$ given by the complex gradient of the elliptic equation, we need to ensure the three conditions to hold:
\smallskip

\noindent (1) $f\not=0$,\\
\noindent (2) the multiplicity condition~\eqref{cond-m}, and \\
\noindent (3) the Miniowitz condition~\eqref{est-min}.
\smallskip

The first condition reads that the gradient $\nabla u\not= 0$ in $\BB$.

Upon applying the Stoilow factorization, we obtain that $f=h\circ g$ for holomorphic function $h$ and quasiconformal change of variables $g$ in $\BB$, and so by Definition~\ref{def-multi} the multiplicity condition reads for all $0<r<1$ and all $x\in \BB$
\begin{equation*}
 N(f, B_\rho(x, r))=N(h \circ g, B_\rho(x, r))=N(h, g(B_\rho(x, r)))\leq \frac{C}{(1-r)^a},\quad 0\leq a<1.
\end{equation*}
Therefore, the multiplicity conditions transfers naturally to the planar holomorphic function $h$ and images of hyperbolic balls under quasiconformal map $g$. Similar growth results appeared in the literature for the value distribution theory, see e.g.~\cite{Lo}, \cite[Chapter V]{Ric}, \cite[Chapter 11]{vuo2} and~\cite{Ak}.

It amounts to checking the Miniowitz condition, which as presented below follows from the gradient estimates for elliptic PDEs.
Suppose that the coefficients $a,b,c$ are H\"older regular in $\BB$ for the non-divergence equation $Lu$=0 (respectively, the matrix $A$ has $C^1$-entries for the above equation in the divergence form $\mathcal{L}u=0$). By the Schauder theory and Theorem 12.4 (or Corollary 6.3) in~\cite{gt}, applied to balls $B(0,r)\subset \BB$ we have the upper Miniowitz estimate for $Lu=0$ and $z=(x,y)$:
\begin{equation}\label{est-up-grad}
\frac{|\nabla u(z)|}{|\nabla u(0)|}\leq \frac{C(\lambda, \Lambda)\|u\|_{L^\infty(\BB)}}{(1-|z|)^{s}},\quad s=s(\lambda, \Lambda).
\end{equation}
The analogous estimate holds for solutions of the divergence type equation $\mathcal{L}u=0$, see~\cite[Theorem 13.1]{gt}.

In order to show the lower part of condition \eqref{est-min} we will appeal to the fact that $\nabla u\not =0$ in $\BB$ and so we will have that 
\begin{equation}\label{est-Aless}
 \frac{|\nabla u(z)|}{|\nabla u(0)|}\geq C \geq C (1-|z|)^{s},\quad z\in \BB,
\end{equation}
for the constant $C$ depending on $L$ and the geometry of the space. The estimates similar to~\eqref{est-Aless} have been studied in the literature and so, instead of repeating their technical details, we present only the sketch of the reasonings in the proofs of Theorem 2.1 in~\cite{al2}, see pg. 240-242 there and of Theorem 1.3 in~\cite{al1}, see pg. 273-275 there. By the direct formal computations we find that $|\nabla u|^2$ is a solution in $\BB$ of the following equation
\[
L|\nabla u|^2={\rm tr}(A \cdot [H(u)^2]):=F,
\]
where $A$ is the symmetric matrix with entries $a,b,c$ and $H(u)$ stands for the Hessian matrix of $u$. The justification that $L$ can be evaluated on $|\nabla u|^2$ follows the same lines as the proofs of Lemma 2.1 and 2.3 in~\cite{al2} and, therefore, we omit it. Moreover, by the uniform ellipticity condition we have that 
\[
0\leq {\rm tr}(A \cdot [H(u)^2])\leq \Lambda |H(u)|^2
\]
As in the proof of~\cite[Theorem 2.1]{al2}, see the paragraph following (2.32) on pg. 241, we appeal to the Morrey inequality, the Schauder estimates for the equation $Lv=F$ and get that for balls $B_r\subset B_R\Subset \BB$ it holds
\begin{align*}
 \inf_{B_r}|\nabla u| \gtrsim (R-r)^{1-\frac{2}{q}} \|H(u)\|_{L^p(B_R)} 
\end{align*}
for the appropriate choice of $R$ and $q>2$, e.g. $R-r\leq 1-R$ and $(R-r)^{1-\frac{2}{q}} \geq (1-R)^{4}$.
Similar argument gives us also the lower bound for divergence type equations, see the proof of Theorem 1.3 in~\cite{al1}, pg. 273-275.

As consequence of the above discussion we arrive at the following observations:
\begin{itemize}
\item[(1)]
Since the upper estimate in~\eqref{est-min} transfers to the growth condition on $|f|=|\nabla u|$, cf.~\eqref{cond-g}, the multiplicity condition~\eqref{cond-m} and~\eqref{est-up-grad} give the hypotheses of Theorem 4.1 in~\cite{kmv}, and so the vector field $f=(u_x, -u_y)$ has non-tangential limits at $\partial \BB$ at all points except for the set of the Hausdorff dimension ${\rm dim}_{\mathcal{H}}\leq \frac{2a}{1+a}$. Therefore, the non-tangential limit map $\tilde{f}$ exists a.e. and is well-defined on $\mathbb{S}^1$ 
\item[(2)] If $\tilde{f}\in L^p(\mathbb{S}^1)$, then the implication $(1)\Rightarrow (2)$ in Theorem~\ref{thm:characterizationsnolder} applied to $f$, gives us for $0<p<1$ that it holds 
$$
\int_{\B}\a_f^p(x)\, (1-|x|)^{p-1}\ud x<\infty,
$$
the estimate which does not follow from the Schauder estimates for $|H(u)|$. Indeed, by the Schauder estimates for solutions of $Lu=F$, see e.g. Theorem 6.2 and Corollary 6.3 in~\cite{gt}, we get that 
\begin{equation}\label{est-af-Sch}
 \a_f^2(x)=\frac{1}{|2B_x|} \int_{B_x}|H(u)|^2 \lesssim \frac{1}{|2B_x|} \int_{B_x} \frac{\|u\|^2_{L^{\infty}(\BB)}}{(1-|x|)^4}\approx \frac{\|u\|^2_{L^{\infty}(\BB)}}{(1-|x|)^4}.
\end{equation}
Thus, 
\[
 \int_{\B}\a_f^p(x)\, (1-|x|)^{p-1}\ud x \approx_{\|u\|_{L^{\infty}(\BB)}} \int_{\B} (1-|x|)^{-1-p}\ud x,
\]
which diverges.
\item[(3)] Theorem~\ref{thm72} implies that if $\tilde{f}\in BMO(\Sn)$, then $a_f(x)(1-|x|)^{1+\frac{a}{2}}\leq C$  for all $x \in \B$. Again, this is new growth estimate which does not follow from the Schauder estimates for $H(u)$. Indeed, appealing to such estimates, by~\eqref{est-af-Sch}, we only get that
\[
\a_f(x)(1-|x|)^{1+\frac{a}{n}}\lesssim (1-|x|)^{\frac{a}{2}-1},
\]
which grows unbounded when $|x|$ approaches $1$.
\end{itemize}

\subsection{Case $n>2$} Recall that, by (3.3)-(3.7) in~\cite[Chapter 3]{hkm}, an $\A$-harmonic operator is of type $n$,
if it satisfies the ellipticity-, the continuity-, the monotonicity conditions, and the homogeneity condition of degree $n-1$. The related are the $\A$-harmonic equations ${\rm div} \A(x,\nabla u)=0$ and one of the most profound examples of such equations is the $n$-harmonic one, i.e. ${\rm div} (|\nabla u|^{n-2}\nabla u)=0$, which in the plane reduces to the Laplace equation. For an $\A$-harmonic operator of type $n$, one shows that $f^{\#}\A$, the pull-back of $\A$ under a quasiregular map $f$, is also of type $n$, see Section 14.35 and Lemma 14.38 in~\cite{hkm}. It follows, that given a non-constant $K$-quasiregular mapping $f:\B \to \Rn \setminus\{0\}$, the function $u(x):=\ln|f(x)|$ is an $f^{\#}\A$-harmonic function in $\B$, where $\A$ is the $n$-harmonic operator with the ellipticity parameter $\frac 1K$, and the growth parameter $K$, see~\cite[Lemma 14.19]{hkm}. 

If $f$ satisfies~\eqref{cond-m} and ~\eqref{est-min}, then by Theorem 3.2 in~\cite{ag2} (and by Theorem 4.1 in~\cite{kmv}) the non-tangential limit map $\tilde{f}$ exists a.e. in $\Sn$ and the same holds for $\tilde{u}:=\ln |\tilde{f}|$. Thus, $u$ is an example of the $\A$-harmonic function for which the Fatou theorem holds. This is a non-canonical case for non-linear PDEs, see~\cite{manwe} and~\cite{fabes}.  

In order to illustrate Theorem~\ref{thm72} for $n>2$, let us observe that, since $|\nabla u|\leq \frac{|Df|}{|f|}$, it holds that
\[
\int_{B_x }|\nabla u|^n e^{n u} \leq \int_{B_x }|\nabla u|^n |f|^n\leq \int_{B_x }|Df|^n. 
\]
 
Therefore, the direct computations allow us to rephrase the implication $(1)\Rightarrow (4)$ in Theorem~\ref{thm72} as follows
\[
 \tilde{u}=\ln |\tilde{f}| \in \hbox{BMO}(\Sn) \Rightarrow \int_{B_x }|\nabla u|^n e^{n u} \lesssim \frac{C}{(1-|x|)^a},\quad x\in \B, 0\leq a<n-1.
\]
Moreover, if the multiplicity of $f$ is finite in $\B$, then the above gives a characterization of functions $\tilde{u}\in BMO(\Sn)$ in terms of the growth condition of the integral on the right-hand side.

\begin{appendix}

\section{}\label{appA}

\begin{proof}[Proof of Lemma~\ref{lem:WhitneyCovering}] The proof follows the lines of the corresponding result in Proposition 4.1.15 in~\cite{hkst}. However, since the lemma is a refinement of that result, we present the full discussion.

	For each $x\in\Omega$ we denote by $d(x):=d(x,X\setminus \Omega )$, and for $k\in\N$ define the following collection of hyperbolic balls in $\Om$:
$$
 \mathcal F_k:=\{ B(x,\frac{\eta}{5}d(x)):2^{k-1}\leq d(x)\leq 2^k\}.
$$
	By the $5B$-covering Lemma there exists a disjoint countable subfamily $\mathcal G_k\subset\mathcal F_k$ such that $\mathcal F_k\subset\bigcup_{B\in \mathcal G_k}5B$. Set
$$
 \mathcal W_k:=\{ 5B:B\subset \mathcal G_k\}.
$$
Fix $x_0\in \Omega$ and denote by $B_i=B(x_i,\eta d(x_i))$ the balls in $\mathcal W_k$. 
\smallskip

\emph{Claim:} It holds that 
\begin{equation}\label{eq:Whitney2}
	B(x_i,\eta\tau d(x_i))\subset B(x_1,3\eta\tau d(x_1)).
\end{equation}

In order to prove the claim, suppose that $x_0\in \tau B_i$ for some $i=1,\dots ,M$. Without the loss of generality we may additionally order the indices $i$ so that $d(x_1)\geq d(x_i)$ for all $i=2,\dots ,M$. Then by the triangle inequality
\begin{eqnarray*}
	&d(x_1,x_i)\leq d(x_0,x_1)+d(x_0,x_i)\leq \tau\mathrm{rad}(B_1)+\tau\mathrm{rad}(B_i)=\eta\tau (d(x_1)+d(x_i)),\\
	&d(x_1)\leq d(x_1,x_i)+d(x_i)\leq (\eta\tau +1)  d(x_i)+\eta\tau d(x_1).
	\end{eqnarray*}
Thus
\begin{equation}\label{eq:Whitney1}
	d(x_i)\geq \frac{1-\eta\tau}{1+\eta\tau}d(x_1).
\end{equation}
On the other hand, let $y\in \tau B_i$. Then, again by the triangle inequality, it holds that
\begin{align*}
	d(x_1,y)&\leq d(x_1,x_i)+d(x_i,y)\leq \eta\tau d(x_1)+2\eta\tau d(x_i)\leq 3\eta\tau d(x_1),
\end{align*}
and so the claim is proven.

Next, recall that $\mathcal G_k\subset \mathcal F_k$ so $x_i$ are, in particular, centers of balls in $\mathcal F_k$. Thus, by the definition of the family $\mathcal F_k$, we have that $d(x_i)\geq 2^{k-1}\geq \frac{1}{2}d(x_1)$ for all $i\in\N$. As the consequence, since all the balls $\frac{1}{5}B_i\in \mathcal G_k$ and so $B_i\cap B_j =\emptyset$ for all different $i,j\in\{ 1,\dots ,M\}$, we have
\begin{equation*}
	d(x_i,x_j) \geq \mathrm{rad} (\frac{1}{5}B_i)+\mathrm{rad} (\frac{1}{5}B_j) =\frac{\eta}{5}(d(x_i)+d(x_j)) \geq \frac{\eta}{5}d(x_1).
\end{equation*}
Furthermore, by \eqref{eq:Whitney2} the ball $3\tau B_1$ contains $M$ points $x_1,\dots ,x_M$ which are $\frac{\eta}{5}d(x_1)$-separated. Notice that
$2^{4+\log_2\tau}\frac{\eta}{5}\geq 3\eta\tau$, and thus by \cite[Lemma 4.1.12]{hkst} we obtain $M\leq N^{4+\log_2\tau}$, where $N$ is the Assuad dimension of $X$. 

Consider now the family $\mathcal W:=\bigcup_{k\in\N}\mathcal W_k$ and, as before, assume $x_0,x_1,\dots ,x_M$ are such that $x_0\in B(x_i,\tau\eta d(x_i))$ and $d(x_1)\geq d(x_i)$ for $i=1,\dots ,M$. Let $k_1\in\N$ be such that $B_1\in \mathcal \mathcal W_{k_1}$, and thus we also have that $x_1$ is a center of a ball in $\mathcal F_{k_1}$. By \eqref{eq:Whitney1} we also have $d(x_i)\geq\frac{1-\eta\tau}{1+\eta\tau}d(x_1)\geq 2^{k_1-s-1}$, where $s$ is the biggest integer so that $2^{-s}\leq \frac{1-\eta\tau}{1+\eta\tau}$. Then
$$
x_i\in\bigcup_{m=0}^{s+1}\mathcal F_{k_1-m}\;\mbox{ for all }i=1,\dots ,M.
$$
This yields that $M\leq (s+1)N^{4+\log_2\tau}$, proving the uniformly bounded overlap property.

\end{proof}

\section{}\label{appB}

First, we express the Morrey estimate in terms of averaged derivatives. Then, we prove Lemmas~\ref{lem:shadowmeasure} and~\ref{lem:leavingshadow}, see the paragraph before the statement of Lemma~\ref{lem:shadowmeasure} for the geometric-analytic meaning of the lemma.

\begin{prop}[Morrey's inequality]\label{lem:Morrey}
Let $f:\B\rightarrow\R^n$ be $K$-quasiregular. Let $0<\eta\leq 1$, then for almost every $x\in \B$ and $y\in \frac{\eta}{2}B_x$ we have
$$
|f(x)-f(y)|\leq C(n,K,\eta) (1-|x|)^\beta \a_{f,\eta}(x),
$$
for any $\beta\in [0,1]$.	
\end{prop}
\begin{proof}
 The proof follows immediately from~\cite[Theorem 5.2]{bi}. In particular, the Sobolev embedding argument and the Gehring estimate for quasiregular mappings with the exponent $p=p(n,K)>n$ give the following inequality for $x\in \B$ and $y\in\frac{\eta}{2}B_x$:
 \begin{align*}
  |f(x)-f(y)|&\leq C(n,K) \left(\frac{|x-y|}{\diam (\eta B_x)}\right)^{1-\frac{n}{p(n,K)}}(1-|x|)\left(\frac{1}{|2B_x|}\int_{\eta B_x} |Df|^n\right)^{\frac1n} \\
  &\leq C(n,K,\eta) (1-|x|) \a_{f,\eta}(x).
 \end{align*}
 Since $1-|x|<1$, the assertion follows from the trivial observation that $1-|x|\leq (1-|x|)^\beta$ for any $\beta\in[0,1]$.
\end{proof}

\begin{proof}[Proof of Lemma~\ref{lem:shadowmeasure}]
First let us prove the estimate at $x=0$, namely that
\begin{equation}\label{lem:assert-zero}
 \sigma\left(\left\{ \om\in S^{n-1}:|f(0)-\tilde{f}(\om)|\geq M\mathrm a_{f,\eta} (0)\right\}\right)\leq C(n,K,\eta) (\log M)^{1-n}.\end{equation}
 Let 
 $$
 E:=\left\{ \om\in S^{n-1}:|f(0)-\tilde{f}(\om)|\geq M\mathrm a_{f,\eta} (0)\right\}.
 $$
 For $x\in \frac{\eta}{2} B_0=B(0,\frac{\eta}{4})$ we have, by Proposition~\ref{lem:Morrey}, that $|f(0)-f(x)|\leq C(n,K,\eta) a_{f,\eta}(0)$, and thus $f(B(0,\frac{\eta}{4}))\subset B(f(0),C\mathrm a_{f,\eta}(0))$. Here, for simplicity of the presentation we denote by $C=C(n, K, \eta)$.
		
On the other hand, by the definition of $E$ we have that $f(E)\subset \R^n\setminus B(f(0),Ma_{f,\eta}(0))$. Suppose first that $M>C$ and let
\begin{align*}
&\Gamma_E:=\Gamma (B(0,\frac{\eta}{4}),E,\B ),\\
&\Gamma^\prime_E=f\Gamma_E .
\end{align*}
By definition $\Gamma^\prime_E$ consists of curves joining points in $f(B(0,\frac{\eta}{4}))$ and in $f(E)$. Thus, by the above discussion, $\Gamma^\prime_E$ majorizes the curve family consisting of curves in $\Rn$ with one endpoint in the ball $B(f(0),C\a_{f,\eta}(0))$ and the other one in $\Rn \setminus B(f(0),Ma_{f,\eta}(0))$, in the sense of Definition 6.3 in~\cite{va}. The latter family of curves contains the spherical ring for $M>C$. Therefore, we have the following modulus estimate
$$
 \mathrm{Mod}(\Gamma^\prime_E)\lesssim_n (\log \frac{M}{C})^{1-n}\lesssim_n (\log M)^{1-n}.
$$
The last inequality holds by the price of increasing admissible $M>C$. However, this is not a restriction for our investigations, as in what follows we will apply Lemma~\ref{lem:shadowmeasure} only for $M$ big enough, see the computations for $\|\tilde{f}\|_{L^p(\Sn)}$ at the end of the proof of Proposition~\ref{prop:Lp-norm-bound}.

On the other hand, by the definition of $\Gamma_E$ we also have the following modulus estimate, cf. Remark 7.7 in~\cite{va}:
$$
 \mathrm{Mod} (\Gamma_E)=\sigma (E)(\log\frac{4}{\eta})^{1-n}.
$$
Upon combining the above estimates for $\mathrm{Mod}(\Gamma^\prime_E)$ and $\mathrm{Mod} (\Gamma_E)$, together with the modulus definition of quasiregular map and the bound $N(f,\B )\leq N<\infty$, we arrive at the estimate
\begin{align*}
 \sigma (E)(\log\frac{4}{\eta})^{1-n} = \mathrm{Mod} (\Gamma_E)\leq N\mathrm{Mod}(\Gamma^\prime_E)\lesssim N (\log M)^{1-n}.
\end{align*}
Thus,
\[
\sigma (E)\lesssim N(\log\frac{4}{\eta})^{n-1} (\log M)^{1-n}\approx_{n, N, \eta}(\log M)^{1-n}.
\]
For second case, let us assume that $1<M\leq C$ (we can always increase constant $C$ in Proposition~\ref{lem:Morrey}, if $C<1$). Then, trivially $1 \leq \left(\frac{\log C}{\log M}\right)^{n-1}$ and so
\[
  \sigma(E)=\int_{E} 1\ud \sigma \leq \sigma(\Sn) (\log C)^{n-1}(\log M)^{1-n}\approx_{n,C}(\log M)^{1-n}.
\]
Hence, the assertion~\eqref{lem:assert-zero} is proven for $x=0$.
		
Now let $x\in\B$ and set $g=f\circ T_x^{-1}$, which will still have uniformly bounded multiplicity. Since $g(0)=f(x)$, applying the previous case to $g$ at the origin and $\tilde{\eta}\in (0,1)$, whose value will be determine later in the proof, we obtain 
$$
\sigma\left(\left\{ \om \in S^{n-1}:|f(x)-\tilde{g}(\om)|\geq M\mathrm a_{g,\tilde{\eta}} (0)\right\}\right)\leq C (\log M)^{1-n}.
$$
Moreover, by the definition of $g$ this also reads as
\begin{equation}\label{lem:assert-x}
\sigma\left(T_x\left(\left\{y\in S(x):|f(x)-\tilde{f}(y)|\geq M\mathrm a_{g,\tilde{\eta}} (0)\right\}\right)\right)\leq C (\log M)^{1-n},
\end{equation}
since we directly check that given $\om=T_x(y)$, a point in the set on the left-hand side above, for $y\in S(x)$, we have 
$$
M\mathrm a_{g,\tilde{\eta}} (0) \leq |f(x)-\tilde{f}(y)|=|f(x)-\tilde{f}(T_x^{-1}(\om))|=|f(x)-\tilde{g}(\om)|.
$$
On the other hand, the following relation holds
\[
\mathrm a_{g,\tilde{\eta}} (0) \approx_\eta (1-|x|)\mathrm a_{f,\eta}(x).
\] 
%
In order to see that formula we note that by the change of variables and the definition of a conformal map
\begin{align*}
|2B_0|\, \mathrm a_{g,\tilde{\eta}}^n (0) &= \int_{\tilde{\eta} B_0} |D(f\circ T_x^{-1})(y)|^n \\
&=\int_{\tilde{\eta} B_0} |Df(T_x^{-1}(y))|^n|\,|D T_x^{-1}(y)|^n J_{T_x^{-1}}(y)J_{T_x}(T_x^{-1}(y))\ud y \\
&=\int_{T_x^{-1}(\tilde{\eta} B_0)} |Df(z)|^n|\,|D T_x^{-1}(T_x(z))|^n J_{T_x}(z)\ud z \\
&=\int_{T_x^{-1}(\tilde{\eta} B_0)} |Df(z)|^n \ud z \qquad \text{\scriptsize{$(J_{T_x}(z)=[ J_{T_x^{-1}}(T_x(z)) ]^{-1}=|DT_x^{-1}(T_{x}(z)|^{-n})$.}}
\end{align*}
By Lemma 2.2 in~\cite{ag2} we have that for any $x\in \B$
\[
B(x, c (\eta) (1-|x|)) \subset T_x^{-1} (\tilde{\eta} B_0) \subset B(x, \frac{\eta}{2} (1-|x|)),
\]	
provided that $\tilde{\eta}=\frac{\eta/2-(\eta/2)^2}{(2+\eta/2)^2}$. Moreover, $0<\tilde{\eta}<1$ if and only if $0<\eta<2$. We apply this inclusion relation to the computations for $a_{g,\tilde{\eta}}^n (0)$ and obtain that
\[
\mathrm a_{g,\tilde{\eta}} (0) \leq \left(\frac{|2B_x|}{|2B_0|}\right)^{\frac1n} \left(\frac{1}{|2B_x|}\int_{\eta B_x} |Df(z)|^n \ud z\right)^{\frac1n}=(1-|x|) \mathrm a_{f,\eta}(x).
\]
Hence, by~\eqref{lem:assert-x} we get
\begin{align*}
&\sigma\left(T_x\left(\left\{y\in S(x):|f(x)-\tilde{f}(y)| \geq M(1-|x|)\mathrm a_{f,\eta} (x)\right\}\right)\right) \\
& \leq \sigma\left(T_x\left(\left\{y\in S(x):|f(x)-\tilde{f}(y)|\geq M\mathrm a_{g,\tilde{\eta}} (0)\right\}\right)\right)\leq C (\log M)^{1-n}.
\end{align*}
Finally, recall the following estimate (3) in~\cite{z} for $y,z\in S(x)$, the shadow associated with $x$:
\[
	\frac{1}{9(1-|x|)} |y-z|\leq |T_x(y)-T_x(z)|\leq \frac{2}{1-|x|} |y-z|.
\]
This, together with the observation that $\sigma(S(x)) \approx (1-|x|)^{n-1}$ concludes the proof.
	
\end{proof}

\begin{proof}[Proof of Lemma~\ref{lem:leavingshadow}]
	For a fixed parameter $c>0$ consider $x\in E$, so that $d(S(x),\Sn\setminus U)\simeq_c 1-|x|$, and let also $\tau >1$. Note that this assumption excludes the case $x=0$, as then $d(S(0),\Sn\setminus U)=d(\Sn,\Sn\setminus U)=0$ for any $U\subsetneq \Sn$. Therefore, we define
	\begin{equation}\label{def:y-tau}
	 y_\tau(x) :=\left( |x|-\frac{\tau}{2}(1-|x|)\right) \frac{x}{|x|},
	\end{equation}
	which satisfies 
	\begin{align*}
	&|y_\tau(x) -x|= \left| x- \frac{\tau}{2}\frac{x}{|x|}(1-|x|) -x \right| =\frac{\tau}{2}(1-|x|),\,\hbox{ hence } y_\tau(x)\in \partial(\tau B_x) \hbox{ and moreover},\\
	&\left|\frac{x}{|x|}-y_\tau(x)\right|=\left( 1+\frac{\tau}{2}\right) (1-|x|).
	\end{align*}
 Note that $y_{\tau}(x)$ may belong to $\Rn\setminus \B$.	
On the other hand, since $U$ and $\Sn\setminus S(x)$ ar closed, there exist $w_1\in S(x)$ and $w_2\in \Sn\setminus U$ such that $d(S(x),\Sn\setminus U)=d(w_1,w_2)\simeq_{n} |w_1-w_2|$. In particular, by the property of $x\in E$ we have $|w_1-w_2|\approx_{c, n} 1-|x|$ with $c$ independent of $x$. Recall also that, since $w_1\in S(x)$, it holds that $|w_1-x|\leq (1+\alpha)(1-|x|)$. Therefore,
\begin{align*}
	|y_\tau(x) -w_2|\leq |y_\tau(x) -x|+|x-w_1|+|w_1-w_2|&\leq \left( \frac{\tau}{2}+(1+\alpha)+C\right) (1-|x|)\\
	&= \frac{\frac{\tau}{2}+(1+\alpha)+C}{1+\frac{\tau}{2}}\left|\frac{x}{|x|}-y_\tau(x)\right|.
\end{align*}
Notice that the fraction on the right-hand side above is a decreasing function of $\tau$ and
$$
\lim_{\tau\to\infty }\frac{\frac{\tau}{2}+C+\alpha +1}{\frac{\tau}{2}+1}=1.
$$
Therefore, for sufficiently large $\tau:=\tau_0(\alpha)$ we have $|y_{\tau_0}(x)-w_2|\leq (1+\alpha )\left|\frac{x}{|x|}-y_{\tau_0}(x)\right|$. The following two cases may occur:
\smallskip

\noindent (1) $\left|\frac{x}{|x|}-y_{\tau_0}(x)\right|\leq 1$. Then it holds that $\left|\frac{x}{|x|}-y_{\tau_0}(x)\right|=1-|y_{\tau_0}(x)|$ and, hence, $|y_{\tau_0}(x)-w_2|\leq (1+\alpha)(1-|y_{\tau_0}(x)|)$. This, in particular, means that $w_2\in S(y_\tau(x))$. Furthermore, recall that $w_2\notin U$ as assumed above. We set $y_x:=y_{\tau_0}(x)$ and directly verify that $y_x$ satisfies~\eqref{prop1:lem-shadow} and~\eqref{prop2:lem-shadow}. 
\smallskip

\noindent (2) Otherwise, let $\left|\frac{x}{|x|}-y_{\tau_0}(x)\right|> 1$. Then, upon setting $y_x:=0$, we clearly get~\eqref{prop1:lem-shadow}, since $S(0)=\Sn$. As for~\eqref{prop2:lem-shadow} observe that if $\left|\frac{x}{|x|}-y_{\tau_0}(x)\right|>1$, then by the construction of $y_\tau(x)$ in~\eqref{def:y-tau}, the ball $\overline{\tau_0 B_x}$ contains the origin and this holds regardless whether $y_{\tau_0}(x)\in \B$ or not. In the latter case, i.e. if $y_{\tau_0}(x)\notin \B$, the ball $\overline{\tau_0B_x}$ contains $\B$, again by~\eqref{def:y-tau}. This discussion justifies the choice $y_x:=0$.
\end{proof}
\end{appendix}

\end{document}